\documentclass[11pt, reqno]{amsart}

\usepackage[top=3.75cm, bottom=3cm, left=3cm, right=3cm]{geometry}
\usepackage{amsmath}
\usepackage{amsthm}
\usepackage{amsfonts}
\usepackage{amssymb}
\usepackage{mathtools}
\usepackage[mathscr]{euscript}
\usepackage{bm}
\usepackage[colorlinks=true, citecolor=darkblue, urlcolor=darkblue, linkcolor=darkblue, linktocpage = true]{hyperref}
\usepackage{color}
\usepackage[all,cmtip]{xy}
\usepackage{enumitem}
\usepackage{tikz}  
\usepackage{graphicx}
\usepackage{scalerel}
\usepackage{comment}
\usepackage{stmaryrd}
\usepackage{tikz-cd}

\usepackage[normalem]{ulem}

\definecolor{darkblue}{rgb}{0,0,.75}

\numberwithin{equation}{section}

\theoremstyle{plain}
\newtheorem{theorem}{Theorem}[section]
\newtheorem{lemma}[theorem]{Lemma}
\newtheorem{proposition}[theorem]{Proposition}
\newtheorem{corollary}[theorem]{Corollary}

\newtheorem{question}[theorem]{Question}

\theoremstyle{definition}
\newtheorem{definition}[theorem]{Definition}
\newtheorem{example}[theorem]{Example}
\newtheorem{remark}[theorem]{Remark}
\newtheorem{notation}[theorem]{Notation}

\newtheorem{setup}[theorem]{Setup}

\makeatletter
\newtheoremstyle{italicsname}% <name>
 {3pt}% <Space above>
 {3pt}% <Space below>
 {\itshape}% <Body font>
 {}% <Indent amount>
{\bf}% <Theorem head font>
 {.}% <Punctuation after theorem head>
 {.5em}% <Space after theorem heading>
 {\thmname{#1}\thmnumber{\@ifnotempty{#1}{ }#2}%
 \thmnote{ {\the\thm@notefont(#3)}}}% <Theorem head spec (can be left empty, meaning `normal')>
\makeatother
\theoremstyle{italicsname}

\setlist[itemize]{topsep=5pt,itemsep=3pt}
\setlist[enumerate]{topsep=5pt,itemsep=3pt}
\newcommand{\sth}{~\vline~}
\newcommand{\set}[1]{\left\{ \, #1 \, \right\}}

\newcommand{\cin}[1]{{\tau}_{#1}^{-}}
\newcommand{\ccin}[1]{{\tau}_{#1}^{+}}

\newcommand{\ol}{\overline}

\newcommand{\Db}{\mathrm{D^b}}
\newcommand{\Dm}{\mathrm{D^-}}
\newcommand{\Dperf}{\mathrm{D}_{\mathrm{perf}}}
\newcommand{\Dqc}{\mathrm{D}_{\mathrm{qc}}}

\newcommand{\BN}{\mathscr{BN}}

\newcommand{\emphsf}[1]{{\sf #1}}

\DeclareMathOperator{\colim}{colim}

\newcommand{\cB}{\mathscr{B}}

\newcommand{\op}{\mathrm{op}}

\newcommand{\kk}{\Bbbk}

\DeclareMathOperator{\cRHom}{\mathrm{R}\mathcal{H}\!{\it om}}
\DeclareMathOperator{\RHom}{\mathrm{RHom}}

\DeclareMathOperator{\Hom}{Hom}

\DeclareMathOperator{\Ext}{Ext}

\DeclareMathOperator{\HH}{HH}

\DeclareMathOperator{\Vect}{Vect}
\newcommand{\fd}{\mathrm{fd}}

\newcommand{\id}{\mathrm{id}}
\newcommand{\pr}{\mathrm{pr}}

\newcommand{\Coh}{\mathrm{Coh}}

\DeclareMathOperator\coker{coker}
\DeclareMathOperator\im{im}

\newcommand{\bF}{\mathbf{F}}

\newcommand{\ev}{\mathrm{ev}}

\DeclareMathOperator{\perv}{perv}

\DeclareMathOperator\Rep{Rep}

\DeclareMathOperator{\cone}{cone}

\newcommand{\cO}{\mathcal{O}}
\newcommand{\cA}{\mathscr{A}}
\newcommand{\cC}{\mathscr{C}}
\newcommand{\cD}{\mathscr{D}}

\newcommand{\cF}{\mathscr{F}}

\newcommand{\cH}{\mathcal{H}}

\newcommand{\cP}{\mathcal{P}}

\newcommand{\cR}{\mathcal{R}}

\newcommand{\cT}{\mathscr{T}}

\newcommand{\rD}{\mathrm{D}}
\newcommand{\rE}{\mathrm{E}}

\newcommand{\rH}{\mathrm{H}}

\newcommand{\rK}{\mathrm{K}}
\newcommand{\rS}{\mathrm{S}}

\newcommand{\fa}{\mathfrak{a}}

\newcommand{\bC}{\mathbf{C}}

\newcommand{\bZ}{\mathbf{Z}}
\newcommand{\bP}{\mathbf{P}}
\newcommand{\bQ}{\mathbf{Q}}
\newcommand{\bR}{\mathbf{R}}

\newcommand{\hcD}{\widehat{\cD}}
\newcommand{\hcC}{\widehat{\cC}}
\newcommand{\htau}{\widehat\tau}

\newcommand{\hgamma}{\widehat{\gamma}}

\DeclareMathOperator{\ext}{ext}

\usepackage{xpatch}
\makeatletter   
\xpatchcmd{\@tocline}
{\hfil\hbox to\@pnumwidth{\@tocpagenum{#7}}\par}
{\ifnum#1<0\hfill\else\dotfill\fi\hbox to\@pnumwidth{\@tocpagenum{#7}}\par}
{}{}
\makeatother 

\makeatletter
\renewcommand\part{%
   \if@noskipsec \leavevmode \fi
   \par
   \addvspace{4ex}%
   \@afterindentfalse
   \secdef\@part\@spart}

\def\@part[#1]#2{%
    \ifnum \c@secnumdepth >\m@ne
      \refstepcounter{part}%
      \addcontentsline{toc}{part}{Part \thepart.\hspace{1em}#1}%
    \else
      \addcontentsline{toc}{part}{#1}%
    \fi
    {\parindent \z@ \raggedright
     \interlinepenalty \@M
     \normalfont
     \ifnum \c@secnumdepth >\m@ne
     \centering 
     \Large\bfseries \partname\nobreakspace\thepart     
       \nobreak. 
     \fi
     \Large \bfseries { #2}%
     \par}%
    \nobreak
    \vskip 3ex
    \@afterheading}
\def\@spart#1{%
    {\parindent \z@ \raggedright
     \interlinepenalty \@M
     \normalfont
     \huge \bfseries #1\par}%
     \nobreak
     \vskip 3ex
     \@afterheading}
\makeatother

\renewcommand{\thepart}{\Roman{part}}

\makeatletter 
\def\l@subsection{\@tocline{2}{0pt}{3pc}{6pc}{}} 
\makeatother

\begin{document}

\title{Inducing t-structures on semiorthogonal components} 

\author{Alexander Kuznetsov} 
\address{{\sloppy
\parbox{0.9\textwidth}{
Algebraic Geometry Section, Steklov Mathematical Institute of Russian Academy of Sciences,
8 Gubkin str., Moscow 119991 Russia
\\[5pt]
Laboratory of Algebraic Geometry, National Research University Higher School of Economics, Russian Federation
}\medskip}}
\email{akuznet@mi-ras.ru}

\author{Shengxuan Liu}
\address{Department of Mathematics, University of Michigan, Ann Arbor, MI 48109 \smallskip}
\email{liusx@umich.edu}

\author{Alexander Perry}
\address{Department of Mathematics, University of Michigan, Ann Arbor, MI 48109 \smallskip}
\email{arper@umich.edu}

\begin{abstract}
Given a triangulated category with a t-structure, we introduce a method for inducing t-structures on its semiorthogonal components, 
based on the construction of an associated perverse t-structure on the ambient category. 
As applications, we construct bounded t-structures in many new examples, including:   
almost all known phantom and quasiphantom categories;  
the semiorthogonal complement of the structure sheaf on a Fano variety; 
the residual component of an Enriques surface;
the categorical resolution of a nodal cubic curve appearing in an early counterexample to the Jordan--H\"{o}lder property for semiorthogonal decompositions; 
and Brill--Noether modifications of the derived category of a curve. 
\end{abstract} 

\maketitle

\setcounter{tocdepth}{1}
\tableofcontents

%%%%%%%%%%%%%%%%%%%%%%%%%%%%%%%%%%%%%%%%%%%%%%%%%%%%%%%

\section{Introduction}

The bounded derived category of coherent sheaves~$\Db(X)$ on a smooth proper variety~$X$
plays a central role in modern algebraic geometry. 
One of the most important methods for analyzing this category 
is to break it into smaller pieces, 
via the notion of a semiorthogonal decomposition. 
A triangulated category~$\cC$ appearing as a component in such a decomposition can fruitfully be regarded as a noncommutative algebraic variety.  
Indeed, even though $\cC$ does not possess an underlying space of points, 
it exhibits many of the structures of (the derived category of) an actual variety. For instance: 
\begin{itemize}
\item $\cC$ admits a Serre functor, which plays the role of the canonical bundle in Serre duality; 
\item $\cC$ carries intrinsic motivic invariants, such as K-theory and Hochschild homology, generalizing the corresponding invariants of varieties; and 
\item $\cC$ has an associated moduli stack which parameterizes its objects. 
\end{itemize} 

A fundamental feature of the category $\Db(X)$ is that it admits a bounded t-structure,
which formalizes the fact that each of its objects can be built up via extensions and shifts
from finitely many objects in the abelian category $\Coh(X)$ of coherent sheaves.
This paper is motivated by the question of whether semiorthogonal components also exhibit this structure. 

\begin{question}
\label{question-soc-t-structure}
    If~$\cC \subset \Db(X)$ is a semiorthogonal component of the derived category of a smooth proper variety~$X$,
    then does~$\cC$ admit a bounded t-structure?
\end{question}

Question~\ref{question-soc-t-structure} is wide open. 
To our knowledge, no negative results are known, while positive results were previously only known
for some very special examples arising as residual components of low-dimensional Fano varieties~\cite{BLMS, GM-stability, peize}. 

The purpose of this paper is to provide a new approach to the problem.
Given a triangulated category~$\cD$, such as~$\Db(X)$, 
our main result provides a criterion   
for inducing a t-structure on a semiorthogonal component of~$\cD$
from a t-structure on~$\cD$. 
This leads to the construction of bounded t-structures on many interesting semiorthogonal components,
including almost all known phantom and quasiphantom categories.

An important motivation for Question~\ref{question-soc-t-structure}
is Bridgeland's theory of stability conditions on triangulated categories~\cite{bridgeland-stability}.
The handful of residual components of Fano varieties where stability conditions are known to exist have led to many applications,
so it is natural to search for a construction that applies to semiorthogonal components more widely;
since any stability condition comes with an underlying bounded t-structure,
a positive answer to Question~\ref{question-soc-t-structure} would provide a first step in this direction.
In a sequel to this work, we will build on our methods to provide criteria for
inducing stability conditions on semiorthogonal components.

\subsection{Induced t-structures} 
\label{section-intro-main-result} 
Recall that a t-structure~$\tau = (\cD^{\leq 0},\cD^{\geq 0})$ on a triangulated category~$\cD$
consists of a pair of subcategories~$\cD^{\leq 0}$ and $\cD^{\ge 0}$,
called the \emphsf{connective} and \emphsf{coconnective parts} of~$\tau$, satisfying three axioms,
which can informally be described as \emphsf{compatibility with shifts}, \emphsf{$\Hom$-orthogonality},
and the \emphsf{generation property} (see Definition~\ref{def:t-str}). 

Let~$\cC \subset \cD$ be a triangulated subcategory, with embedding functor~$\gamma \colon \cC \to \cD$.
There are at least two naive ways to attempt to construct a t-structure on~$\cC$ from the t-structure~$\tau$ on~$\cD$.
For the first, consider the following pair of subcategories of~$\cC$:
\begin{equation}
\label{eq:t-restricted}
\gamma^{-1}(\tau) \coloneqq (\gamma^{-1}(\cD^{\le 0}), \gamma^{-1}(\cD^{\ge 0})), 
\end{equation}
where~$\gamma^{-1}(\cD^{\leq 0}) \subset \cC$ denotes the subcategory of objects~$C \in \cC$ with~$\gamma(C) \in \cD^{\leq 0}$,
and similarly for~$\gamma^{-1}(\cD^{\ge 0})$.
This pair always satisfies compatibility with shifts and $\Hom$-orthogonality, but the generation property may fail. 
If, however, the generation property holds, then we say that~\emphsf{$\tau$ restricts to~$\cC$}
and call~$\gamma^{-1}(\tau)$ (or~$\tau\vert_\cC$ when we wish to suppress the functor~$\gamma$) the \emphsf{restricted t-structure}.

For the second construction, assume that~$\cC \subset \cD$ is right admissible,
i.e., that~$\gamma$ admits a right adjoint~$\gamma^! \colon \cD \to \cC$;
there is a parallel story for left admissible subcategories.
Then we may also consider the following pair of subcategories of $\cC$: 
\begin{equation}
\label{eq:t-projected}
\gamma^!(\tau) \coloneqq (\gamma^!(\cD^{\le 0}), \gamma^!(\cD^{\ge 0})), 
\end{equation}
where $\gamma^!(\cD^{\leq 0}) \subset \cC$ denotes the full subcategory of objects of the form $\gamma^!(D)$ for $D \in \cD^{\leq 0}$, and similarly for $\gamma^!(\cD^{\ge 0})$. 
This pair always satisfies compatibility with shifts and the generation property,
while $\Hom$-orthogonality may fail. 
If, however, $\Hom$-orthogonality holds, we say that~\emphsf{$\tau$ projects to~$\cC$ along~$\gamma^!$}
or \emphsf{right projects to~$\cC$}, 
and call~$\gamma^!(\tau)$ the \emphsf{projected t-structure}.

Note that the isomorphism of functors~$\gamma^! \circ \gamma \cong \id_\cC$
implies an inclusion~$\gamma^{-1}(\cD^{\le 0}) \subset \gamma^!(\cD^{\le 0})$,
i.e. the connective part of~\eqref{eq:t-restricted} is always contained in the connective part of~\eqref{eq:t-projected}.
On the other hand, if~\eqref{eq:t-projected} is a t-structure,
it is easy to see that~$\gamma^{-1}(\cD^{\le 0}) = \gamma^!(\cD^{\le 0})$ (see Lemma~\ref{lemma-projection-tau-cB}\ref{beta*t-geq0}),
i.e. the connective parts of~\eqref{eq:t-restricted} and~\eqref{eq:t-projected} coincide.
Thus, \eqref{eq:t-restricted} and~\eqref{eq:t-projected} can be considered as two explicit attempts
to find a t-structure on~$\cC$ with connective part
\begin{equation}
\label{eq:ccm-intro}
\cC^{\le 0} = \gamma^{-1}(\cD^{\le 0}).
\end{equation}
From this point of view, it is natural to ask if there is any t-structure on~$\cC$
with connective part~\eqref{eq:ccm-intro}.
If one exists, then it is unique, with coconnective part given
by the shifted right $\Hom$-orthogonal complement of~\eqref{eq:ccm-intro}, i.e.
\begin{equation}
\label{eq:t-induced}
{\cC^{\ge 0} = {}}
\set{ C \in \cC \sth \Hom(C'{[1]},C) = 0 \text{ \textup{for all} } C' \in {\cC^{\le 0}} }.
\end{equation} 
When {the pair~$(\cC^{\le 0},\cC^{\ge 0})$ defined by~\eqref{eq:ccm-intro} and~\eqref{eq:t-induced}}
determines a t-structure, we say that~\emphsf{$\tau$ connectively induces a t-structure on~$\cC$};
the resulting t-structure on $\cC$ is denoted by $\cin{\cC}$ and called the \emphsf{connectively induced t-structure}. 

We similarly define what it means for~$\tau$ to \emphsf{coconnectively induce a t-structure on~$\cC$}.
When this holds, the \emphsf{coconnectively induced t-structure~$\ccin{\cC}$ on~$\cC$}
has coconnective part~${\cC^{\ge 0} = {}} \gamma^{-1}(\cD^{\ge 0})$
and connective part~{$\cC^{\le 0}$} given by its shifted left $\Hom$-orthogonal complement.

\begin{example}
\label{example-rational-singularities}
Let~$f \colon X \to Y$ be a proper surjective morphism of smooth varieties
such that~$f_* \cO_X \simeq \cO_Y$, where $f_*$ denotes the derived pushforward.
In this case, the derived pullback functor~$f^* \colon \Db(Y) \to \Db(X)$
between the bounded derived categories of coherent sheaves is fully faithful,
and there is a semiorthogonal decomposition
\begin{equation}
\label{intro-DmX-DmY}
\Db(X) = \langle \cB, f^*\Db(Y) \rangle ,
\end{equation}
where explicitly $\cB = \ker(f_* \colon \Db(X) \to \Db(Y))$.
In these terms, the standard t-structures on~$\Db(X)$ and $\Db(Y)$ are related by the formula
\begin{equation*}
\Db(Y)^{\leq 0} = \set{ F \in \Db(Y) \sth f^*F \in \Db(X)^{\leq 0} }.
\end{equation*}
Thus, the standard t-structure on~$\Db(Y)$ is connectively induced by the standard t-structure on~$\Db(X)$.
On the other hand, this t-structure is generally neither restricted (unless~$f^*$ is left t-exact, i.e. $f$ is flat),
nor projected (unless~$f_*$ is right t-exact).
\end{example}

In our main theorem below, we require that~$\tau = (\cD^{\leq 0},\cD^{\geq 0})$ is noetherian or artinian,
meaning that its heart --- the abelian category given by~$\cD^{\heartsuit} = \cD^{\leq 0} \cap \cD^{\geq 0}$ ---
is noetherian or artinian. 
Noetherianity is a finiteness condition which is satisfied by many examples,
like the standard t-structure on~$\Db(X)$ in Example~\ref{example-rational-singularities}
(whose heart is the category~$\Coh(X)$ of coherent sheaves); 
dually, artinianity is satisfied by the t-structure on the opposite category in such examples
(see Remark~\ref{remark-opposite-t-structure} and Example~\ref{example-dual-standard-t-structure}).
For some special examples, like the standard t-structure on the bounded derived category of finite-dimensional representations of a quiver,
or more general categories of finite length,
both conditions are simultaneously satisfied.

\begin{theorem}
\label{main-theorem}
Let $\cD = \langle \cB, \cC \rangle$ be a semiorthogonal decomposition of a triangulated category. 
Let~$\tau = (\cD^{\leq 0}, \cD^{\geq 0})$ be a t-structure on~$\cD$.
\begin{enumerate}[label={\textup{(\alph*)}}]
\item
\label{main-theorem-induce-C}
If~$\tau$ is noetherian and restricts to a t-structure on~$\cB$,
then~$\tau$ {connectively} induces a t-structure~{$\cin{\cC}$} on~$\cC$, which is bounded if~$\tau$ is bounded.
\item
\label{main-theorem-induce-B}
If~$\tau$ is artinian and restricts to a t-structure on~$\cC$,
then~$\tau$ {coconnectively} induces a t-structure~{$\ccin{\cB}$} on~$\cB$,
which is bounded if $\tau$ is bounded.
\end{enumerate} 
\end{theorem}

To prove Theorem~\ref{main-theorem} we modify the t-structure~$\tau$ on~$\cD$
in such a way that it projects to~$\cC$ or~$\cB$ and gives the required induced t-structure. 
For instance, in the situation of Theorem~\ref{main-theorem}\ref{main-theorem-induce-C}
we observe that a t-structure on~$\cD$ right projects to $\cC$ if and only if
the right projection functor~$\gamma \circ \gamma^! \colon \cD \to \cD$ onto~$\cC$
is right t-exact (Theorem~\ref{theorem-projected-t-structures}).
This need not hold for~$\tau$, but we show that the assumption that~$\tau$ restricts to~$\cB$ implies that it almost does:
$\gamma \circ \gamma^!$ has right t-amplitude~$\leq 1$,
i.e., $\gamma\gamma^!(\cD^{\leq 0}) \subset \cD^{\leq 1}$ (Lemma~\ref{lemma-restrict-t-structure-vs-projection-amplitude}).
Then, using the noetherianity of~$\tau$, we construct by tilting
a new t-structure~$\cin{\gamma \gamma^!} = \Big({}^{\cin{\gamma\gamma^!}}\cD^{\leq 0}, {}^{\cin{\gamma\gamma^!}}\cD^{\geq 0} \Big)$ on~$\cD$
characterized by
\begin{equation}
\label{taugammagamma!}
{}^{\cin{\gamma\gamma^!}}\cD^{\leq 0} = \set{ D \in \cD^{\leq 0} \sth \gamma\gamma^!(D) \in \cD^{\leq 0} } .
\end{equation}
By construction~$\gamma \circ \gamma^!$ is right t-exact with respect to~$\cin{\gamma \gamma^!}$,
so~$\cin{\gamma\gamma^!}$ right projects to a t-structure on~$\cC$,
which we verify satisfies the characterizing property of the induced t-structure on~$\cC$.

\begin{remark}
\label{remark-computing-tauC}
A priori, even if the induced t-structure~{$\cin{\cC}$} exists, it may not be ``computable'';
for instance, it may be difficult to describe explicit objects in its heart,
because the coconnective part~$\cC^{\geq 0}$ is given by the indirect formula~\eqref{eq:t-induced}. 
Our description of $\cin{\cC}$ as the projection of~$\cin{\gamma \gamma^!}$ helps remedy this,
since it allows us to make computations with the given t-structure~$\tau$
and then project to~$\cC$ (see Lemma~\ref{lemma-object-in-heart-tauB}).
As an illustration, in Example~\ref{example-mattoo-in-heart} we describe interesting objects
in the heart of an induced t-structure on a phantom category.
\end{remark} 

\begin{example}[Perverse t-structures]
\label{example-perverse-t-structures} 
In the situation of Example~\ref{example-rational-singularities},
suppose that the fibers of~$f \colon X \to Y$ have dimension~$\leq 1$.
Then, as observed by Bridgeland~\cite{bridgeland-flops},
the standard t-structure on $\Db(X)$ restricts to the component $\cB$ in~\eqref{intro-DmX-DmY} 
(see Example~\ref{example-bridgeland}, which also treats the case where $X$ and $Y$ are not necessarily smooth).  
Thus, the assumptions of Theorem~\ref{main-theorem}\ref{main-theorem-induce-C} are satisfied.
By the discussion in Example~\ref{example-rational-singularities},
the resulting connectively induced t-structure on~$\Db(Y)$ is nothing but the standard t-structure.
However, the intermediate t-structure~$\cin{\gamma\gamma^!}$ on~$\cD = \Db(X)$ appearing in our construction turns out to be very interesting:
it coincides with one of Bridgeland's perverse t-structures.
In~\S\ref{section-perverse-t-structures} we generalize Bridgeland's theory of perverse coherent sheaves
to the abstract setting of Theorem~\ref{main-theorem},
with the induced t-structure~$\cin{\cC}$ playing the role of the standard t-structure on~$\Db(Y)$.
\end{example} 

\begin{remark}
In Appendix~\ref{sec:ind}, we present a second independent proof of Theorem~\ref{main-theorem}\ref{main-theorem-induce-C}
under a mild additional assumption (automatically satisfied for instance if~$\cD$ is enhanced),
based on t-structures generated by a collection of compact objects in a cocomplete category.
\end{remark}

Our proof of Theorem~\ref{main-theorem}\ref{main-theorem-induce-C}
fits into a broader method for inducing t-structures, developed in~\S\ref{subsection-t-exactability},
which gives a potential approach to Question~\ref{question-soc-t-structure} in general.   
Given any semiorthogonal decomposition~$\cD = \langle \cB, \cC \rangle$ and a t-structure~$\tau$ on $\cD$,
we say that the projection functor~$\gamma \circ \gamma^!$ is \emphsf{right t-exactable with respect to~$\tau$}
if there exists a (uniquely determined) t-structure~$\cin{\gamma \gamma^!}$
with connective part given by the formula~\eqref{taugammagamma!};
in this case, motivated by Example~\ref{example-perverse-t-structures},
we call~$\cin{\gamma \gamma^!}$ the~\emphsf{$\gamma\gamma^!$-perverse t-structure}.
Whenever this holds, it follows that~$\tau$ connectively induces a t-structure on~$\cC$,
given by the right projection of~$\cin{\gamma \gamma^!}$ (Theorem~\ref{theorem-t-exactable-induced-tauC}).

Moreover, we propose a general method for proving right t-exactability of the functor~\mbox{$\gamma \circ \gamma^!$},
provided that it has bounded right t-amplitude with respect to~$\tau$,
which for instance automatically holds when~$\cD = \Db(X)$
for a smooth proper variety~$X$ (Remark~\ref{remark-boundedness-functors-spc}).
Namely, we introduce a candidate torsion pair in~$\cD^{\heartsuit}$
such that~$\gamma \circ \gamma^!$ has right t-amplitude one less
with respect to the corresponding tilted t-structure~{$\tau^-$} (Proposition~\ref{proposition-tauPhib}).
An important caveat is that we only know how to prove our candidate torsion pair
is an actual torsion pair when~$\tau$ is noetherian;
moreover, even if~$\tau$ is noetherian, the tilt~{$\tau^-$} need not be {(Remark~\ref{remark-noetherian-phantom}), 
so we cannot blindly iterate this construction.
However, whenever one can show that our tilting construction can be iterated
to reduce the right t-amplitude of~$\gamma \circ \gamma^!$ to zero,
the existence of the perverse t-structure~$\cin{\gamma \gamma^!}$ follows (Remark~\ref{remark-t-exactable-tilting}).
Theorem~\ref{main-theorem}\ref{main-theorem-induce-C}
is simply the case where only one tilt is required, so no iteration is needed.

In order to effectively apply Theorem~\ref{main-theorem},
we develop some criteria for restricting t-structures to a triangulated subcategory $\cB \subset \cD$
in the presence of a semiorthogonal decomposition of~$\cB$ or a nice collection of generators.
This leads to the following convenient version of Theorem~\ref{main-theorem}\ref{main-theorem-induce-C}
{(see also Theorem~\ref{theorem-induce-via-simple-objects})}.

\begin{theorem}
\label{theorem-induce-via-exceptional-sequence}
Let~$\cD$ be a $\kk$-linear triangulated category for a field~$\kk$.
Let~$\tau = (\cD^{\leq 0}, \cD^{\geq 0})$ be a noetherian t-structure on~$\cD$.
Let~$E_1, \dots, E_n \in \cD^{\heartsuit}$ be an exceptional sequence of objects in the heart of~$\tau$
which satisfies the forward $\Hom$-vanishing condition:
\begin{equation}
\label{eq:hom-vanishing-intro}
\Hom(E_i, E_j) = 0
\quad
\text{for all $i < j$}.
\end{equation}
Let
\begin{equation*}
\cD = \langle E_1, E_2, \dots, E_n, \cC \rangle
\end{equation*}
be the corresponding semiorthogonal decomposition of~$\cD$.
Then~$\tau$ connectively induces a t-structure~$\cin{\cC}$ on~$\cC$, which is bounded if~$\tau$ is bounded.
\end{theorem}

Note that the condition~\eqref{eq:hom-vanishing-intro} is vacuous when the exceptional sequence has length~$n = 1$.
This makes it tempting to iterate Theorem~\ref{theorem-induce-via-exceptional-sequence} one exceptional object at a time
to handle general exceptional sequences. 
However, the t-structure obtained after the first step may not be noetherian (see Remark~\ref{remark-noetherian-phantom}),
in which case we cannot apply Theorem~\ref{main-theorem}. 

On the other hand, we prove that in the situation of Theorem~\ref{theorem-induce-via-exceptional-sequence}
with only one exceptional object~$E$,
if~$E$ is $\sigma$-stable for a stability condition~$\sigma$ on~$\cD$,
then there exists an induced (even projected) noetherian bounded t-structure on~$\cC$
(see Proposition~\ref{proposition-pre-stability-tau} and Corollary~\ref{corollary-stability-noetherian-t-structure}). 
This gives one indication that, when the given t-structure on $\cD$ underlies a stability condition,
the induced t-structures on semiorthogonal components may be better behaved. 

\begin{remark}
The earlier paper \cite{BLMS} developed a method for constructing bounded t-structures
(and stability conditions) on semiorthogonal complements of exceptional sequences,
which has been applied to the residual components of some low-dimensional Fano varieties.
As we explain in Remark~\ref{remark-BLMS} and Example~\ref{example-BLMS},
this construction can be thought of as a very special case of our results. 
\end{remark}

\subsection{Applications} 
\label{subsection-intro-applications}

The forward $\Hom$-vanishing condition~\eqref{eq:hom-vanishing-intro} in Theorem~\ref{theorem-induce-via-exceptional-sequence} is fairly restrictive.
However, there are still many interesting examples satisfying this condition, some of which we describe below. 

\subsubsection*{Exceptional objects} 

As we already mentioned, condition~\eqref{eq:hom-vanishing-intro} in Theorem~\ref{theorem-induce-via-exceptional-sequence}
is vacuous when the exceptional sequence has length~$1$. 
Therefore, the semiorthogonal component~$\cC$ defined by
\begin{equation*}
\Db(X) = \langle E, \cC \rangle
\end{equation*} 
admits a bounded t-structure, connectively induced by the standard one on~$\Db(X)$.
More generally, Theorem~\ref{theorem-induce-via-exceptional-sequence} directly applies to an exceptional object in the standard heart of the derived category of a noetherian abelian category. 

The simplest example is when~$X$ is the blowup in a point of a smooth surface~$Y$,
and we take~$E = \cO_C(-1)$ where~$C \subset X$ is the exceptional curve, so that~$\cC \simeq \Db(Y)$ by Orlov's blowup formula.
In this example, it is easy to see that the t-structure~$\cin{\Db(Y)}$ connectively induced on~$\Db(Y)$
coincides with the standard t-structure. 

For other simple examples of exceptional objects, however, we obtain bounded t-structures on highly nontrivial categories: 
\begin{itemize}
\item 
If~$X$ is a Fano variety, then the structure sheaf~$\cO_X$ is exceptional.
Taking~$X$ to be any stacky weighted Fano hypersurface of index~$1$,
we obtain bounded t-structures on infinitely many fractional Calabi--Yau categories (Theorem~\ref{thm:fcy}).

\item
We construct a bounded t-structure on a curious categorical resolution~$\cC$ of a nodal cubic curve
appearing in one of the early counterexamples to the Jordan--H\"{o}lder property
for semiorthogonal decompositions~\cite{kuznetsov-jordan-holder} (Theorem~\ref{theorem-bondal-soc});
this answers a question of Haiden and Wu~\cite{haiden-wu}, who recently proved that~$\cC$ does not admit a stability condition.

\item
We construct a bounded t-structure on the Brill--Noether modifications of the derived category of a curve
recently introduced in~\cite{kuznetsov-alexeev} (Theorem~\ref{theorem-BN-modification});
as a special case, this gives such a t-structure on the crepant categorical resolution
of the residual category of a nodal cubic threefold (Example~\ref{example-cubic-threefold}).
\end{itemize}

\subsubsection*{Enriques surfaces} 
For an Enriques surface $X$, there is an exceptional sequence $L_1, \dots, L_{10}$ of $10$ line bundles, defining an interesting semiorthogonal decomposition 
\begin{equation*}
\Db(X) = \langle L_1, \dots, L_{10}, \cR_X \rangle 
\end{equation*} 
that has been studied in various works \cite{zube, ingalls-kuznetsov, LNSZ, LSZ}. 
We prove that $\cR_X$ admits a bounded t-structure (Theorem~\ref{thm:enriques}). 
When~$X$ is generic, the exceptional sequence is completely orthogonal in the sense that~$\RHom(L_i, L_j) = 0$ for~$i \neq j$,
so Theorem~\ref{theorem-induce-via-exceptional-sequence} directly applies;
in general, we instead appeal to a mild strengthening of this result, Theorem~\ref{theorem-induce-via-simple-objects}. 

\subsubsection*{Phantom categories} 

Perhaps the most interesting application of our results is to phantom categories.
Recall that a semiorthogonal component~$\cP \subset \Db(X)$ of the derived category
of a smooth projective variety is called a \emphsf{phantom}
if both its Grothendieck group~$\rK_0(\cP)$ and Hochschild homology~$\HH_\bullet(\cP)$ vanish. 
The existence of such strange categories came as a surprise when Gorchinskiy and Orlov~\cite{phantoms-orlov}
and B\"{o}hning, Graf von Bothmer, Katzarkov, and Sosna~\cite{phantoms-bohning}
constructed the first examples as semiorthogonal components on fairly exotic varieties of general type.
A striking recent example due to Krah~\cite{krah} shows that a phantom can also exist on a rational surface, 
and further examples of this type were produced in~\cite{KKLLMMPRV} (see also~\cite{MXY}). 

A phantom category cannot admit a stability condition, due to the vanishing of $\rK_0(\cP)$.
This suggests that phantoms form a critical test case for Question~\ref{question-soc-t-structure}.
Indeed, since the first examples of phantoms were constructed
it has been a well-known open question --- to our knowledge first raised in print in~2015 by Sosna~\cite{sosna-phantoms} ---
whether a phantom can admit a bounded t-structure.
As we show in Theorem~\ref{theorem-phantoms}, our Theorem~\ref{theorem-induce-via-exceptional-sequence} applies
to give an affirmative answer for almost all of the known examples of phantom categories,
as well as quasiphantom categories (defined by relaxing the vanishing of~$\rK_0(\cP)$ to finiteness).
This shows that these categories are not as ill-behaved as may have previously been expected.

\begin{remark}
\label{remark-noetherian-phantom}
Yeqin Liu \cite{liu-phantoms} proved that phantom categories do not admit \emph{noetherian} bounded t-structures.
In particular, the bounded t-structures produced in Theorem~\ref{theorem-phantoms} are not noetherian. 
This shows that in Theorem~\ref{main-theorem}\ref{main-theorem-induce-C} the induced t-structures need not inherit noetherianity.
It follows that the perverse t-structures on $\cD$ constructed in the proof of Theorem~\ref{main-theorem}\ref{main-theorem-induce-C},
which project to the induced t-structures, also need not inherit noetherianity
(see Lemma~\ref{lemma-projection-tau-cB}\ref{Phi!tau-noetherian}).
\end{remark}

From the existence of a bounded t-structure, we also deduce that the above phantoms and quasiphantoms
have a classical generator~$G \in \cP$ 
with the property that~\mbox{$\Ext^{i}(G,G) = 0$} for~\mbox{$i < 0$};
in particular, $A = \RHom(G,G)$ is a coconnective DG algebra such that~$\Dperf(A)$ is a (quasi)phantom category (Corollary~\ref{corollary-coconnective-DG-algebra}).
This answers a question of Ben Antieau.

In the case of Krah's phantom $\cP$, another solution to Antieau's question was recently found by Amal Mattoo~\cite{Mattoo},
who constructed several interesting explicit classes of objects in $\cP$ with vanishing negative self-Ext groups. 
In Example~\ref{example-mattoo-in-heart}, we give a conceptual explanation of this vanishing for some of these objects, 
by observing that they are contained in the heart of an induced t-structure. 

\subsubsection*{Smoothings of cyclic quotient singularities} 

For a large class of smooth projective surfaces~$X$ admitting a degeneration
to a surface with cyclic quotient singularities,
the paper~\cite[Theorem~1.12(2)]{jenia} produces an interesting exceptional collection~$E_1, \dots, E_n$ of vector bundles
which satisfies the forward $\Hom$-vanishing condition~\eqref{eq:hom-vanishing-intro}
and generates the derived category of an acyclic quiver without relations.
Applying Theorem~\ref{theorem-induce-via-exceptional-sequence} to the semiorthogonal decomposition
\begin{equation*}
\Db(X) = \langle E_1, \dots, E_n, \cC \rangle, 
\end{equation*}  
we find that $\cC$ admits a bounded t-structure, connectively induced by the standard one on~$\Db(X)$.

\subsection{Further directions} 

It would be very interesting to carry out our general approach to Question~\ref{question-soc-t-structure} via perverse t-structures (detailed in \S\ref{subsection-t-exactability}) in new situations.  
Even in the situation of Example~\ref{example-rational-singularities}, 
although we know that the t-structure on $\Db(X)$ connectively induces the standard t-structure on $\Db(Y)$, 
we only know that the corresponding perverse t-structure on $\Db(X)$ exists when the fibers of $f \colon X \to Y$ have dimension $\leq 1$ as in Example~\ref{example-perverse-t-structures}. 
If the perverse t-structure could be constructed without this restriction, it may give a route to generalizing Bridgeland's work on derived equivalences of threefold flops \cite{bridgeland-flops} to higher dimensions. 

As mentioned above, our interest in Question~\ref{question-soc-t-structure} was largely motivated by the theory of stability conditions. 
We will return to this topic in a sequel to this paper, 
where we will provide criteria for inducing stability conditions on semiorthogonal components. 
This is subtle since, as mentioned above, there exist semiorthogonal components (such as phantoms) that do not admit a stability condition. 

Even in cases where stability conditions do not exist, it would be interesting to study moduli spaces of objects
in the heart of the induced t-structure on a semiorthogonal component.
Example~\ref{example-mattoo-in-heart} suggests that such moduli spaces are rich even for phantoms. 
We highlight two general questions:
\begin{enumerate}
\item
\label{heart-open}
Under what conditions does the property of lying in the heart of an induced t-structure define an open condition,
and hence cut out an open subspace of the moduli space of all objects?
\item
\label{phantom-stability}
Is there an extension of the usual theory of stability conditions (which applies, for example, to phantoms)
that gives rise to proper moduli spaces of objects in the hearts of induced t-structures?
\end{enumerate} 
In relation to \eqref{phantom-stability}, it is natural to ask whether it is possible to induce slicings on phantom categories,
as a substitute for a stability condition. 

Relatedly, we believe it is worth studying moduli spaces of objects in the hearts of the perverse t-structures that we construct in the setting of Theorem~\ref{main-theorem} (see Example~\ref{example-perverse-t-structures}). 
As Bridgeland showed \cite{bridgeland-flops}, already in the geometric situation of a threefold flopping contraction, such moduli spaces have striking applications.

\subsection{Organization of the paper} 
In~\S\ref{section-preliminaries} we review preliminaries on semiorthogonal decompositions and t-structures. 
In~\S\ref{section-restricting-t-structures} we discuss restricted t-structures and criteria for their existence. 
In~\S\ref{section-projected-t-structures} we discuss projected t-structures and criteria for their existence, including as an application the existence of noetherian bounded t-structures on the complement of a stable exceptional object.  
In~\S\ref{section-induced-t-structures} we discuss induced t-structures and criteria for their existence; in particular, we prove our main result, Theorem~\ref{main-theorem}, as well as Theorem~\ref{theorem-induce-via-exceptional-sequence}. 
In~\S\ref{section-applications} we provide the details for the applications described in~\S\ref{subsection-intro-applications}.  
Finally, in Appendix~\ref{sec:ind} we discuss another approach to the proof of Theorem~\ref{main-theorem}. 

\subsection{Conventions}
A variety over a field~$\kk$ is an integral scheme which is separated and of finite type over~$\kk$.
For a variety, or more generally an algebraic stack~$X$, 
we denote by~$\Dqc(X)$ the unbounded derived category of quasi-coherent sheaves.
If~$X$ is noetherian, 
then  we denote by $\rD(X)$, $\rD^{+}(X)$, $\Dm(X)$, and~$\Db(X)$ the subcategories of~$\Dqc(X)$
consisting of objects with coherent cohomology sheaves that are unbounded, bounded below, bounded above, or bounded. 
All functors are derived. 
Given objects~$D,D'$ in a triangulated category~$\cD$,
we write~$\RHom(D,D') \coloneqq \bigoplus_{n \in \bZ} \Hom(D,D'[n])[-n]$.
\subsection{Acknowledgements}

We are grateful to the following people for useful discussions related to this work:
Arend Bayer, Sasha Efimov, Yeqin Liu, Zhiyu Liu, Emanuele Macr\`{i}, Amal Mattoo, Dmitrii Pirozhkov,
Jørgen Rennemo, Germ\'an Stefanich, Paolo Stellari, Jenia Tevelev, and Yukinobu Toda.

Parts of this work were completed while S.L. and A.P. were in residence at the Simons Laufer Mathematical Sciences Institute in Spring 2024, 
while S.L. was visiting the Max Planck Institute for Mathematics, Kavli Institute for the Physics and Mathematics of the Universe, Universi\'e Paris--Saclay, and Universit\`a degli Studi di Milano,
and while A.K. and A.P. were at the University of Milan for Paolo Stellari's birthday conference. 
We thank these institutions for providing excellent working conditions. 

During the preparation of this paper,
A.K.~was partially supported by the HSE University Basic Research Program,
S.L.~was partially supported by an AMS-Simons Travel Grant,
and A.P. was partially supported by NSF grants DMS-2112747, DMS-2052750, and DMS-2143271, and a Sloan Research Fellowship.

%%%%%%%%%%%%%%%%%%%%%%%%%%%%%%%%%%%%%%%%%%%%%%%%%%%%%%%

\section{Preliminaries}
\label{section-preliminaries} 

The purpose of this section is to review some basic facts about semiorthogonal decompositions, t-structures, and torsion pairs.

\subsection{Semiorthogonal decompositions}
\label{section-sod}

\begin{definition}
\label{definition-sod}
Let~$\cD$ be a triangulated category.
A \emphsf{semiorthogonal decomposition} of~$\cD$, denoted~$\cD = \langle \cD_1, \cD_2, \dots, \cD_n \rangle$,
is a sequence~$\cD_1, \cD_2, \dots, \cD_n$ of full triangulated subcategories of~$\cD$ such that:
\begin{enumerate}[label={\textup{(\arabic*)}}]
\item
(\emphsf{$\RHom$-vanishing})
For all $D_i \in \cD_i$, $D_j \in \cD_j$ with $i > j$, we have $\RHom(D_i,D_j) = 0$. 
\item
\label{sod-filtration} 
(\emphsf{Generation property})
For any $D \in \cD$, there exists a sequence of morphisms
\begin{equation*}
0 = D_n \to D_{n-1} \to \cdots \to D_1 \to D_0 = D
\end{equation*}
such that for all $1 \leq i \leq n$ the cone of the morphism $D_{i} \to D_{i-1}$ lies in $\cD_i$.
\end{enumerate}
The subcategories $\cD_i \subset \cD$ are called the \emphsf{components} of the semiorthogonal decomposition. 
\end{definition}

It follows from Definition~\ref{definition-sod} that the filtration of~$D \in \cD$
appearing in condition~\ref{sod-filtration}, as well as its graded pieces $\cone(D_{i} \to D_{i-1})$, are functorial.
The functor~$\pr_i \colon \cD \to \cD$ given by the $i$th graded piece
is called the \emphsf{projection functor} onto~$\cD_i$.

Given a subcategory $\cC \subset \cD$ of a triangulated category $\cD$, its right and left orthogonal subcategories are defined by
\begin{align*}
    \cC^{\perp} & = \set{D \in \cD \sth {\RHom(C, D)} = 0 \text{ for all } C \in \cC \text{ and } n \in \bZ}, \\
    {^\perp}\cC & = \set{D \in \cD \sth {\RHom(D, C)} = 0 \text{ for all } C \in \cC \text{ and } n \in \bZ}.
\end{align*}
In these terms, the following well-known (and straightforward) lemma gives a useful method for constructing semiorthogonal decompositions.

\begin{lemma}
\label{lemma-sod}
    Let~$\cD$ be a triangulated category.
    Let~$\cB, \cC \subset \cD$ be triangulated subcategories.
    Then the following are equivalent:
    \begin{enumerate}[label={\textup{(\arabic*)}}]
        \item $\cD = \langle \cB, \cC \rangle$ is a semiorthogonal decomposition.
        \item \label{left-admissible} 
        $\cB \subset \cD$ is \emphsf{left admissible}, i.e. the inclusion admits a left adjoint, and~$\cC = {^\perp}\cB$.
        \item \label{right-admissible} 
        $\cC \subset \cD$ is \emphsf{right admissible}, i.e. the inclusion admits a right adjoint, and~$\cB = \cC^{\perp}$.
    \end{enumerate}
Explicitly, if~$\cD = \langle \cB, \cC \rangle$ is a semiorthogonal decomposition,
then the projection functor onto~$\cB$ is the left adjoint to the inclusion~$\cB \to \cD$,
and the projection functor onto~$\cC$ is the right adjoint to the inclusion~$\cC \to \cD$.
\end{lemma}

Now we specialize to the case where~$\cD$ is a $\kk$-linear triangulated category for a field~$\kk$.
We say that an object~$D \in \cD$ is \emphsf{homologically finite-dimensional}
if~$\RHom(D,D') \in \Db(\kk)$ and~$\RHom(D',D) \in \Db(\kk)$
for all~$D' \in \cD$ (cf.~\cite[Definition~3.1]{KS25}).
Note that if~$\cD \subset  \Db(X)$ where~$X$ is smooth and proper,
then every object in~$\cD$ is homologically finite-dimensional.

\begin{definition}
\label{def:exceptional}
Let~$\cD$ be a $\kk$-linear triangulated category for a field~$\kk$.
An object~$E \in \cD$ is \emphsf{exceptional} if it is homologically finite-dimensional and the canonical map~$\kk \to \RHom(E,E)$
(classifying the identity morphism~$\id_E \colon E \to E$) is an isomorphism.
A sequence of objects~$E_1, \dots, E_n \in \cD$ is called an \emphsf{exceptional sequence}
if the object $E_i$ is exceptional for all $i$ and~$\RHom(E_i, E_j) = 0$ for~$i > j$.
\end{definition}

The following result is well-known.
\begin{lemma}
\label{lemma-exceptional-sequence}
Let~$\cD$ be a $\kk$-linear triangulated category for a field $\kk$.
\begin{enumerate}[label={\textup{(\arabic*)}}]
\item
\label{exceptional-object-ff-adjoints}
Let $E \in \cD$ be an exceptional object.
Then the functor $\phi_E \colon \Db(\kk) \to \cD$, $V \mapsto V \otimes E$ is fully faithful
and admits left and right adjoints.
\item
\label{exceptional-sequence-sod}
Let~$E_1, \dots, E_n$ be an exceptional sequence in~$\cD$.
Then there are semiorthogonal decompositions
\begin{align*}
\cD &= \langle E_1, \dots, E_n, \cC \rangle , \\
\cD &= \langle \cB, E_1, \dots, E_n \rangle,
\end{align*}
where~$\cB = \langle E_1, \dots, E_n \rangle^\perp$ and~$\cC = {}^\perp\langle E_1, \dots, E_n \rangle$,
and as is customary we write simply~$E_i$ for the triangulated subcategory~$\langle E_i \rangle = \phi_{E_i}(\Db(\kk)) \subset {\cD}$
generated by~$E_i$ in the above decompositions.
\end{enumerate}
\end{lemma}

\begin{proof}
\ref{exceptional-object-ff-adjoints}
The fully faithfulness of~$\phi_E$ follows directly from the exceptionality of~$E$.
Moreover, it is straightforward to check that the functors~$\phi_E^*, \phi_E^! \colon \cD \to \Db(\kk)$ given by
\begin{align*}
\phi_E^*(D) = \RHom(D,E)^{\vee}
\quad \text{and} \quad
\phi_E^!(D) = \RHom(E,D)
\end{align*}
are the left and right adjoints of~$\phi_E$;
note that they take values in $\Db(\kk)$ by the assumption that $E$ is homologically finite-dimensional.

In view of~\ref{exceptional-object-ff-adjoints}, the claims in~\ref{exceptional-sequence-sod}
follow from Lemma~\ref{lemma-sod} by induction. 
\end{proof}

\subsection{t-structures}
\label{section-t-structure}

The general reference for this subsection is~\cite{BBDG}.

\begin{definition}
\label{def:t-str}
Let~$\cD$ be a triangulated category.
A \emphsf{t-structure} on~$\cD$ is a pair~$\tau = (\cD^{\leq 0}, \cD^{\geq 0})$ of strictly full subcategories of~$\cD$ such that:
    \begin{enumerate}[label={\textup{(\arabic*)}}]
        \item
        \label{it:t-str-shift}
        (\emphsf{Compatibility with shifts})
        There are inclusions~$\cD^{\leq 0}[1] \subset \cD^{\leq 0}$ and~$\cD^{\geq 0}[-1] \subset \cD^{\geq 0}$.
        \item
        \label{it:t-str-perp}
        (\emphsf{$\Hom$-orthogonality})
        For all~$D  \in \cD^{\leq 0}$ and~$D'  \in \cD^{\geq 0}[-1]$, we have $\Hom(D, D') = 0$. 
        \item
        \label{it:t-str-triangle}
        (\emphsf{Generation property})
        For any object~$D \in \cD$ there exists a distinguished triangle
        \begin{equation*}
            D^{\le 0} \to D \to D^{\ge 1}
        \end{equation*}
        where~$D^{\le 0} \in \cD^{\leq 0}$ and~$D^{\ge 1} \in \cD^{\geq 0}[-1]$.
    \end{enumerate}
The subcategories~$\cD^{\leq 0}$ and~$\cD^{\geq 0}$ of~$\cD$
are called the \emphsf{connective} and \emphsf{coconnective} parts of the t-structure. 
\end{definition}

The $\Hom$-vanishing in~\ref{it:t-str-perp} implies that the triangle in~\ref{it:t-str-triangle} is unique and functorial. 
That is, there are functors~$\tau^{\le 0}, \tau^{\ge 1} \colon \cD \to \cD$ given on objects by 
\begin{equation*}
\tau^{\le 0}(D) \coloneqq D^{\le 0}
\qquad\text{and}\qquad
\tau^{\ge 1}(D) \coloneqq D^{\ge 1}. 
\end{equation*}
Moreover, $\tau^{\le 0}$ is (the composition with the inclusion of) the right adjoint of~$\cD^{\le 0} \hookrightarrow \cD$,
and~$\tau^{\geq 1}$ is (the composition with the inclusion of) the left adjoint of~$\cD^{\geq 0}[-1] \hookrightarrow \cD$.
The above axioms also imply that~$\cD^{\le 0}$ and~$\cD^{\ge 0}$ determine each other:
\begin{equation}
\label{eq:cdm-cdp}
\begin{aligned}
\cD^{\le 0} &= \set{ D \in \cD \sth \Hom(D, D') = 0  \text{ for all } D' \in {\cD}^{\ge 0}[-1] },
\\
\cD^{\ge 0} &= \set{ D \in \cD \sth \Hom(D', D) = 0  \text{ for all } D' \in {\cD}^{\le 0}[1] }.
\end{aligned}
\end{equation}
These formulas also show that the subcategories~$\cD^{\le 0}$ and~$\cD^{\ge 0}$ are extension closed.

Given a t-structure~$\tau$, for any~$n \in \bZ$ we define full subcategories of~$\cD$ by
\begin{equation*}
\cD^{\leq n} = \cD^{\leq 0}[-n] \quad \text{and} \quad
\cD^{\geq n} = \cD^{\geq 0}[-n],
\end{equation*}
and functors $\tau^{\leq n}, \tau^{\geq n+1} \colon \cD \to \cD$ by 
\begin{equation*}
\tau^{\le n}(D) \coloneqq \tau^{\le 0}(D[n])[-n]
\qquad\text{and}\qquad
\tau^{\ge n+1}(D) \coloneqq \tau^{\ge 1}(D[n])[-n], 
\end{equation*}
so that 
for any object $D \in \cD$ we have a distinguished triangle
\begin{equation*}
\tau^{\leq n} D \to D \to \tau^{\geq n+1} D,
\end{equation*}
where the first map is the counit of the adjunction and the second is the unit.
The triangle in Definition~\ref{def:t-str} is a particular case of this for~$n = 0$.
These triangles are called \emphsf{truncation triangles},
and the functors~$\tau^{\le n}$, $\tau^{\ge n+1}$ are called \emphsf{truncation functors}.

The \emphsf{heart} of $\tau$ is defined as the full subcategory of $\cD$ given by the intersection
\begin{equation*}
\cD^{\heartsuit} = \cD^{\leq 0} \cap \cD^{\geq 0}.
\end{equation*}
It is a classical result \cite[{Th\'eor\`eme~1.3.6}]{BBDG} that $\cD^{\heartsuit}$ is an abelian category.
Moreover, for any $n \in \bZ$ we have the \emphsf{cohomology functor}
\begin{equation*}
\mathcal{H}^n \coloneqq [n] \circ  \tau^{\leq n} \circ \tau^{\geq n}
\cong \tau^{\leq 0} \circ \tau^{\geq 0} \circ [n]
\colon  \cD \to \cD^{\heartsuit}.
\end{equation*}
In particular, $\cH^0 = \tau^{\leq 0} \circ \tau^{\geq 0}$ and moreover, $\cH^n(D) = \cH^0(D[n])$.
The cohomology functors are often used to translate a distinguished triangle in~$\cD$
to a long exact sequence in its heart~$\cD^\heartsuit$ --- if~$D_1 \to D_2 \to D_3$ is a distinguished triangle, then
\begin{equation}
\label{eq:les}
\dots \to \cH^{i-1}(D_3) \to \cH^i(D_1) \to \cH^i(D_2) \to \cH^i(D_3) \to \cH^{i+1}(D_1) \to \dots
\end{equation}
is a long exact sequence.

\begin{notation}
For clarity, given a t-structure~$\tau$ on a triangulated category~$\cD$,
we sometimes decorate its attendant structures with the symbol~$\tau$.
In particular, we write~${^\tau}\cD^{\leq n}$, ${^\tau}\cD^{\geq n}$, ${^\tau}\cD^{\heartsuit}$, and~${^\tau}\cH^n$
for the subcategories and cohomology functors which were denoted above without the symbol~$\tau$.
\end{notation}

For functors between categories equipped with t-structures, there are natural notions of t-exactness.

\begin{definition}
\label{def:amplitude}
    Let $\cC$ and $\cD$ be triangulated categories equipped with t-structures~$\tau_{\cC}$ and~$\tau_{\cD}$.
    Let $\Phi \colon \cC \to \cD$ be a triangulated functor.
    \begin{enumerate}[label={\textup{(\arabic*)}}]
        \item $\Phi$ is \emphsf{left t-exact} with respect to $\tau_{\cC}$ and $\tau_{\cD}$ if
        $\Phi({^{\tau_{\cC}}}\cC^{\geq 0}) \subset {^{\tau_{\cD}}}\cD^{\geq 0}$.
        \item $\Phi$ is \emphsf{right t-exact} with respect to $\tau_{\cC}$ and $\tau_{\cD}$ if $\Phi({^{\tau_{\cC}}}\cC^{\leq 0}) \subset {^{\tau_{\cD}}}\cD^{\leq 0}$.
        \item $\Phi$ is \emphsf{t-exact} with respect to $\tau_{\cC}$ and $\tau_{\cD}$ if it is both left and right t-exact.
    \end{enumerate}
More generally, for integers $a, b \in \bZ$, we say: 
\begin{enumerate}[label={\textup{(\arabic*)}}, resume]
\item
$\Phi$ has \emphsf{left t-amplitude~$\geq a$} with respect to $\tau_{\cC}$ and $\tau_{\cD}$
if~$\Phi({^{\tau_{\cC}}}\cC^{\geq 0}) \subset {^{\tau_{\cD}}}\cD^{\geq a}$.
\item
$\Phi$ has \emphsf{right t-amplitude~$\leq b$} with respect to $\tau_{\cC}$ and $\tau_{\cD}$
if~$\Phi({^{\tau_{\cC}}}\cC^{\leq 0}) \subset {^{\tau_{\cD}}}\cD^{\leq b}$.
\item
$\Phi$ has \emphsf{t-amplitude in $[a,b]$} with respect to $\tau_{\cC}$ and $\tau_{\cD}$
if it has left t-amplitude $\geq a$ and right t-amplitude~$\leq b$.
\end{enumerate}
\end{definition}

The following lemma is a well-known formal consequence of the definitions.
\begin{lemma}
\label{lemma-t-exactness-adjoints}
Let~$\Phi \colon \cC \to \cD$ be a triangulated functor which admits a right adjoint $\Phi^! \colon \cD \to \cC$. 
Assume~$\cC$ and~$\cD$ are equipped with t-structures~$\tau_{\cC}$ and~$\tau_{\cD}$.
Then~$\Phi$ is right t-exact with respect to~$\tau_{\cC}$ and~$\tau_{\cD}$
if and only if~$\Phi^!$ is left t-exact with respect to~$\tau_{\cD}$ and~$\tau_{\cC}$. 
Moreover, in this case the functor~${^{\tau_{\cD}}}\cH^0 \circ \Phi \colon {^{\tau_{\cC}}}\cC^{\heartsuit} \to {^{\tau_{\cD}}}\cD^{\heartsuit}$
is left adjoint to the functor~${^{\tau_{\cC}}}\cH^{0} \circ \Phi^! \colon {^{\tau_{\cD}}}\cD^{\heartsuit} \to {^{\tau_{\cC}}}\cC^{\heartsuit}$
between the corresponding abelian categories. 
\end{lemma}

One often considers the following boundedness conditions. 

\begin{definition}
\label{definition-t-structure-conditions}
Let $\tau  = (\cD^{\leq 0}, \cD^{\geq 0})$ be a t-structure on a triangulated category~$\cD$.
We define the triangulated subcategories
\begin{equation*}
{}^\tau\cD^+ \coloneqq \bigcup_{n \in \bZ} \cD^{\geq n},
\qquad
{}^\tau\cD^- \coloneqq \bigcup_{n \in \bZ} \cD^{\leq n},
\quad\text{and}\quad
{}^\tau\cD^{\mathrm{b}} \coloneqq {}^\tau\cD^+ \cap {}^\tau\cD^-
\end{equation*}
of \emphsf{bounded below}, \emphsf{bounded above}, and \emphsf{bounded} objects in~$\cD$, respectively.

We say~$\tau$ is \emphsf{bounded below} if~$\cD = {}^\tau\cD^+$, \emphsf{bounded above} if $\cD = {}^\tau\cD^-$, and \emphsf{bounded} if $\cD = {}^\tau\cD^{\mathrm{b}}$. 
\end{definition}

\begin{remark}
If~$\tau$ is a bounded below t-structure, then~${^\tau}\cD^{\leq 0}$ can be described
as the smallest extension closed subcategory of~$\cD$ containing~${^\tau}\cD^{\heartsuit}[n]$ for all~$n \geq 0$.
Similarly, if~$\tau$ is a bounded above t-structure, then~${^\tau}\cD^{\geq 0}$
is the smallest extension closed subcategory of~$\cD$ containing~${^\tau}\cD^{\heartsuit}[n]$ for all~$n \leq 0$.
Taking~\eqref{eq:cdm-cdp} into account, we conclude that
in either case~$\tau$ is completely determined by its heart~${^\tau}\cD^{\heartsuit} \subset \cD$.
\end{remark}

The following finiteness conditions on the heart of a t-structure play an important role in this paper.

\begin{definition}
An abelian category $\cA$ is \emphsf{noetherian}
if any ascending chain of subobjects~$A_0 \subset A_1 \subset \cdots \subset A$ of an object~$A \in \cA$ stabilizes, or equivalently, if any chain of epimorphisms~$A_0 \twoheadrightarrow A_1 \twoheadrightarrow A_2 \twoheadrightarrow \cdots$ in $\cA$ stabilizes.

Similarly, $\cA$ is \emphsf{artinian}
if any descending chain of subobjects~$A_0 \supset A_1 \supset A_2 \supset \dots$ stabilizes,
or equivalently, if any ascending chain of epimorphisms~$A_0 \twoheadleftarrow A_1 \twoheadleftarrow \dots \twoheadleftarrow A$
of an object~\mbox{$A \in \cA$} stabilizes.

A t-structure~$\tau$ on a triangulated category~$\cD$ is \emphsf{noetherian} or \emphsf{artinian}
if its heart~${}^{{\tau}}\cD^{\heartsuit}$ is a noetherian or artinian abelian category. 
\end{definition}

\begin{example}
\label{ex:tau-x}
Let~$X$ be a scheme, or more generally an algebraic stack.
The unbounded derived category of quasi-coherent sheaves~$\Dqc(X)$
carries the \emphsf{standard t-structure}~$\tau_X$, defined by
\begin{align*}
{^{\tau_X}}\Dqc(X)^{\leq 0} & = \set{ F \in \Dqc(X) \sth \cH^n(F) = 0 \text{ for } n > 0},  \\
{^{\tau_X}}\Dqc(X)^{\geq 0} & = \set{ F \in \Dqc(X) \sth \cH^n(F) = 0 \text{ for } n < 0},  
\end{align*}
where~$\cH^n(F)$ denotes the degree~$n$ cohomology sheaf of~$F$
(which coincides with the cohomology object~${^{\tau_X}}\cH^n(F)$ defined by $\tau_X)$.
If~$X$ is noetherian, then there are similarly defined t-structures on the subcategories~$\rD(X)$,
$\rD^{+}(X)$, $\Dm(X)$, $\Db(X)$ of objects with coherent cohomology sheaves
that are unbounded, bounded below, bounded above, or bounded.
For simplicity, we still denote the t-structure by~$\tau_X$ in all of these cases.
The t-structure~$\tau_X$ on~$\rD(X)$, $\rD^{+}(X)$, $\Dm(X)$, and~$\Db(X)$ is noetherian,
and it is moreover bounded on~$\Db(X)$.
\end{example} 

\begin{remark}
\label{remark-opposite-t-structure}
While many naturally occurring t-structures (such as $\tau_X$ on $\Db(X)$ above) are noetherian,
artinian t-structures are less common. 
However, there is a formal way to turn any noetherian t-structure into a artinian one, by passing to opposite categories.
More precisely, let $\cD$ be a triangulated category, equipped with a t-structure $\tau$. 
Let $\cD^{\op}$ be the opposite category of $\cD$, 
with the natural triangulated structure (in which the shift~$[1]$ is equal to~$[-1]^\op$). 
Then there is an \emphsf{opposite t-structure} on $\cD^{\op}$, given by 
\begin{equation*}
\tau^\op \coloneqq \big( ({}^\tau\cD^{\ge 0})^\op,({}^\tau\cD^{\le 0})^\op \big). 
\end{equation*} 
As is clear from the definition, passing to opposites preserves the property that a t-structure is bounded, and swaps the properties of being bounded below and bounded above. 
Further, the heart of the opposite t-structure is the opposite of the heart of $\tau$, i.e. 
\begin{equation*}
{^{\tau^{\op}}}(\cD^{\op})^{\heartsuit} = (^{\tau}\cD^{\heartsuit})^{\op}. 
\end{equation*}
In particular, we see that passing to opposites swaps the noetherian and artinian properties of t-structures.
\end{remark}

\begin{example}
\label{example-dual-standard-t-structure}
Let $X$ be a noetherian scheme which admits a dualizing complex $\omega_{X}^{\bullet}$; for instance, $X$ could be a scheme of finite type over a field. 
Then by \cite[\href{https://stacks.math.columbia.edu/tag/0A89}{Tag 0A89}]{stacks-project} 
 the functor $\cRHom(- , \omega^\bullet_X)$ 
gives equivalences 
\begin{equation*}
\rD(X) \simeq \rD(X)^{\op}, \quad 
\rD^{+}(X) \simeq \Dm(X)^{\op}, \quad \Dm(X) \simeq \rD^{+}(X)^{\op}, \quad \text{and} \quad 
\Db(X) \simeq \Db(X)^{\op}. 
\end{equation*} 
By Remark~\ref{remark-opposite-t-structure}, the opposite categories appearing on the right-hand side of these equivalences are equipped with the opposite standard t-structure $\tau_{X}^{\op}$, which we can transport along the equivalences to obtain the \emph{dual standard t-structure} $\tau_X^{\vee}$ on $\rD(X)$, $\rD^{+}(X)$, $\Dm(X)$, and~$\Db(X)$. 
For instance, on $\Db(X)$ the dual standard t-structure is given explicitly by 
\begin{align*}
{}^{\tau_X^\vee}\Db(X)^{\leq 0} & = \set{ F \in \Db(X) \sth \cRHom(F,\omega^\bullet_X) \in {}^{\tau_X}\Db(X)^{\ge 0} },
\\
{}^{\tau_X^\vee}\Db(X)^{\geq 0} & = \set{ F \in \Db(X) \sth \cRHom(F,\omega^\bullet_X) \in {}^{\tau_X}\Db(X)^{\le 0} }. 
\end{align*}
By Remark~\ref{remark-opposite-t-structure},
this t-structure on $\Db(X)$ is bounded with heart~${}^{\tau_X^\vee}\Db(X)^\heartsuit \simeq \Coh(X)^{\op}$;
in particular, this t-structure is artinian.
\end{example}

\subsection{Torsion pairs}
\label{section-torsion-pairs}

Finally, we recall a useful procedure, known as tilting, for constructing new t-structures from a given one.

\begin{definition}
\label{def:torsion-pair}
    Let $\cA$ be an abelian category. 
    A \emphsf{torsion pair} in $\cA$ is a pair of full subcategories $(\cT, \cF)$ of $\cA$ such that: 
    \begin{enumerate}[label={\textup{(\arabic*)}}]
        \item
        \label{torsion:orthogonality} 
        (\emphsf{$\Hom$-orthogonality})
        For all $T \in \cT$ and $F \in \cF$ we have $\Hom(T,F) = 0$.
        \item
        \label{torsion:decomposition}
        (\emphsf{Generation property})
        For any $A \in \cA$ there exists a short exact sequence
        \begin{equation*}
            0 \to T_A \to A \to F_A \to 0
        \end{equation*}
        where $T_A \in \cT$ and $F_A \in \cF$. 
    \end{enumerate}
The subcategories~$\cT$ and~$\cF$ in~$\cA$ are called the \emphsf{torsion} and \emphsf{torsion free} parts of the torsion pair.
\end{definition}

If $(\cT, \cF)$ is a torsion pair in $\cA$, then~$\cT$ and~$\cF$ determine each other:
\begin{equation}
\label{eq:ct-cf}
\begin{aligned}
\cT  &= \set{ A \in \cA \sth \Hom(A, F) = 0  \text{ for all }F \in \cF },
\\
\cF &= \set{ A  \in \cA \sth \Hom(T, A) = 0  \text{ for all } T \in \cT }.
\end{aligned}
\end{equation}
It follows that~$\cT$ is closed under extensions and quotients,
while~$\cF$ is closed under extensions and subobjects.
Moreover, by the following well-known lemma, 
it is easy to construct torsion pairs in a noetherian or artinian abelian category:
it is enough to specify a subcategory~$\cT \subset \cA$ or~$\cF \subset \cA$ with appropriate properties. 

\begin{lemma}
\label{lemma-torsion-pair} 
Let~$\cA$ be a noetherian abelian category.
Let~$\cT \subset \cA$ be a full subcategory which is closed under extensions and quotients.
Let~$\cF = \set{ A  \in {\cA} \sth \Hom(T, A) = 0 ~ \text{for all~$T \in \cT$} }$.
Then~$(\cT, \cF)$ is a torsion pair in $\cA$.

Similarly, let~$\cA$ be an artinian abelian category.
Let~$\cF \subset \cA$ be a full subcategory which is closed under extensions and subobjects.
Let~$\cT = \set{ A \in \cA \sth \Hom(A, F) = 0  \text{ for all }F \in \cF }$.
Then~$(\cT, \cF)$ is a torsion pair in $\cA$.
\end{lemma} 

\begin{proof}
For any object~$A \in \cA$, by noetherianity we can choose a subobject~$T_A \subset A$
which is maximal with the property that~${T_A} \in \cT$.
Then using that~$\cT$ is closed under extensions and quotients, it is straightforward to check
that the quotient~$F_A = A/T_A$ is contained in~$\cF$.

Similarly, for any object~$A \in \cA$, by artinianity 
we can choose a quotient~\mbox{$A \twoheadrightarrow F_A$}
which is maximal with the property that~${F_A} \in \cF$
and then it is straightforward to check that the subobject~$T_A = \ker(A \to F_A)$ is contained in~$\cT$.
\end{proof}

Any torsion pair in the heart of a t-structure gives rise to a new t-structure as follows. 

\begin{proposition}[{\cite[{Proposition~I.2.1}]{HRS}}]
\label{proposition-tilting}
Let~$\cD$ be a triangulated category equipped  with a t-structure~$\tau$.
Let~$(\cT, \cF)$ be a torsion pair in~${^\tau}\cD^{\heartsuit}$.
Then there is a t-structure~$\tau^{\sharp}$, called the \emphsf{tilt of~$\tau$ with respect to~$(\cT, \cF)$}, given by
\begin{align*}
    {^{\tau^{\sharp}}}\cD^{\leq 0} & = \set{D \in {^\tau}\cD^{\leq 0\hphantom{-}} \sth {^\tau}\cH^0(D) \in \cT} , \\
    {^{\tau^{\sharp}}}\cD^{\geq 0} & = \set{D \in {^\tau}\cD^{\geq -1} \sth {^\tau}\cH^{-1}(D) \in \cF}. 
\end{align*}
\end{proposition}

\begin{remark}
\label{remark-characterizing-tilt}
If $\tau^{\sharp}$ is a t-structure obtained from $\tau$ by tilting, then it follows from the definition that there are inclusions 
\begin{equation}
\label{eq:tilt-inclusions}
{^{\tau}}\cD^{\leq -1} \subset
{^{\tau^{\sharp}}}\cD^{\leq 0} \subset
{^{\tau}}\cD^{\leq 0}
\quad \text{and} \quad
{^{\tau}}\cD^{\ge 0} \subset {^{\tau^{\sharp}}}\cD^{\ge 0} \subset {^{\tau}}\cD^{\ge -1}.
\end{equation}
Therefore, the subcategories~${}^{\tau^\sharp}\cD^? = {}^\tau\cD^?$ with~$? \in \{+,-,\mathrm{b}\}$
of bounded below, bounded above, and bounded objects in~$\cD$ are preserved by tilting.
In particular, the property that a t-structure is bounded below, or bounded above, or bounded is preserved.

We also note that tilted t-structures are characterized by relation~\eqref{eq:tilt-inclusions}. 
Namely, if~$\tau^{\sharp}$ is a t-structure for which the  inclusions~\eqref{eq:tilt-inclusions}
(or equivalently for which one of the string of inclusions hold),
then~$\tau^{\sharp}$ is obtained from~$\tau$ by tilting~\cite[Lemma~1.1.2]{polishchuk2007aisle}.
\end{remark} 

%%%%%%%%%%%%%%%%%%%%%%%%%%%%%%%%%%%%%%%%%%%%%%%%%%%%%%%

\section{Restricted t-structures} 
\label{section-restricting-t-structures} 

In this section, we discuss restriction of t-structures,
which is the most naive method for constructing a t-structure on a triangulated subcategory from one on the ambient category. 

In \S\ref{ss:restricted-def} we define restricted t-structures and discuss their basic properties. 
In~\S\ref{section-via-sod} and~\S\ref{section-via-simple-generators}
we discuss criteria for the existence of a restricted t-structure
when the subcategory is equipped with extra structure:
either a semiorthogonal decomposition or a suitable collection of simple generators.

\subsection{Definition and basic properties} 
\label{ss:restricted-def}

\begin{definition}
\label{definition-restriction-t-structure}
Let~$\cD$ be a triangulated category with a t-structure $\tau = ({}^\tau\cD^{\leq 0}, {}^\tau\cD^{\geq 0})$.
Let~$\cB \subset \cD$ be a full triangulated subcategory.
We say that $\tau$ \emphsf{restricts to a t-structure} on~$\cB$ if the pair 
\begin{equation}
\label{eq:t-restricted-body}
\tau\vert_{\cB} \coloneqq ({}^\tau\cD^{\leq 0} \cap \cB, {}^\tau\cD^{\geq 0} \cap \cB)
\end{equation}
defines a t-structure on~$\cB$, in which case $\tau\vert_{\cB}$ is called the \emphsf{restricted t-structure}. 
\end{definition}

\begin{remark}
The definition of the restricted t-structure does not require the existence of an adjoint to the embedding functor $\cB \hookrightarrow \cD$ (unlike projected t-structures discussed in \S\ref{section-projected-t-structures}), so~$\cB \subset \cD$ need not be admissible.
\end{remark} 

\begin{lemma}
\label{lemma-beta-exact}
Let~$\cD$ be a triangulated category with a t-structure~$\tau = ({}^\tau\cD^{\leq 0}, {}^\tau\cD^{\geq 0})$.
Let~$\cB \subset \cD$ be a full triangulated subcategory with embedding functor $\beta \colon \cB \to \cD$.
If $\tau$ restricts to a t-structure on $\cB$, then the restricted t-structure $\tau \vert_{\cB}$ has the following properties: 
\begin{enumerate}[label={\textup{(\arabic*)}}]
\item
\label{beta-t-exact}
$\beta \colon \cB \to \cD$ is t-exact with respect to $\tau \vert_{\cB}$ and $\tau$.
\item
\label{tau-cB-bounded}
If $\tau$ is bounded below, or bounded above, or bounded, then so is $\tau \vert_{\cB}$.
\item
\label{tau-cB-noetherian}
If $\tau$ is noetherian or artinian, then so is $\tau \vert_{\cB}$.
\end{enumerate} 
As a converse to~\ref{beta-t-exact}, if $\tau_{\cB}$ is a t-structure on $\cB$ such that $\beta \colon \cB \to \cD$ is t-exact with respect to $\tau_{\cB}$ and $\tau$, then $\tau$ restricts to a t-structure on $\cB$ and the restriction $\tau \vert_{\cB}$ coincides with $\tau_{\cB}$. 
\end{lemma}

\begin{proof}
All of the claims follow directly from the definitions. 
\end{proof} 

\begin{lemma}
\label{lemma-restriction-t-structure}
Let~$\cD$ be a triangulated category with a t-structure $\tau = ({}^\tau\cD^{\leq 0}, {}^\tau\cD^{\geq 0})$. 
Let~$\cB \subset \cD$ be a full triangulated subcategory.
Then the following conditions are equivalent: 
\begin{enumerate}[label={\textup{(\roman*)}}]
\item $\tau$ restricts to a t-structure on $\cB$. 
\item The truncation functor $\tau^{\leq 0}$ preserves the subcategory $\cB \subset \cD$. 
\item The truncation functor $\tau^{\geq 0}$ preserves the subcategory $\cB \subset \cD$. 
\item For all $n \in \bZ$ the truncation functors $\tau^{\leq n}$ and $\tau^{\geq n}$ preserve the subcategory $\cB \subset \cD$. 
\end{enumerate} 
\end{lemma}

\begin{proof}
This is an immediate consequence of the definitions. 
\end{proof}

\begin{example}
\label{example-E-in-Dheart}
Let $\cD$ be a $\kk$-linear triangulated category for a field $\kk$. 
Let $\tau$ be a t-structure on $\cD$. 
Let $E \in {^\tau}\cD^{\heartsuit}$ be an exceptional object contained in the heart of $\tau$. 
Then $\tau$ restricts to a t-structure on the triangulated subcategory $\cB = \langle E \rangle \subset \cD$ 
generated by $E$, 
with heart
\begin{equation*}
{^{\tau\vert_{\cB}}}\cB^{\heartsuit} = \set{E^{\oplus n} \sth n \geq 0}. 
\end{equation*}
Under the equivalence $\Db(\kk) \to \cB$, $V \mapsto V \otimes E$,
the restricted t-structure $\tau_{\cB}$ simply corresponds to the standard t-structure on~$\Db(\kk)$.
\end{example} 

\begin{example}
Let~$\cD$ be a triangulated category equipped with a t-structure~$\tau$.
Then~$\tau$ restricts to:
\begin{itemize}
\item a bounded below t-structure on the subcategory~${}^\tau\cD^+$;
\item a bounded above t-structure on the subcategory~${}^\tau\cD^-$; and
\item a bounded t-structure on the subcategory~${}^\tau\cD^{\mathrm{b}}$.
\end{itemize} 
Moreover, the heart of each of these restricted t-structures coincides with that of~$\tau$.
\end{example}

\subsection{Existence via semiorthogonal decompositions}
\label{section-via-sod} 

The following result is a slight elaboration of a lemma of Collins and Polishchuk~\cite{collins2010gluing}. 

\begin{lemma}
\label{lemma-t-structure-gluing}
Let~$\cB$ be a triangulated category and let~$\cB = \langle \cB_1, \cB_2, \dots, \cB_n \rangle$ be a semiorthogonal decomposition
with projection functors~$\pr_i \colon \cB \to \cB_i$.
Assume that~$\tau_i$ is a t-structure on~$\cB_i$ for~\mbox{$1 \leq i \leq n$} such that
\begin{equation*}
\Hom(B_i,B_j[m]) = 0
\quad\text{for all $B_i \in {^{\tau_i}}\cB_i^{\heartsuit}$, $B_j \in {^{\tau_j}}\cB_j^{\heartsuit}$, $i < j$, and~$m \leq 0$.}
\end{equation*}
Then there is a t-structure~${\tau_\cB}$ on~$\cB$, called the \emphsf{gluing of the~$\tau_i$}, defined by
\begin{equation}
\label{eq:t-gluing}
\begin{aligned}
{^{\tau_\cB}}\cB^{\leq 0} & = \set{ B \in \cB \sth \pr_i(B) \in {^{\tau_i}}\cB_i^{\leq 0} \textup{ for all } i} , \\
{^{\tau_\cB}}\cB^{\geq 0} & = \set{ B \in \cB \sth \pr_i(B) \in {^{\tau_i}}\cB_i^{\geq 0} \textup{ for all } i} .
\end{aligned}
\end{equation}
Moreover, we have the following t-exactness properties:
\begin{enumerate}[label={\textup{(\arabic*)}}]
\item
\label{it:inclusions-exact}
The inclusions $\cB_i \to \cB$ are t-exact with respect to $\tau_i$ and $\tau_\cB$.
\item
\label{phi-glued-t-structure-t-exact}
If~$\cD$ is a triangulated category with a t-structure~${\tau}$ and~${\beta} \colon \cB \to \cD$ is a triangulated functor,
then~$\beta$ is \textup(right or left\textup) t-exact with respect to~$\tau_\cB$ and~$\tau$
if and only if the compositions~$\cB_i \to \cB \xrightarrow{\beta} \cD$ are \textup(right or left\textup) t-exact
with respect to~$\tau_i$ and~$\tau$ for all~$i$.
\end{enumerate}
\end{lemma}

\begin{proof}
The description of the t-structure for~$n = 2$ is~\cite[Lemma 2.1]{collins2010gluing},
and the general case follows easily from this by induction.
It remains to prove properties~\ref{it:inclusions-exact} and~\ref{phi-glued-t-structure-t-exact}.

\ref{it:inclusions-exact}
The t-exactness of the inclusion~$\cB_i \to \cB$ is immediate from~\eqref{eq:t-gluing}.

\ref{phi-glued-t-structure-t-exact}
If~$\beta$ is (right or left) t-exact then~\ref{it:inclusions-exact} implies that so is the composition~$\cB_i \to \cD$.
Conversely, suppose for instance that the compositions~$\cB_i \to \cD$ are right t-exact for all~$i$.
For any~$B \in \cB$, its image~$\beta(B)$ admits a filtration with graded pieces the objects~$\beta(\pr_i(B))$.
If~$B \in {^{\tau_\cB}}\cB^{\leq 0}$ then~{$\pr_i(B) \in {}^{\tau_i}\cB_i^{\le 0}$ by the definition of~$\tau_\cB$. 
Then by the right t-exactness of~$\beta$ on each~$\cB_i$ we have~$\beta(\pr_i(B)) \in {^{\tau}}\cD^{\leq 0}$,
and hence~$\beta(B) \in {^{\tau}}\cD^{\leq 0}$}.
Therefore, $\beta$ is right t-exact.
The case where the compositions $\cB_i \to \cD$ are left t-exact is analogous.
\end{proof}

Given a triangulated subcategory $\cB \subset \cD$ which admits a semiorthogonal decomposition,
we obtain conditions under which a t-structure restricts to $\cB$ if it does so component-by-component.

\begin{lemma}
\label{lemma-t-structure-restrict-component-wise}
Let~$\cD$ be a triangulated category equipped with a t-structure~$\tau$.
Let~$\cB \subset \cD$ be a triangulated subcategory with a semiorthogonal decomposition~$\cB = \langle \cB_1, \dots, \cB_n \rangle$.
Assume that~$\tau$ 
restricts to a t-structure~$\tau_{i}  = \tau\vert_{\cB_i}$ on~$\cB_i$ for all~$1 \leq i \leq n$, and
\begin{equation*}
\Hom(B_i, B_j) = 0
\quad
\text{for all~$B_i \in {^{\tau_{i}}}\cB_{i}^{\heartsuit}$, $B_j \in {^{\tau_{j}}}\cB_{j}^{\heartsuit}$, and~$i < j$.}
\end{equation*}
Then~$\tau$ restricts to a t-structure on~$\cB$, which coincides with the gluing of the~$\tau_{i}$.
\end{lemma}

\begin{proof} 
It follows from~\eqref{eq:t-restricted-body} that~${}^{\tau_i}\cB_i^\heartsuit = \cB_i \cap {}^\tau\cD^\heartsuit$;
therefore, $\Hom(B_i,B_j[m]) = 0$ for all~\mbox{$B_i \in {^{\tau_{i}}}\cB_{i}^{\heartsuit}$},
\mbox{$B_j \in {^{\tau_{j}}}\cB_{j}^{\heartsuit}$}, $i < j$, and~$m < 0$,
because such~$B_i$ and $B_j$ are in particular contained in the heart~${^\tau}\cD^{\heartsuit}$.
Hence the assumptions of Lemma~\ref{lemma-t-structure-gluing} are indeed satisfied,
so that we can glue the t-structures~$\tau_i$ to get a t-structure~$\tau_{\cB}$ on~$\cB$.
Since each inclusion~$\cB_i \to \cD$ is t-exact with respect to~$\tau_i$ and~$\tau$,
by Lemma~\ref{lemma-t-structure-gluing}\ref{phi-glued-t-structure-t-exact}
the inclusion~$\cB \to \cD$ is also t-exact with respect to~$\tau_{\cB}$ and~$\tau$.
Finally, by Lemma~\ref{lemma-beta-exact}, $\tau$ restricts to the t-structure~$\tau_{\cB}  = \tau\vert_\cB$ on~$\cB$.
\end{proof} 

\begin{corollary}
Let~$\cD$ be a $\kk$-linear triangulated category for a field~$\kk$,
equipped with a t-structure~$\tau$.
Let~$E_1, \dots, E_n \in \cD^{\heartsuit}$ be an exceptional sequence of objects in the heart of~$\tau$ 
which satisfies the forward $\Hom$-vanishing condition:
\begin{equation}
\label{eq:hom-vanishing}
\Hom(E_i, E_j) = 0
\quad
\text{for all $i < j$}.
\end{equation}
Then $\tau$ restricts to a t-structure on the triangulated subcategory 
$\langle E_1, \dots, E_n \rangle \subset \cD$ generated by the exceptional sequence. 
\end{corollary} 

\begin{proof}
By Example~\ref{example-E-in-Dheart}, for each $1 \leq i \leq n$ the t-structure $\tau$ restricts to a t-structure $\tau_i$ 
on the triangulated subcategory~$\cB_i \coloneqq \langle E_i \rangle \subset \cD$ generated by $E_i$.
In view of the condition~\eqref{eq:hom-vanishing}, 
we may apply Lemma~\ref{lemma-t-structure-restrict-component-wise} to deduce the result. 
\end{proof} 

 \subsection{Existence via simple generators}
 \label{section-via-simple-generators}

Now we discuss the case where our subcategory is generated by a $\Hom$-orthogonal collection of $\Hom$-simple objects. 

\begin{proposition}
\label{prop:simple-heart}
Let~$\cD$ be a $\kk$-linear triangulated category for a field~$\kk$, equipped with a t-structure~$\tau$.
Let~$B_1,\dots,B_n\in {}^\tau\cD^\heartsuit$ be a finite set of objects in the heart of~$\tau$ such that
\begin{enumerate}[label={\textup{(\arabic*)}}]
\item
$\Hom(B_i,B_i) = \kk$ for all~$i$, and
\item
$\Hom(B_i,B_j) = 0$ for all~$i \ne j$.
\end{enumerate}
Then the full subcategory~$\cB \subset \cD$ consisting of objects~$B \in \cD$
such that every cohomology object~${}^\tau\cH^n(B)$ belongs to the subcategory
\begin{equation*}
\langle B_1, \dots, B_n \rangle_{\mathrm{ext}} \subset {}^\tau\cD^\heartsuit
\end{equation*}
formed by finite iterated extensions of the objects~$B_i$ is a thick triangulated subcategory of~$\cD$. 

Moreover, $\tau$ restricts to a t-structure on~$\cB$ with noetherian and artinian heart,
and to the subcategory
\begin{equation*}
{}^{\tau\vert_\cB}\cB^{\mathrm{b}} = {}^\tau\cD^{\mathrm{b}} \cap \cB
\end{equation*}
which is the triangulated subcategory of~${}^\tau\cD^{\mathrm{b}}$ generated by~$B_1,\dots,B_n$.
\end{proposition}

\begin{proof}
The first claim follows from the argument of~\cite[Lemma~3.6]{Pir25}.
Since the context of~\cite[Lemma~3.6]{Pir25} is somewhat different, we sketch the proof.

Let~$B',B'' \in \langle B_1, \dots, B_n \rangle_{\mathrm{ext}}$.
We check that the kernel, image, and cokernel of any morphism~$\phi \colon B' \to B''$
all belong to~$\langle B_1, \dots, B_n \rangle_{\mathrm{ext}}$.

First, assume~\mbox{$B' = B_i$}.
The object~$B''$ by assumption has a filtration with all factors isomorphic to~$B_j$ for some~$1 \le j \le n$.
If the composition~$B_i \xrightarrow{\ \phi\ } B'' \xrightarrow{\quad} B_j$ with the projection to the last factor of the filtration is nonzero,
then~$j = i$ and~$\phi$ is a split monomorphism, hence~$\ker(\phi) = 0$, $\im(\phi) = B_i$,
and~$\coker(\phi) \in \langle B_1, \dots, B_n \rangle_{\mathrm{ext}}$.
Otherwise~$\phi$ factors through~$B''' \coloneqq \ker(B'' \to B_j)$, which belongs to~$\langle B_1, \dots, B_n \rangle_{\mathrm{ext}}$.
By induction on the length of the filtration,
the kernel, image and cokernel of~$B_i \to B'''$ are in~$\langle B_1, \dots, B_n \rangle_{\mathrm{ext}}$,
and it follows easily that the same is true for~$\ker(\phi)$, $\im(\phi)$, and~$\coker(\phi)$.

Next, if~$B' \in \langle B_1, \dots, B_n \rangle_{\mathrm{ext}}$ is arbitrary, a simple induction
on the length of the filtration of~$B'$ with factors isomorphic to~$B_i$ proves the claim.

Now, it follows easily that the subcategory~$\cB \subset \cD$ consisting of objects~$B$
such that every cohomology~${}^\tau\cH^n(B)$ belongs to~$\langle B_1, \dots, B_n \rangle_{\mathrm{ext}}$
is a thick triangulated subcategory of~$\cD$.

The above description implies that~$\cB$ is stable under the truncation functors of~$\tau$,
and hence~$\tau$ restricts to a t-structure~$\tau_\cB$ on~$\cB$ (Lemma~\ref{lemma-restriction-t-structure}).
The above argument also shows that any nonzero morphism~$B_i \to {B}$ in~${}^{\tau_\cB}\cB^\heartsuit$ is a monomorphism,
hence~$B_i$ is a simple object in~${}^{\tau_\cB}\cB^\heartsuit$,
and therefore~${}^{\tau_\cB}\cB^\heartsuit$ is a category of finite length;
in particular, it is noetherian and artinian.

It remains to note that the category~${^{\tau \vert_{\cB}}}\cB^{\mathrm{b}}$ is generated by the objects~$B_1,\dots,B_n$. 
Indeed, every object~$B \in {^{\tau \vert_{\cB}}}\cB^{\mathrm{b}}$ is a finite iterated extension
of its (appropriately shifted) cohomology objects~${}^{\tau}\cH^n(B) \in \langle B_1, \dots, B_n \rangle_{\mathrm{ext}}$, and 
every object of~$\langle B_1, \dots, B_n \rangle_{\mathrm{ext}}$
is by definition a finite iterated extension of the objects $B_i$. 
\end{proof}

%%%%%%%%%%%%%%%%%%%%%%%%%%%%%%%%%%%%%%%

\section{Projected t-structures} 
\label{section-projected-t-structures}

In this section, we discuss projection of t-structures, which is a method for constructing a t-structure
on a right or left admissible subcategory from one on the ambient category.

In  \S\ref{ss:projected-basic} we define projected t-structures and discuss their basic properties. 
In~\S\ref{ss:projected-t-amplitude} we study t-amplitude properties of the projection functors for a semiorthogonal decomposition, and deduce a criterion for the existence of projected t-structures (Theorem~\ref{theorem-projected-t-structures}). 
In the auxiliary subsections~\S\ref{ss:projected-stability} and~\S\ref{section-complements-es}, 
we explain some special features of the projection of a t-structure underlying a stability condition, as well as the relation between our results and those in~\cite{BLMS}. 

\subsection{Definition and basic properties} 
\label{ss:projected-basic}

\begin{definition}
\label{definition-projection-t-structure}
Let~$\cD$ be a triangulated category with a t-structure~$\tau = ({}^\tau\cD^{\leq 0}, {}^\tau\cD^{\geq 0})$.
Let~$\cC$ be a triangulated category with a fully faithful triangulated functor~$\gamma \colon \cC \to \cD$.
We say that:
\begin{enumerate}[label={\textup{(\alph*)}}]
\item
$\tau$ \emphsf{right projects to a t-structure on~$\cC$} if the functor~$\gamma$ has a right adjoint~$\gamma^!$ and the pair
\begin{equation*}
\gamma^!\tau \coloneqq  (\gamma^!({}^\tau\cD^{\leq 0}), \gamma^!({}^\tau\cD^{\geq 0}))
\end{equation*}
defines a t-structure on $\cC$, in which case $\gamma^!\tau$ is called the \emphsf{right projected t-structure}.

\item
$\tau$ \emphsf{left projects to a t-structure on~$\cC$} if the functor~$\gamma$ has a left adjoint~$\gamma^*$ and the pair
\begin{equation*}
\gamma^*\tau \coloneqq  (\gamma^*({}^\tau\cD^{\leq 0}), \gamma^*({}^\tau\cD^{\geq 0}))
\end{equation*} 
defines a t-structure on $\cC$, in which case $\gamma^*\tau$ is called the \emphsf{left projected t-structure}.
\end{enumerate}
\end{definition}

The following lemma summarizes the basic properties of projected t-structures.

\begin{lemma}
\label{lemma-projection-tau-cB}
Let~$\cD$ be a triangulated category equipped with a t-structure~$\tau = ({}^\tau\cD^{\leq 0}, {}^\tau\cD^{\geq 0})$.
Let $\cC$ be a triangulated category with a fully faithful triangulated functor~$\gamma \colon \cC \to \cD$. 
Assume either: 
\begin{enumerate}[label={\textup{(\alph*)}}]
\item
\label{it:projection-right}
$\tau$ right projects to a t-structure on $\cC$; or 
\item 
\label{it:projection-left} 
$\tau$ left projects to a t-structure on $\cC$. 
\end{enumerate}
For uniformity of notation below, we let~$\gamma^{\bullet} \colon \cD \to \cC$
denote the right adjoint~$\gamma^!$ in case~\ref{it:projection-right}
or the left adjoint~$\gamma^*$ in case~\ref{it:projection-left}.

The projected t-structure $\gamma^{\bullet}\tau$ on $\cC$ has the following properties:
\begin{enumerate}[label={\textup{(\arabic*)}}]
\item
\label{beta*-t-exact}
$\gamma^\bullet \colon \cD \to \cC$ is t-exact with respect to~$\tau$ and~$\gamma^\bullet\tau$.

\item
\label{beta-left-t-exact}
$\gamma \colon \cC \to \cD$ is right t-exact with respect to $\gamma^{\bullet}\tau$ and $\tau$ in case~\ref{it:projection-right}, and left t-exact with respect to $\gamma^{\bullet}\tau$ and $\tau$  in case~\ref{it:projection-left}. 

\item
\label{beta*t-geq0}
 ${^{\gamma^\bullet\tau}}\cC^{\leq 0} =  {^{\tau}}\cD^{\leq 0} \cap \cC$ in case~\ref{it:projection-right}  
and~${^{\gamma^{\bullet}\tau}}\cC^{\geq 0} =  {^{\tau}}\cD^{\geq 0} \cap \cC$ in case~\ref{it:projection-left}. 

\item
\label{t-induction-cB-truncations}
For any $n \in \bZ$ the truncation and cohomology functors of $\gamma^\bullet\tau$ are given by
\begin{equation*}
(\gamma^\bullet\tau)^{\le n} \cong \gamma^\bullet \circ \tau^{\le n} \circ \gamma,
\qquad
(\gamma^\bullet\tau)^{\ge n} \cong \gamma^\bullet \circ \tau^{\ge n} \circ \gamma,
\qquad
{}^{\gamma^\bullet\tau}\cH^{n} \cong \gamma^\bullet \circ {}^{\tau}\cH^{n} \circ \gamma.
\end{equation*}

\item
\label{t-induction-cB-heart}
The heart of~$\gamma^{\bullet}\tau$ and the subcategories of bounded objects are given by
\begin{equation*}
{}^{\gamma^\bullet\tau}\cC^{\heartsuit} = \gamma^\bullet({}^\tau\cD^{\heartsuit})
\qquad\text{and}\qquad
{}^{\gamma^\bullet\tau}\cC^{?} = \gamma^\bullet({}^\tau\cD^{?})
\quad\text{with}\quad
? \in \set{+,-,\mathrm{b}}.
\end{equation*}

\item \label{t-induced-bounded-cB}
 If~$\tau$ is bounded below, or bounded above, or bounded,
then so is $\gamma^\bullet\tau$.
\item
\label{Phi!tau-noetherian}
If~$\tau$ is noetherian or artinian, then so is~$\gamma^\bullet\tau$.
\end{enumerate}
\end{lemma}

\begin{proof}
Property~\ref{beta*-t-exact} follows immediately from the definition of the projected t-structure.
Then property~\ref{beta-left-t-exact} follows from Lemma~\ref{lemma-t-exactness-adjoints}. 
Then the claims in~\ref{beta*t-geq0}--\ref{t-induced-bounded-cB}
follow immediately from the fact that~$\gamma^\bullet \circ \gamma \simeq \id_{\cC}$.

Finally, assume that~$\tau$ is noetherian.
By properties~\ref{beta*-t-exact} and~\ref{beta-left-t-exact} and Lemma~\ref{lemma-t-exactness-adjoints}
the functors
\begin{equation*}
{}^\tau\cH^0 \circ \gamma \colon {}^{\gamma^\bullet\tau}\cC^\heartsuit \to {}^\tau\cD^\heartsuit
\qquad\text{and}\qquad
{}^{\gamma^\bullet\tau}\cH^0 \circ \gamma^\bullet \cong \gamma^\bullet \colon {}^\tau\cD^\heartsuit \to {}^{\gamma^\bullet\tau}\cC^\heartsuit
\end{equation*}
form an adjoint pair, where the second functor is the right adjoint in case~\ref{it:projection-right}
and left adjoint in case~\ref{it:projection-left}.
Using also the last isomorphism in~\ref{t-induction-cB-truncations} 
we conclude that~${}^\tau\cH^0 \circ \gamma$ is fully faithful.

Now, assume we are in case~\ref{it:projection-right};
then the functor~${}^\tau\cH^0 \circ \gamma$ is right exact by property~\ref{beta-left-t-exact}.
Thus, if~$C_0 \twoheadrightarrow C_1 \twoheadrightarrow C_2 \twoheadrightarrow \cdots$
is a chain of epimorphisms in ${^{\gamma^!\tau}}\cC^{\heartsuit}$, then
\begin{equation*}
{}^\tau\cH^0(\gamma(C_0)) \twoheadrightarrow
{}^\tau\cH^0(\gamma(C_1)) \twoheadrightarrow
{}^\tau\cH^0(\gamma(C_2)) \twoheadrightarrow \cdots
\end{equation*}
is a chain of epimorphisms in~${^{\tau}}\cD^{\heartsuit}$.
The latter sequence must stabilize to isomorphisms because~$\tau$ is noetherian,
and then the former must also stabilize to isomorphisms
because the functor~${}^\tau\cH^0 \circ \gamma$ is fully faithful.

Similarly, assume we are in case~\ref{it:projection-left};
then the functor~${}^\tau\cH^0 \circ \gamma$ is left exact by property~\ref{beta-left-t-exact}.
Thus, if $C_0 \hookrightarrow C_1 \hookrightarrow \cdots \hookrightarrow C$ is a chain of subobjects of an object $C \in {^{\gamma^*\tau}}\cC^{\heartsuit}$, then
\begin{equation*}
{}^\tau\cH^0(\gamma(C_0)) \hookrightarrow
{}^\tau\cH^0(\gamma(C_1)) \hookrightarrow
{}^\tau\cH^0(\gamma(C_2)) \hookrightarrow
\cdots \hookrightarrow
{}^\tau\cH^0(\gamma(C))
\end{equation*}
is a chain of subobjects of ${}^\tau\cH^0(\gamma(C)) \in {^{\tau}}\cD^{\heartsuit}$.
The latter sequence must stabilize to isomorphisms because~$\tau$ is noetherian,
and then the former must also stabilize to isomorphisms
because the functor~${}^\tau\cH^0 \circ \gamma$ is fully faithful.

If~$\tau$ is artinian, the argument is analogous.
\end{proof}

\begin{remark}
\label{remark-opposite-category}
Any result for the right projection of~$\tau$ formally implies a ``dual'' result
for the left projection, by passing to the opposite categories.
Indeed, by Remark~\ref{remark-opposite-t-structure}, the opposite category $\cD^{\op}$ is equipped with the opposite t-structure $\tau^{\op}$.
Moreover, if~$\cC^\op$ is the opposite category of~$\cC$
and~$\gamma^\op \colon \cC^\op \to \cD^\op$ is the opposite functor of~$\gamma \colon \cC \to \cD$, 
then~$(\gamma^{\op})^* \simeq (\gamma^!)^{\op}$ and~$(\gamma^{\op})^! \simeq (\gamma^*)^{\op}$, hence
\begin{equation*}
(\gamma^{\op})^*(\tau^\op) = (\gamma^!\tau)^\op
\qquad\text{and}\qquad
(\gamma^{\op})^!(\tau^\op) = (\gamma^*\tau)^\op.
\end{equation*}

For instance, by this duality, the results in Lemma~\ref{lemma-projection-tau-cB} in case~\ref{it:projection-right}
are equivalent to those in case~\ref{it:projection-left}, except for the noetherian
or artinian properties in~\ref{Phi!tau-noetherian}, which are dual to one another.

Throughout the paper, there will be many results with right and left versions.
By the same principle as above, it suffices to prove only one of these versions.
\end{remark}

Now we can give a simple characterization of when $\tau$ projects to a t-structure. 

\begin{lemma}
\label{lemma-project-B-iff-lte}
Let~$\cD$ be a triangulated category equipped with a t-structure~$\tau = ({}^\tau\cD^{\leq 0}, {}^\tau\cD^{\geq 0})$. 
Let $\cC$ be a triangulated category with a fully faithful triangulated functor~$\gamma \colon \cC \to \cD$. 
\begin{enumerate}[label={\textup{(\alph*)}}]
\item
\label{it:projection-right-criterion}
The t-structure~$\tau$ right projects to a t-structure on $\cC$
if and only if the right adjoint~$\gamma^!$ exists and the composition~$\gamma \circ \gamma^!$ is right t-exact with respect to~$\tau$.

\item
\label{it:projection-left-criterion}
The t-structure~$\tau$ left projects to a t-structure on $\cC$
if and only if the left adjoint~$\gamma^*$ exists and the composition~$\gamma \circ \gamma^*$ is left t-exact with respect to~$\tau$.
\end{enumerate}
\end{lemma}

\begin{proof}
As in Lemma~\ref{lemma-projection-tau-cB}, we use notation~$\gamma^\bullet$ for the right or left adjoint functor of~$\gamma$, according to whether we are considering statement~\ref{it:projection-right-criterion} or~\ref{it:projection-left-criterion}.

If~$\tau$ right or left projects to a t-structure on~$\cC$,
then by parts~\ref{beta*-t-exact} and~\ref{beta-left-t-exact} of Lemma~\ref{lemma-projection-tau-cB}
we see that~$\gamma \circ \gamma^\bullet$ is right or left t-exact with respect to~$\tau$.

Conversely, assume that~$\gamma \circ \gamma^\bullet$ is right or left t-exact with respect to~$\tau$.
To show that~$\gamma^\bullet\tau$ is a t-structure, we directly verify the conditions~\ref{it:t-str-shift}--\ref{it:t-str-triangle}
in Definition~\ref{def:t-str} for the pair of subcategories~$(\gamma^\bullet({}^\tau\cD^{\le 0}), \gamma^\bullet({}^\tau\cD^{\ge 0}))$.
The condition~\ref{it:t-str-shift} on compatibility with shifts is clear. 
For the $\Hom$-orthogonality condition~\ref{it:t-str-perp}, 
let~\mbox{$D \in {}^\tau\cD^{\le 0}$} and~\mbox{$D' \in {}^\tau\cD^{\ge 1}$}.
Then by adjunction 
\begin{align*}
\Hom(\gamma^!(D), \gamma^!(D')) &= \Hom(\gamma\gamma^!(D), D'),
\\
\Hom(\gamma^*(D), \gamma^*(D')) &= \Hom(D, \gamma\gamma^*(D')),
\end{align*}
which vanishes
because~$\gamma\gamma^!(D) \in {}^\tau\cD^{\leq 0}$ by the right t-exactness of~$\gamma \circ \gamma^!$
and~$\gamma\gamma^*(D') \in {}^\tau\cD^{\geq 1}$ by the left t-exactness of~$\gamma \circ \gamma^*$, respectively.
For the generation condition~\ref{it:t-str-triangle},
let~$C \in \cC$ be an object and set~$D = \gamma(C)$, so that~\mbox{$C \cong \gamma^\bullet(D)$}.
Then applying the functor~$\gamma^\bullet$ to the truncation triangle~$\tau^{\leq 0}D \to D \to \tau^{\geq 1}D$ we obtain a distinguished triangle
\begin{equation*}
\gamma^\bullet\tau^{\leq 0}(D) \to C \to \gamma^\bullet\tau^{\geq 1}(D),
\end{equation*}
where the first term is in~$\gamma^\bullet({}^\tau\cD^{\le 0})$ and the last is in~$\gamma^\bullet({}^\tau\cD^{\ge {1}})$.
This verifies condition~\ref{it:t-str-triangle}. 
\end{proof}

\subsection{t-amplitude of projection functors and existence of projected t-structures}
\label{ss:projected-t-amplitude}

For this subsection, we work in the following setting.

\begin{setup}
\label{setup-sod}
Let~$\cD$ be a triangulated category equipped with a t-structure~$\tau = ({}^\tau\cD^{\leq 0}, {}^\tau\cD^{\geq 0})$
and a semiorthogonal decomposition~$\cD = \langle \cB, \cC \rangle$.
Let~$\beta \colon \cB \to \cD$ and~$\gamma \colon \cC \to \cD$ be the inclusion functors,
and~$\beta^* \colon \cD \to \cB$ and~$\gamma^! \colon \cD \to \cC$ their left and right adjoints. 
\end{setup}

Note that in this setup, the projection functors of the semiorthogonal decomposition of~$\cD$
are given by~$\pr_{\cB} = \beta \circ \beta^*$ and~$\pr_{\cC} = \gamma \circ \gamma^!$.

In Setup~\ref{setup-sod}, Lemma~\ref{lemma-project-B-iff-lte} shows 
that we can construct a left projected t-structure
on~$\cB$ when its projection functor~$\beta \circ \beta^*$ is left t-exact,
and a right projected t-structure on~$\cC$ when its projection functor~$\gamma \circ \gamma^!$ is right t-exact.}
Our next goal is to explain some useful criteria for verifying such t-exactness properties, as well as ``one off'' t-exactness
(i.e. left t-amplitude~\mbox{$\geq -1$} or right t-amplitude~\mbox{$\leq 1$} in the sense of  Definition~\ref{def:amplitude}).

First, we observe that the t-amplitudes of the projection functors $\beta \circ \beta^*$ and $\gamma \circ \gamma^!$ agree up to a shift by $1$: 

\begin{lemma}
\label{lemma-projection-functor-amplitudes} 
In Setup~\textup{\ref{setup-sod}},
consider the following conditions for $a \in \bZ$: 
\begin{enumerate}[label={\textup{(\roman*)}}]
\item \label{beta-lta}
The projection functor $\beta \circ \beta^*$ has left t-amplitude $\geq a$ with respect to $\tau$.  
\item \label{gamma-lta} 
The projection functor $\gamma \circ \gamma^!$ has left t-amplitude $\geq a +1$ with respect to $\tau$. 
\end{enumerate}
If $a \leq 0$ then we have the implication~\ref{gamma-lta} $\implies$ \ref{beta-lta}, 
and if $a \leq -1$ then we have the converse implication~\ref{beta-lta} $\implies$ \ref{gamma-lta}.

Similarly, consider the following conditions for $b \in \bZ$:  
\begin{enumerate}[label={\textup{(\roman*$'$)}}]
\item \label{gamma-rta} 
The projection functor $\gamma \circ \gamma^!$ has right t-amplitude $\leq b$ with respect to $\tau$. 
\item \label{beta-rta}
The projection functor $\beta \circ \beta^*$ has right t-amplitude $\leq b-1$ with respect to $\tau$. 
\end{enumerate} 
If $b \geq 0$ then we have the implication~\ref{beta-rta} $\implies$ \ref{gamma-rta}, 
and if $b \geq 1$ then we have the converse implication~\ref{gamma-rta} $\implies$ \ref{beta-rta}. 
\end{lemma}

\begin{proof}
For any object $D \in \cD$, we have the distinguished triangle 
\begin{equation*}
\gamma\gamma^!(D) \to D \to \beta\beta^*(D)
\end{equation*} 
where the first morphism is the counit of the adjunction and the second is the unit of the adjunction. 
All of the claims follow easily from this triangle. 
For instance, assume that $a \leq -1$ and~\ref{beta-lta} holds. 
If $D \in {^\tau}\cD^{\geq 0}$, then rotating the above distinguished triangle gives a distinguished triangle 
\begin{equation*}
\beta\beta^*(D)[-1] \to \gamma\gamma^!(D) \to D 
\end{equation*}
where the first term is in ${^\tau}\cD^{\geq a+1}$ and the third term is in ${^\tau}\cD^{\geq 0}$, which is contained in ${^\tau}\cD^{\geq a+1}$ because $0 \geq a+1$. 
It follows that $\gamma\gamma^!(D)$ is also contained in ${^\tau}\cD^{\geq a+1}$, and hence that~\ref{gamma-lta} holds. 
\end{proof} 

Now we can give a criterion for when the conditions in Lemma~\ref{lemma-projection-functor-amplitudes} hold for $a=-1$ or $b = 1$, in terms of the notion of restricted t-structures from Definition~\ref{definition-restriction-t-structure}. 

\begin{lemma}
\label{lemma-restrict-t-structure-vs-projection-amplitude}
In Setup~\textup{\ref{setup-sod}}, the following conditions are equivalent:
\begin{enumerate}[label={\textup{(\roman*)}}]
\item
\label{tau-restricts-cB}
The t-structure~$\tau$ restricts to a t-structure on~$\cB$.

\item
\label{projection-B-right-t-exact}
The projection functor~$\beta \circ \beta^*$ is right t-exact with respect to~$\tau$.

\item
\label{projection-C-t-amplitude-1}
The projection functor~$\gamma \circ \gamma^!$ has right t-amplitude~$\leq 1$ with respect to~$\tau$.
\end{enumerate}
Similarly, the following conditions are equivalent: 
\begin{enumerate}[label={\textup{(\roman*$'$)}}]
\item
\label{tau-restricts-cC}
The t-structure~$\tau$ restricts to a t-structure on~$\cC$.

\item
\label{projection-C-left-t-exact}
The projection functor~$\gamma \circ \gamma^!$ is left t-exact with respect to~$\tau$.

\item
\label{projection-C-t-amplitude--1}
The projection functor~$\beta \circ \beta^*$ has left t-amplitude~$\geq -1$ with respect to~$\tau$.
\end{enumerate} 
\end{lemma} 

\begin{proof}
\ref{tau-restricts-cB} $\implies$ \ref{projection-B-right-t-exact}: 
The functor~$\beta \colon \cB \to \cD$ is t-exact with respect to the restricted t-structure $\tau\vert_\cB$ and~$\tau$,
and hence the functor~$\beta^* \colon \cD \to \cB$ is right t-exact with respect to~$\tau$ and~$\tau\vert_\cB$ (Lemma~\ref{lemma-t-exactness-adjoints}). 
Therefore~$\beta \circ \beta^*$ is right t-exact with respect to~$\tau$.

\ref{projection-B-right-t-exact} $\implies$ \ref{projection-C-t-amplitude-1}: 
This holds by Lemma~\ref{lemma-projection-functor-amplitudes}. 

\ref{projection-C-t-amplitude-1} $\implies$ \ref{tau-restricts-cB}: 
For~$B \in \cB$ consider the truncation triangle
\begin{equation*}
\tau^{\le 0}{\beta(B)} \to \beta(B) \to \tau^{\ge 1}\beta(B).
\end{equation*}
To show that~$\tau$ restricts to $\cB$, it suffices by Lemma~\ref{lemma-restriction-t-structure}
to show that~$\tau^{\leq 0}\beta(B)$ and~$\tau^{\geq 1}\beta(B)$ are contained in $\cB$,
or equivalently that they are killed by~$\gamma^!$.
The functor~$\gamma^!$ kills~$\cB$, so applying it to the above truncation triangle gives an isomorphism
\begin{equation}
\label{gamma!tau0tau1}
\gamma^!(\tau^{\le 0}\beta(B)[1]) \cong \gamma^!(\tau^{\ge 1}\beta(B)).
\end{equation}
By adjunction, this isomorphism corresponds
to a morphism ${\varphi} \colon \gamma\gamma^!(\tau^{\le 0}\beta(B)[1]) \to \tau^{\ge 1}\beta(B)$.
The source of~$\varphi$ is contained in~${}^\tau\cD^{\le 0}$ because~$\gamma \circ \gamma^!$ has right t-amplitude~$\leq 1$,
while the target of~$\varphi$ is contained in~${}^\tau\cD^{\ge 1}$.
It follows that~$\varphi = 0$.
Thus the isomorphism~\eqref{gamma!tau0tau1} must also be the zero map,
which means that its source and target vanish, as required. 

The implications \ref{tau-restricts-cC} $\implies$ \ref{projection-C-left-t-exact} $\implies$ \ref{projection-C-t-amplitude--1} $\implies$ \ref{tau-restricts-cC}  follow similarly (or can be deduced formally from the first part, as in Remark~\ref{remark-opposite-category}). 
\end{proof}

The following motivating geometric example is useful to keep in mind. 
It is an elaboration on the first part of Example~\ref{example-perverse-t-structures} from the introduction,
except that here we do not assume~$X$ and~$Y$ are smooth
and, consequently, must replace the bounded categories~$\Db(X)$ and~$\Db(Y)$
by the bounded above categories~$\Dm(X)$ and~$\Dm(Y)$.
Recall that~$\tau_X$ and~$\tau_Y$ denote the standard t-structures of~$\Dm(X)$ and~$\Dm(Y)$  (see Example~\ref{ex:tau-x}). 

\begin{example}
\label{example-bridgeland}
Let~$f \colon X \to Y$ be a proper surjective morphism with fibers of dimension~$\leq 1$
between noetherian schemes such that~$f_* \cO_X \simeq \cO_Y$, where we recall that~$f_*$ denotes the derived pushforward.
By the projection formula we have~$f_*f^*F \simeq F$ for all~$F \in \Dm(Y)$. 
Hence the pullback~$f^* \colon \Dm(Y) \to \Dm(X)$ is fully faithful,
and there is a semiorthogonal decomposition
\begin{equation*}
\Dm(X) = \langle \cB, \cC \rangle,
\end{equation*}
where
\begin{equation*}
\cC = f^*\Dm(Y)
\qquad\text{and}\qquad
\cB = (f^*\Dm(Y))^{\perp} = \set{ B \in \Dm(X) \sth f_* B \simeq 0 }.
\end{equation*}
{Moreover, if~$F \in {}^{\tau_X}\Dm(X)^{\le 0}$ then~$f_*F \in {}^{\tau_Y}\Dm(Y)^{\le 1}$
(because the fibers of~$f$ have dimension~\mbox{$\leq 1$})
and therefore~$f^*f_*F \in {}^{\tau_X}\Dm(X)^{\le 1}$ (because~$f^*$ is right exact), so}
it follows that the composition~$f^* \circ f_* \colon \Dm(X) \to \Dm(X)$
has right t-amplitude $\leq 1$ with respect to the standard t-structure~$\tau_X$.
By Lemma~\ref{lemma-restrict-t-structure-vs-projection-amplitude}, the t-structure~$\tau_X$ restricts to a t-structure on~$\cB$; 
this was also observed much earlier by Bridgeland~\cite[Lemma 3.1]{bridgeland-flops}. 
\end{example}

There is a refinement of Lemma~\ref{lemma-restrict-t-structure-vs-projection-amplitude} 
in which the heart of the restricted t-structure is required to be a Serre subcategory of the ambient heart. 
Recall that a \emphsf{Serre subcategory} of an abelian category $\cA$
is a subcategory which is closed under taking subobjects, quotient objects, and extensions in $\cA$.

\begin{lemma}
\label{lemma-restrict-t-structure-ssc}
In Setup~\textup{\ref{setup-sod}}, the following conditions are equivalent:
\begin{enumerate}[label={\textup{(\roman*)}}]
\item
\label{tau-restricts-cB-ssc}
$\tau$ restricts to a t-structure~$\tau\vert_{\cB}$ on~$\cB$
whose heart~${^{\tau\vert_{\cB}}}\cB^{\heartsuit}$ is a Serre subcategory of~${^{\tau}}\cD^{\heartsuit}$.
\item
\label{tau-restricts-cB-cus}
$\tau$ restricts to a t-structure~$\tau\vert_{\cB}$ on~$\cB$
whose heart is closed under subobjects in~${^{\tau}}\cD^{\heartsuit}$.
\item
\label{tau-restricts-cB-cuq}
$\tau$ restricts to a t-structure~$\tau\vert_{\cB}$ on~$\cB$
whose heart is closed under quotients in~${^{\tau}}\cD^{\heartsuit}$.
\item
\label{projection-C-t-amplitude-0}
The projection functor~$\gamma \circ \gamma^!$ is right t-exact with respect to~$\tau$.
\end{enumerate}
Similarly, the following conditions are equivalent: 
\begin{enumerate}[label={\textup{(\roman*$'$)}}]
\item
\label{tau-restricts-cC-ssc}
$\tau$ restricts to a t-structure~$\tau\vert_{\cC}$ on~$\cC$
whose heart~${^{\tau\vert_{\cC}}}\cC^{\heartsuit}$ is a Serre subcategory of~${^{\tau}}\cD^{\heartsuit}$.
\item
\label{tau-restricts-cC-cus}
$\tau$ restricts to a t-structure~$\tau\vert_{\cC}$ on~$\cC$
whose heart is closed under subobjects in~${^{\tau}}\cD^{\heartsuit}$.
\item
\label{tau-restricts-cC-cuq}
$\tau$ restricts to a t-structure~$\tau\vert_{\cC}$ on~$\cC$
whose heart is closed under quotients in~${^{\tau}}\cD^{\heartsuit}$.
\item
\label{projection-B-t-amplitude-0}
The projection functor~$\beta \circ \beta^*$ is left t-exact with respect to~$\tau$.
\end{enumerate} 
\end{lemma} 

\begin{proof}
\ref{tau-restricts-cB-ssc} $\iff$ \ref{tau-restricts-cB-cus} $\iff$ \ref{tau-restricts-cB-cuq}:  
Note that a short exact sequence in~${^{\tau}}\cD^{\heartsuit}$
is the same as a distinguished triangle in~$\cD$ whose terms lie in~${^{\tau}}\cD^{\heartsuit}$,
and similarly for short exact sequences in~${^{\tau\vert_{\cB}}}\cB^{\heartsuit}$.
Since~$\cB \subset \cD$ is a triangulated subcategory,
it follows that~${^{\tau\vert_{\cB}}}\cB^{\heartsuit}$ is automatically closed
under extensions in~${^{\tau}}\cD^{\heartsuit}$,
and it is closed under subobjects if and only if it is closed under quotients. 

\ref{tau-restricts-cB-cuq} $\implies$ \ref{projection-C-t-amplitude-0}: 
By Lemma~\ref{lemma-restrict-t-structure-vs-projection-amplitude}
we already know that~$\gamma \circ \gamma^!$ has right t-amplitude~$\leq 1$ with respect to~$\tau$.
Thus it suffices to show that for any~$D \in {^\tau}\cD^{\leq 0}$ we have~${^\tau}\cH^1(\gamma \gamma^!(D)) = 0$.
Taking the long exact sequence ~\eqref{eq:les} of cohomology objects associated to the distinguished triangle
\begin{equation*}
\gamma\gamma^!(D) \to D \to \beta\beta^*(D) 
\end{equation*} 
gives a surjection~${^{\tau}}\cH^0(\beta\beta^*(D)) \twoheadrightarrow {^\tau}\cH^1(\gamma\gamma^!(D))$. 
Since~${^{\tau}}\cH^0(\beta\beta^*(D))$ is contained in~${^{\tau\vert_{\cB}}}\cB^{\heartsuit}$
and this category is closed under quotients in~${^\tau}\cD^{\heartsuit}$,
we conclude that~\mbox{${^\tau}\cH^1(\gamma\gamma^!(D)) \in {^{\tau\vert_{\cB}}}\cB^{\heartsuit}$}.
It follows that the canonical morphism 
\begin{equation*}
\gamma\gamma^!(D) \to \tau^{\geq 1}(\gamma \gamma^!(D))  \cong  {^\tau}\cH^1(\gamma\gamma^!(D))[-1]
\end{equation*} 
vanishes, because its source is in~$\cC$ and its target is in~$\cB$.
Thus~${^\tau}\cH^1(\gamma\gamma^!(D)) = 0$, as required.

\ref{projection-C-t-amplitude-0} $\implies$ \ref{tau-restricts-cB-cuq}: 
Let~$0 \to D' \to B \to D'' \to 0$ be an exact sequence in~${^\tau}\cD^{\heartsuit}$
with~\mbox{$B \in {^{\tau\vert_{\cB}}}\cB^{\heartsuit}$}.
We must show that~$D'' \in \cB$, or equivalently that~$\gamma^!(D'') = 0$;
we argue similarly to the proof of the implication~\ref{projection-C-t-amplitude-1} $\implies$ \ref{tau-restricts-cB}
in Lemma~\ref{lemma-restrict-t-structure-vs-projection-amplitude}.
Since~$\gamma^!$ kills the object~$B$, applying it to the exact sequence gives an isomorphism
\begin{equation}
\label{gammaD'1gammaD''}
\gamma^!(D'[1]) \cong \gamma^!(D''). 
\end{equation} 
By adjunction, this isomorphism corresponds to a morphism~${\varphi} \colon \gamma \gamma^!(D'[1]) \to D''$.
The source of~$\varphi$ is contained in~${^{\tau}}\cD^{\leq -1}$ because~$\gamma \circ \gamma^!$ is right t-exact,
while the target is contained in~${^\tau}\cD^{\heartsuit}$.
It follows that~$\varphi = 0$.
Thus the isomorphism~\eqref{gammaD'1gammaD''} must be the zero map, which means that its source and target vanish, as required. 

The equivalence of the statements \ref{tau-restricts-cC-ssc}--\ref{projection-B-t-amplitude-0} follows similarly.
\end{proof} 

Combining the above ingredients, we obtain a useful criterion for the existence of projected t-structures.
\begin{theorem}
\label{theorem-projected-t-structures}
In Setup~\textup{\ref{setup-sod}}, the following conditions are equivalent:
\begin{enumerate}[label={\textup{(\roman*)}}]
\item
\label{projected-t-structures-C-1}
$\tau$ restricts to a t-structure~$\tau\vert_{\cB}$ on~$\cB$
whose heart~${^{\tau\vert_{\cB}}}\cB^{\heartsuit}$ is {a Serre subcategory of}~${^{\tau}}\cD^{\heartsuit}$.
\item
$\gamma \circ \gamma^!$ is right t-exact with respect to~$\tau$.
\item
\label{projected-t-structures-C-2}
$\tau$ {right} projects to a t-structure~$\gamma^!\tau$ on~$\cC$.
\end{enumerate}
Moreover, in this case, if $\tau$ is bounded or noetherian or artinian, then so is $\gamma^!\tau$.

Similarly, the following conditions are equivalent: 
\begin{enumerate}[label={\textup{(\roman*$'$)}}]
\item
\label{projected-t-structures-B-1}
$\tau$ restricts to a t-structure~$\tau\vert_{\cC}$ on~$\cC$
whose heart~${^{\tau\vert_{\cC}}}\cC^{\heartsuit}$ is {a Serre subcategory of}~${^{\tau}}\cD^{\heartsuit}$.
\item
$\beta \circ \beta^*$ is left t-exact with respect to~$\tau$.
\item
\label{projected-t-structures-B-2}
$\tau$ {left} projects to a t-structure~$\beta^*\tau$ on~$\cB$.
\end{enumerate}
Moreover, in this case, if~$\tau$ is bounded or noetherian or artinian, then so is~$\beta^*\tau$.
\end{theorem} 

\begin{proof}
All of the claims follow by combining Lemmas~\ref{lemma-restrict-t-structure-ssc},~\ref{lemma-project-B-iff-lte}, and~\ref{lemma-projection-tau-cB}.
\end{proof}

\subsection{Projected t-structures from stability conditions}
\label{ss:projected-stability}

There is an interesting application of the criterion of Theorem~\ref{theorem-projected-t-structures}
in the context of stability conditions.
Before stating this result, we briefly recall a few relevant facts, referring to \cite[\S12]{stability-families} for further details. 

Given a homomorphism~$v \colon \rK_0(\cD) \to \Lambda$ from the Grothendieck group of~$\cD$
to a finite rank free abelian group, \emphsf{a pre-stability condition on~$\cD$ with respect to~$v$}
consists of a pair~$\sigma = (\tau,Z)$ where~$\tau$ is a bounded t-structure
and~$Z \colon \Lambda \to \bC$ is a group homomorphism satisfying suitable compatibility properties.
A pre-stability condition gives rise to a notion of (semi)stable objects, which are ordered by their \emphsf{phase}~$\phi \in \bR$.
The subcategories~$\cP_{\sigma}(\phi) \subset \cD$ of $\sigma$-semistable objects of phase~$\phi$
for varying~$\phi \in \bR$ form \emphsf{a slicing} of~$\cD$;
in these terms, we can describe the heart~${^\tau}\cD^{\heartsuit} = \cP_{\sigma}(0,1]$
as the extension closure of the subcategories~$\cP_{\sigma}(\phi)$ for~$\phi \in (0,1]$.

\emphsf{A stability condition} is a pre-stability condition satisfying the so-called support property.

\begin{proposition}
\label{proposition-pre-stability-tau} 
Let~$\cD$ be a $\kk$-linear triangulated category for a field~$\kk$.
Let~$\sigma = (\tau, Z)$ be a pre-stability condition on~$\cD$.
Let~$E \in \cD$ be an exceptional object which is $\sigma$-stable of phase~$1$, let
\begin{equation*}
\cD = \langle \cB, E \rangle \quad \text{and} \quad \cD = \langle E, \cC \rangle 
\end{equation*} 
be the corresponding semiorthogonal decompositions,
and let~$\beta \colon \cB \to \cD$ and~$\gamma \colon \cC \to \cD$ be the inclusion functors.
Then~$\tau$ left and right projects to bounded t-structures~$\beta^*\tau$ and~$\gamma^!\tau$ on~$\cB$ and~$\cC$,
which are noetherian or artinian if~$\tau$ is noetherian or artinian.
\end{proposition} 

\begin{proof}
By Example~\ref{example-E-in-Dheart}, 
the t-structure~$\tau$ restricts to a t-structure on the triangulated
subcategory~$\langle E \rangle \subset \cD$ generated by~$E$,
with heart~$\langle E \rangle^{\heartsuit} = \set{E^{\oplus n} \sth n \geq 0}$.
We claim that~$\langle E \rangle^{\heartsuit}$ is closed under quotients in~${^\tau}\cD^{\heartsuit}$.
Indeed, since~$E$ is $\sigma$-stable of phase~$1$, if~$E^{\oplus n} \to D$ is a surjection in~${^\tau}\cD^{\heartsuit}$,
then it follows that~$D \cong E^{\oplus m}$ for some~$0 \leq m \leq n$; in particular, $D \in \langle E \rangle^\heartsuit$.
Now the result follows from Lemma~\ref{lemma-restrict-t-structure-ssc} and Theorem~\ref{theorem-projected-t-structures}.
\end{proof} 

\begin{remark}
Proposition~\ref{proposition-pre-stability-tau} can be upgraded to a criterion for projecting stability conditions on~$\cD$
to its semiorthogonal components~$\cB$ and~$\cC$.
In a sequel to this paper, we will develop this and other extensions
of results in this paper to the case of stability conditions.
\end{remark} 

If we assume that~$\sigma$ is a stability condition, then the existence of a noetherian projected t-structure on $\cC$ 
follows without assuming the t-structure underlying~$\sigma$ itself is noetherian:

\begin{corollary}
\label{corollary-stability-noetherian-t-structure}
Let~$\cD$ be a $\kk$-linear triangulated category for a field~$\kk$.
Let~$\sigma = (\tau, Z)$ be a stability condition on~$\cD$.
Let~$E \in \cD$ be an exceptional object which is $\sigma$-stable, and let 
\begin{equation*}
\cD = \langle \cB, E \rangle \quad \text{and} \quad \cD = \langle E, \cC \rangle 
\end{equation*} 
be the corresponding semiorthogonal decompositions. 
Then~$\cB$ and~$\cC$ admit noetherian bounded t-strucures and artinian bounded t-structures.
\end{corollary} 

\begin{proof}
First, we prove that~$\cC$ admits a noetherian bounded t-structure.
We reduce to Proposition~\ref{proposition-pre-stability-tau}. 
Namely, suppose that~$E$ has phase~$\phi$.
By rotating, we may find a new stability condition~$\sigma' = (\tau', Z')$
for which~$E \in \cP_{\sigma'}(1)$ is stable of phase~$1$.
Then, by Bridgeland's deformation theorem and openness of stability of~$E$,
up to a small deformation of the central charge~$Z'$ we may assume that its image is contained in~$\bQ \oplus i \bQ$.
In this case, the t-structure~$\tau'$ is noetherian by~\cite[Proposition~5.0.1]{AP06},
so we may apply Proposition~\ref{proposition-pre-stability-tau} to conclude.

Since mutation through the object~$E$ gives an equivalence~$\cB \simeq \cC$,
it follows that $\cB$ also admits a noetherian bounded t-structure.

To deduce the artinian version of the result, 
we first observe that~$\sigma$ gives a stability condition on the opposite category~$\cD^{\op}$. 
More precisely, assume that $\sigma$ is a stability condition with respect to~$v \colon \rK_0(\cD) \to \Lambda$. 
Let $v^{\op} \colon \rK_0(\cD^{\op}) \to \Lambda$ be the composition of $v$ with the natural identification $\rK_0(\cD^{\op}) \cong \rK_0(\cD)$, 
let $\tau^{\op}$ be the opposite t-structure as in Remark~\ref{remark-opposite-t-structure}, 
and let $Z^{\op} \colon \Lambda \to \bC$ be the conjugate of $Z$, i.e. $Z^\op(v) = \overline{Z(v)}$ for $v \in \Lambda$.
Then it is easy to see that the the pair~$\sigma^{\op} = (\tau^{\op}, Z^{\op})$ is a stability condition on~$\cD^{\op}$ with respect to~$v^{\op}$. 
Moreover, the object~$E \in \cD^{\op}$ is still exceptional and~$\sigma^{\op}$-stable, and there  are semiorthogonal decompositions
\begin{equation*}
\cD^{\op} = \langle E, \cB^{\op} \rangle \quad \text{and} \quad 
\cD^{\op} = \langle \cC^{\op}, E \rangle. 
\end{equation*} 
By what we have already shown, $\cB^{\op}$ and $\cC^{\op}$ admit noetherian t-structures. By Remark~\ref{remark-opposite-t-structure}, the corresponding opposite t-structures on $\cB$ and $\cC$ are artinian.
\end{proof}

\subsection{Restricting t-structures to complements of exceptional sequences} 
\label{section-complements-es}

Finally, we discuss an application of Lemma~\ref{lemma-restrict-t-structure-vs-projection-amplitude}
to restricting t-structures to a semiorthogonal component~$\cB \subset \cD$.
In contrast to the results in \S\ref{section-restricting-t-structures},
the following criterion is given in terms of the complement of~$\cB$ (as opposed to properties of~$\cB$ itself).
For the result, we need our categories to be enhanced,
where an enhanced triangulated category can either be taken to mean a DG category or stable $\infty$-category.

\begin{lemma}
\label{lemma:restriction-criterion}
Let~$\cD$ be a $\kk$-linear enhanced triangulated category for a field~$\kk$,
equipped with a semiorthogonal decomposition~$\cD = \langle \cB, \cC \rangle$.
Let~$\tau$ be a t-structure on~$\cD$, and assume~$\cC$ is generated by an exceptional collection~$F_1,\dots,F_n$
such that~$F_i \in {}^\tau\cD^\heartsuit$ and
\begin{equation}
\label{eq:rhom-le1}
\RHom(F_i,D) \in \Db(\kk)^{\le 1}
\quad
\text{for all~$D \in {}^\tau\cD^{\le 0}$},
\end{equation}
where~$\Db(\kk)$ is equipped with the standard t-structure.
Then~$\tau$ restricts to a t-structure on~$\cB$.
\end{lemma}

\begin{proof}
Note that for all~$i, j$ we have~$\Ext^p(F_i,F_j) = 0$ for~$p \ne 0,1$:
indeed, for~$p < 0$ this follows from~$F_i,F_j \in {}^\tau\cD^\heartsuit$
and for~$p > 1$ this holds by assumption~\eqref{eq:rhom-le1}.
Moreover, for~$p \in \{0,1\}$ the space~$\Ext^p(F_i,F_j)$ is finite-dimensional
because~$F_i$ are homologically finite-dimensional by Definition~\ref{def:exceptional}.
Let~$\ol{F}_i$ be the universal extension of~$F_i$ by~$F_{i+1}, \dots, F_n$
as in~\cite[Definition~4.1]{HP}.
Then by~\cite[Theorem~4.4]{HP} the direct sum~$\ol{F} \coloneqq \oplus \ol{F}_i$
is a homologically finite-dimensional tilting generator for~$\cC$.
In particular (due to our assumption that $\cD$, and hence also $\cC$, admits an enhancement),
there is an equivalence
\begin{equation}
\label{eq:cc-lambda}
{\cC \simeq \Db(\Lambda\text{-}\mathrm{mod}),
\qquad
C \mapsto \RHom(\ol{F}, C),}
\end{equation}
where~$\Lambda \coloneqq \Hom(\ol{F}, \ol{F})$ is a finite-dimensional algebra
and~$\Lambda\text{-}\mathrm{mod}$ is the category of finitely generated right $\Lambda$-modules.
The standard t-structure on~$\Db(\Lambda\text{-}\mathrm{mod})$
corresponds under this equivalence to the t-structure~$\tau_\cC$ on~$\cC$ defined by
\begin{align*}
{}^{\tau_\cC}\cC^{\le 0} & = \set{ C \in \cC \sth {\RHom(\ol{F}, C)} \in \Db(\kk)^{\le 0} } , \\
\quad
{}^{\tau_\cC}\cC^{\ge 0} & = \set{ C \in \cC \sth {\RHom(\ol{F}, C)} \in \Db(\kk)^{\ge 0} },
\end{align*}
where in both cases~$\Db(\kk)$ is equipped with the standard t-structure.

Let~$\gamma \colon \cC \to \cD$ denote the inclusion functor with right adjoint~$\gamma^!$.
We will show that the endofunctor~$\gamma \circ \gamma^! \colon \cD \to \cD$ has right t-amplitude~$\leq 1$ with respect to~$\tau$, from which the result then follows by Lemma~\ref{lemma-restrict-t-structure-vs-projection-amplitude}.

To see this, first note that for any object $D \in \cD$, we have
\begin{equation*}
\RHom(\ol{F}, \gamma^!(D)) \cong \RHom(\ol{F}, D)
\end{equation*}
by adjunction. 
If~$D \in {}^\tau\cD^{\le 0}$, we have~$\RHom(F_i, D) \in \Db(\kk)^{\le 1}$ for all~$i$ by assumption~\eqref{eq:rhom-le1};
then, since~$\ol{F}$ is an iterated extension of the~$F_i$,
we conclude that~$\RHom(\ol{F}, D) \in \Db(\kk)^{\le 1}$ as well.
By the description of the t-structure~$\tau_{\cC}$ above, we find that~$\gamma^!(D) \in {}^{\tau_\cC}\cC^{\le 1}$. 
In other words, $\gamma^!$ has right t-amplitude~$\leq 1$ with respect to~$\tau$ and~$\tau_{\cC}$.

Similarly, if~$D \in {}^\tau\cD^{\ge 0}$, we have~$\RHom(F_i, D) \in \Db(\kk)^{\ge 0}$ for all~$i$ because~$F_i \in {}^\tau\cD^\heartsuit$, and 
hence~\mbox{$\gamma^!(D) \in {}^{\tau_\cC}\cC^{\ge 0}$}. 
Thus, $\gamma^!$ is left t-exact with respect to $\tau$ and $\tau_{\cC}$,
which by Lemma~\ref{lemma-t-exactness-adjoints} implies that~$\gamma$
is right t-exact with respect to~$\tau_{\cC}$ and~$\tau$.

Since $\gamma^!$ has right t-amplitude $\leq 1$ and $\gamma$ is right t-exact, their composition 
$\gamma \circ \gamma^!$ has right t-amplitude $\leq 1$, as required. 
\end{proof}

\begin{remark}
\label{remark-BLMS}
Lemma~\ref{lemma:restriction-criterion} gives a different perspective
on the construction of restricted t-structures from~\cite[\S4]{BLMS}.
We note that the hypotheses of our Lemma~\ref{lemma:restriction-criterion} are stronger
than those in~\cite[Lemma~4.3]{BLMS}, and our proof is somewhat less direct.
However, we included Lemma~\ref{lemma:restriction-criterion} as it shows how~\cite[\S4]{BLMS}
fits into the framework of this paper.
\end{remark}

\begin{example}
\label{example-BLMS}
Let~$\cD$ be a $\kk$-linear enhanced triangulated category for a field~$\kk$,
and assume~$\cD$ has a Serre functor~$\rS_{\cD}$.
Let~$\cD = \langle \cB, \cC \rangle$ be a semiorthogonal decomposition.
Let~$\tau$ be a t-structure on $\cD$, and assume~$\cC$ is generated by an exceptional collection~$F_1,\dots,F_n$
such that~$\rS_\cD(F_i) \in {}^\tau\cD^{\ge -1}$.
Then~$\rS_\cD(F_i)[-1] \in {}^\tau\cD^{\ge 0}$, and using Serre duality we find that
\begin{equation*}
\Ext^p(F_i,D)^\vee \cong \Ext^{-p}(D, \rS_\cD(F_i)) \cong \Ext^{1-p}(D, \rS_\cD(F_i)[-1])
\end{equation*}
vanishes for all~$D \in {}^\tau\cD^{{\le 0}}$ and all~$p$ such that~$1 - p < 0$, i.e., for~$p > 1$.
Therefore, the hypothesis of Lemma~\ref{lemma:restriction-criterion} is satisfied, and hence~$\tau$ restricts to~$\cB$.
This recovers (a version of)~\cite[Corollary~4.4]{BLMS}.
\end{example}

%%%%%%%%%%%%%%%%%%%%%%%%%%%%%%%%%%%%%%%

\section{Induced t-structures} 
\label{section-induced-t-structures} 

In this section, we discuss induced t-structures,
our third and most general method for constructing a t-structure on a triangulated subcategory from one on the ambient category.

In~\S\ref{ss:induced-def} we define induced t-structures and discuss their basic properties. 
In~\S\ref{subsection-reduction-t-amplitude} we develop a technique for reducing the t-amplitude
of an idempotent endofunctor by passing to a new t-structure.
In~\S\ref{ss:proof} we prove our main result, Theorem~\ref{main-theorem},
as well as Theorem~\ref{theorem-induce-via-exceptional-sequence} and other complements.
In the auxiliary subsections~\S\ref{subsection-t-exactability} and~\S\ref{section-perverse-t-structures}, 
we explain an approach to extending our results to much wider generality, 
as well as a construction of perverse t-structures for other perversities
\`{a} la Bridgeland in the context of Theorem~\ref{main-theorem}.

\subsection{Definition and basic properties} 
\label{ss:induced-def}

\begin{definition}
Let~$\cD$ be a triangulated category with a t-structure~$\tau = ({}^\tau\cD^{\leq 0}, {}^\tau\cD^{\geq 0})$.
Let~$\cC$ be a triangulated category with a fully faithful triangulated functor $\gamma \colon \cC \to \cD$.
We say that:
\begin{enumerate}[label={\textup{(\alph*)}}]
\item
$\tau$ \emphsf{connectively induces a t-structure} on~$\cC$
if the pair~$\cin{\cC} = (\cC^{\leq 0}, \cC^{\geq 0})$ where
\begin{align}
\label{induced-Cleq0}
\cC^{\leq 0} & = {^\tau}\cD^{\le 0} \cap \cC,
\\
\label{induced-Cgeq0}
\cC^{\geq 0} & =
\set{ C \in \cC \sth \Hom(C',C) = 0 \text{ \textup{for all} } C' \in \cC^{\leq 0}[1]},
\end{align}
defines a t-structure on~$\cC$, in which case~$\cin{\cC}$ is called the \emphsf{connectively induced t-structure} on~$\cC$.

\item
$\tau$ \emphsf{coconnectively induces a t-structure} on~$\cC$
if the pair~$\ccin{\cC} = (\cC^{\leq 0}, \cC^{\geq 0})$ where
\begin{align}
\label{induced-Bgeq0}
\cC^{\geq 0} & = {^\tau}\cD^{\ge 0} \cap \cC,
\\
\label{induced-Bleq0}
\cC^{\leq 0} & =
\set{ C \in \cC \sth \Hom({C,C'}) = 0 \text{ \textup{for all} } C' \in \cC^{\geq 0}[-1]},
\end{align}
defines a t-structure on~$\cC$,
in which case~$\ccin{\cC}$ is called the \emphsf{coconnectively induced t-structure} on~$\cC$.
\end{enumerate}
\end{definition}

The following lemma summarizes the basic properties of induced t-structures.
\begin{lemma}
\label{lemma-tauB-basic-properties}
Let~$\cD$ be a triangulated category with a t-structure~$\tau$.  
Let~$\cC$ be a triangulated category with a fully faithful triangulated functor $\gamma \colon \cC \to \cD$. 
Assume either: 
\begin{enumerate}[label={\textup{(\alph*)}}]
\item
\label{it:induction-right}
$\tau$ connectively induces a t-structure~$\cin{\cC}$ on $\cC$; or
\item
\label{it:induction-left}
$\tau$ coconnectively induces a t-structure~$\ccin{\cC}$ on $\cC$.
\end{enumerate}
Then the induced t-structure has the following properties: 
\begin{enumerate}[label={\textup{(\arabic*)}}]
\item
\label{it:induced-gamma}
$\gamma \colon \cC \to \cD$ is right t-exact with respect to $\cin{\cC}$ and $\tau$ in case~\ref{it:induction-right}, 
and left t-exact with respect to~$\ccin{\cC}$ and~$\tau$ in case~\ref{it:induction-left}.

\item
\label{it:induced-gamma-bullet} 
In case~\ref{it:induction-right}, if $\gamma$ admits a right adjoint $\gamma^! \colon \cD \to \cC$, then $\gamma^!$ is left t-exact with respect to $\tau$ and $\cin{\cC}$. 
In case~\ref{it:induction-left}, if $\gamma$ admits a left adjoint $\gamma^* \colon \cD \to \cC$, then 
$\gamma^*$ is right t-exact with respect to~$\tau$ and~$\ccin{\cC}$. 

\item
\label{it:induced-amplitude}
In case~\ref{it:induction-right}, if $\gamma$ admits a right adjoint $\gamma^! \colon \cD \to \cC$ and~$\gamma \circ \gamma^!$ has right t-amplitude~$\leq b$ with respect to~$\tau$ for an integer~$b$,
then~$\gamma^!$ has t-amplitude in~$[0,b]$ with respect to~$\tau$ and~$\cin{\cC}$. 
In case~\ref{it:induction-left},
if $\gamma$ admits a left adjoint~$\gamma^*$ and~$\gamma \circ \gamma^*$ has left t-amplitude~$\geq a$ with respect to~$\tau$ for an integer~$a$,
then~$\gamma^*$ has t-amplitude in~$[a,0]$ with respect to~$\tau$ and~$\ccin{\cC}$.
\item
\label{it:induced-bounded}
In case~\ref{it:induction-right}, if~$\tau$ is bounded above, then so is~$\cin{\cC}$.
In case~\ref{it:induction-left}, if~$\tau$ is bounded below, then so is~$\ccin{\cC}$.
\end{enumerate}
\end{lemma} 

\begin{proof}
The property~\ref{it:induced-gamma} holds by the definition of~$\tau_{\cC}$.
The property~\ref{it:induced-gamma-bullet} then follows from Lemma~\ref{lemma-t-exactness-adjoints}.
The property~\ref{it:induced-bounded} also holds by the definition of~$\tau_{\cC}$.

Finally, we prove~\ref{it:induced-amplitude} in case~\ref{it:induction-right}, case~\ref{it:induction-left} being analogous.
So, assume that $\gamma \circ \gamma^!$ has right t-amplitude~$\leq b$ with respect to~$\tau$.
As we already know that~$\gamma^!$ is left t-exact with respect to~$\tau$ and~$\cin{\cC}$,
it suffices to show that $\gamma^!$ has right t-amplitude~$\leq b$.
If~$D \in {^\tau}\cD^{\leq 0}$, then by our assumption~$\gamma\gamma^!(D) \in {^\tau}\cD^{\leq b}$,
which by the definition of~$\tau_{\cC}$ means that~$\gamma^!(D) \in {^{\tau_{\cC}}}\cC^{\leq b}$;
in other words, $\gamma^!$ has right t-amplitude~$\leq b$, as required.
\end{proof}

\begin{remark}
\label{remark-connectively-induced-opposite}
Similar to the case of projected t-structures discussed in Remark~\ref{remark-opposite-category},
the theories of connectively and coconnectively induced t-structure are formally equivalent under the passage to opposite categories.
Namely, let~$\cD$ be a triangulated category with a t-structure~$\tau$,
and let $\cD^{\op}$ be the opposite category, which by Remark~\ref{remark-opposite-t-structure}
is equipped with the opposite t-structure $\tau^{\op}$.
Let $\gamma \colon \cC \to \cD$ be a fully faithful triangulated functor,
and let~$\gamma^\op \colon \cC^\op \to \cD^\op$ be its opposite.
Then it follows immediately from the definitions that $\tau$ connectively induces a t-structure~$\tau^-_{\cC}$ on $\cC$
if and only if $\tau^{\op}$ coconnectively induces a t-structure $(\tau^{\op})_{\cC^{\op}}^{+}$ on $\cC^{\op}$,
and in this case, $(\tau^{\op})_{\cC^{\op}}^{+} = (\tau^-_{\cC})^\op$, i.e.
$(\tau^{\op})_{\cC^{\op}}^{+}$ is the opposite of the t-structure $\tau^-_{\cC}$.
\end{remark}

The notion of induced t-structures generalizes that of restricted and projected t-structures
studied in~\S\ref{section-restricting-t-structures} and~\S\ref{section-projected-t-structures}:
 
\begin{lemma}
\label{lemma-projected-is-induced}
Let~$\cD$ be a triangulated category equipped with a t-structure~$\tau = ({}^\tau\cD^{\leq 0}, {}^\tau\cD^{\geq 0})$. 
Let~$\cC$ be a triangulated category with a fully faithful triangulated functor $\gamma \colon \cC \to \cD$.
\begin{enumerate}[label={\textup{(\arabic*)}}]
\item
\label{it:res-ind}
If~$\tau$ restricts to a t-structure on~$\cC$,
then~$\tau$ connectively and coconnectively induces t-structures on~$\cC$,
and the restricted t-structure~$\tau\vert_\cC$ and the induced t-structures~$\tau^\pm_{\cC}$ coincide.
\item
\label{it:proj-right-ind}
If~$\tau$ right projects to a t-structure on~$\cC$, then $\tau$ connectively induces a t-structure on~$\cC$,
and the projected t-structure~$\gamma^!\tau$ and the induced t-structure~$\cin{\cC}$ coincide.
\item
\label{it:proj-left-ind}
If~$\tau$ left projects to a t-structure on~$\cC$, then~$\tau$ coconnectively induces a t-structure on~$\cC$,
and the projected t-structure~$\gamma^*\tau$ and the induced t-structure~$\ccin{\cC}$ coincide.
\end{enumerate}
\end{lemma}

\begin{proof}
Property~\ref{it:res-ind} follows immediately from the definitions.
For~\ref{it:proj-right-ind} and~\ref{it:proj-left-ind}, it suffices to show that~$\gamma^!\tau$ and~$\gamma^*\tau$
satisfy~${^{\gamma^!\tau}}\cC^{\leq 0} ={}^\tau\cD^{\leq 0} \cap \cC$
and~${^{\gamma^*\tau}}\cC^{\geq 0} ={}^\tau\cD^{\geq 0} \cap \cC$,
which holds by Lemma~\ref{lemma-projection-tau-cB}\ref{beta*t-geq0}.
\end{proof} 

\subsection{Reduction of t-amplitude}
\label{subsection-reduction-t-amplitude} 

As we have seen in \S\ref{section-restricting-t-structures} and~\S\ref{section-projected-t-structures}
the condition that $\tau$ restricts or projects to a t-structure on $\cC$ is quite strong;
for instance, by Theorem~\ref{theorem-projected-t-structures} the condition
that $\tau$ right or left projects to~$\cC$ is equivalent to the right or left t-exactness
of the corresponding projection functor.
We expect that~$\tau$ induces a t-structure under much weaker hypotheses.
The main goal of this section is to prove some criteria along these lines,
including our main result, Theorem~\ref{main-theorem}.

The key ingredient is a method for reducing the t-amplitude of a projection functor,
by passing to a new t-structure adapted to this purpose.
We explain this crucial tilting construction in greater generality than needed for the proof of Theorem~\ref{main-theorem};
as we speculate in~\S\ref{subsection-t-exactability}, the additional generality may be useful
for constructing t-structures on semiorthogonal components in other situations.

We say that an endofunctor~$\Phi \colon \cD \to \cD$ is \emphsf{idempotent}
if there is an isomorphism of functors~\mbox{$\Phi^2 \simeq \Phi$};
for instance, the projection functors~$\beta \circ \beta^*$ and~$\gamma \circ \gamma^!$
of a semiorthogonal decomposition are idempotent.

Recall from Definition~\ref{def:torsion-pair} the definition of a torsion pair~$(\cT,\cF)$ in an abelian category~$\cA$.
In particular, recall that the torsion part~$\cT$ of a torsion pair must be closed under extensions and quotients in~$\cA$,
the torsion-free part~$\cF$ must be closed under extensions and subobjects in~$\cA$,
and that each of them determines the other by~\eqref{eq:ct-cf}. 

\begin{proposition}
\label{proposition-tauPhib}
Let $\cD$ be a triangulated category with a t-structure $\tau$.
Let $\Phi \colon \cD \to \cD$ be an idempotent endofunctor.

\begin{enumerate}[label={\textup{(\alph*)}}]
\item
\label{it:tilting-minus}
Assume~$b \ge 1$ is an integer such that~$\Phi$ has right t-amplitude~$\leq b$ with respect to~$\tau$.
For~$0 \le k \le b$ define the subcategories 
\begin{equation}
\label{eq:cd-minus}
\cD_{\Phi,k}^{-} \coloneqq \set{D \in  {^\tau}\cD^{\leq 0} \sth
\Phi(D) \in {^\tau}\cD^{\leq k} } 
\end{equation} 
of the connective part of $\tau$, as well as the corresponding subcategories 
\begin{equation}
\label{eq:ct-minus}
\cT^-_{\Phi, k} \coloneqq \cD_{\Phi,k}^{-} \cap {^\tau}\cD^{\heartsuit}
\end{equation}
of the heart of $\tau$, so that we have chains of subcategories 
\begin{align*}
 \cD^-_{\Phi,0}  \subset \dots \subset \cD^-_{\Phi, b-1} &\subset \cD^-_{\Phi, b} = {^\tau} \cD^{\leq 0}, \\
  \cT^-_{\Phi,0} \subset \dots \subset \cT^-_{\Phi,b-1} &\subset \cT^-_{\Phi,b} = {}^\tau\cD^\heartsuit.
\end{align*}  
Then the following properties hold:
\begin{enumerate}[label={\textup{(\arabic*)}}]
\item
\label{it:ct-minus}
$\cT^-_{\Phi, b-1}$ is closed under extensions and quotients in~${^\tau}\cD^\heartsuit$.
\item 
\label{it:ct-minus-phi}
For any $D \in \cD^-_{\Phi, k}$, we have ${^\tau}\cH^{k}(\Phi(D)) \in \cT^-_{\Phi,b-1}$.
\item 
\label{it:DPhi-b-1-looks-like-tilt}
We have~$\cD^-_{\Phi,b-1} = \set{D \in {^\tau}\cD^{\leq 0} \sth {^\tau}\cH^0(D) \in \cT^-_{\Phi,b-1}}$.
\end{enumerate}
Further, if~$\cT^-_{\Phi, b-1}$ extends to a torsion pair in~${^\tau}\cD^\heartsuit$, 
then the tilt~$\tau^-_{\Phi, b-1}$ of~$\tau$ with respect to this torsion pair satisfies the following properties:
\begin{enumerate}[label={\textup{(\arabic*)}}, resume]
\item
\label{it:tau-minus}
The connective part of~$\tau^-_{\Phi, b-1}$ is given by ${}^{\tau^-_{\Phi,b-1}}\cD^{\leq 0} = \cD^-_{\Phi, b-1}$. 

\item
\label{it:tau-minus-phi}
The endofunctor~$\Phi$ has right t-amplitude~$\leq b-1$ with respect to~$\tau^-_{\Phi, b-1}$.
\end{enumerate}

\item
\label{it:tilting-plus}
Assume~$a \le -1$ is an integer such that~$\Phi$ has left t-amplitude~$\geq a$ with respect to~$\tau$.
For~$a \le k \le 0$ define the subcategories 
\begin{equation}
\cD_{\Phi,k}^{+} \coloneqq \set{D \in  {^\tau}\cD^{\geq 0} \sth 
\Phi(D) \in {^\tau}\cD^{\geq k} } 
\end{equation}
of the coconnective part of $\tau$, as well as the corresponding subcategories 
\begin{equation}
\label{eq:ct-plus}
\cF^+_{\Phi, k} \coloneqq \cD_{\Phi,k}^{+} \cap {^\tau}\cD^{\heartsuit}
\end{equation}
of the heart of $\tau$, so that we have chains of subcategories 
\begin{align*}
 \cD^+_{\Phi,0}  \subset \dots \subset \cD^{{+}}_{\Phi, {a+1}}  &\subset \cD^+_{\Phi, a} = {^\tau} \cD^{{\geq 0}}, \\
  \cF^+_{\Phi,0}  \subset \dots \subset \cF^+_{\Phi, {a+1}} &\subset \cF^+_{\Phi,a} = {}^\tau\cD^\heartsuit.
\end{align*}  
Then the following properties hold:
\begin{enumerate}[label={\textup{(\arabic*)}}]
\item
\label{it:ct-plus}
$\cF^+_{\Phi, a+1}$ is closed under extensions and subobjects in~${^\tau}\cD^\heartsuit$.
\item
\label{it:ct-plus-phi}
For any $D \in \cD^+_{\Phi,k}$, we have ${^\tau}\cH^{k}(\Phi(D)) \in \cF^+_{\Phi,a+1}$. 
\item
We have~$\cD^+_{\Phi, a+1} =  \set{D \in {^\tau}\cD^{\geq 0} \sth {^\tau}\cH^0(D) \in \cF^+_{\Phi,a+1}}$.
\end{enumerate}
Further, if~$\cF^+_{\Phi, a+1}$ extends to a torsion pair in~${^\tau}\cD^\heartsuit$
then the t-structure~$\tau^+_{\Phi, a+1}$ given by the shift by~$[-1]$
of the tilt of~$\tau$ with respect to this torsion pair satisfies the following properties:
\begin{enumerate}[label={\textup{(\arabic*)}}, resume]
\item
\label{it:tau-plus}
The coconnective part of~$\tau^+_{\Phi, a+1}$ is given by ${}^{\tau^+_{\Phi,a+1}}\cD^{\geq 0} = \cD_{\Phi, a+1}^+$. 

\item
\label{it:tau-plus-phi}
The endofunctor~$\Phi$ has left t-amplitude~$\geq a+1$ with respect to~$\tau^+_{\Phi, a+1}$.
\end{enumerate}
\end{enumerate}
\end{proposition}

\begin{proof}
We prove case~\ref{it:tilting-minus}, case~\ref{it:tilting-plus} being analogous.

\ref{it:ct-minus}
The category~$\cT^-_{\Phi,b-1}$ is closed under extensions in~${^\tau}\cD^{\heartsuit}$
because~${^\tau}\cD^{\leq b-1}$ is closed under extensions in~$\cD$.
To show that~$\cT^-_{\Phi,b-1}$ is closed under quotients,
consider an exact sequence~\mbox{$0 \to D' \to T \to D'' \to 0$} in~${^{\tau}}\cD^{\heartsuit}$
with~\mbox{$T \in \cT^-_{\Phi,b-1}$}.
Then we obtain a distinguished triangle
\begin{equation*}
\Phi(T) \to \Phi(D'') \to \Phi(D')[1].
\end{equation*}
The last term is contained in~${^\tau}\cD^{\leq b-1}$ by our assumption that~$\Phi$ has right t-amplitude~\mbox{$\leq b$},
while the first is contained in~${^\tau}\cD^{\leq b-1}$ by the definition of~$\cT^-_{\Phi, b-1}$.
It follows that their extension~$\Phi(D'')$ is also contained in~${^\tau}\cD^{\leq b-1}$,
i.e. the quotient~$D''$ of~$T$ is contained in~$\cT^-_{\Phi,b-1}$.

\ref{it:ct-minus-phi}
Let~$D \in \cD^-_{\Phi,k}$, i.e. $\Phi(D) \in {}^\tau\cD^{\le k}$.
We must show that $\Phi({^\tau}\cH^k(\Phi(D))) \in {^{\tau}}\cD^{\leq b-1}$.
By our assumption on the t-amplitude of $\Phi$,
this is equivalent to the vanishing of~${^\tau}\cH^b(\Phi({^\tau}\cH^k(\Phi(D))))$.
By our assumption that $\Phi$ is idempotent, we have an isomorphism $\Phi(\Phi(D)) \cong \Phi(D)$
and hence a spectral sequence
\begin{equation}
\label{eq:ss}
\rE_2^{p,q} = {^\tau}\cH^p(\Phi({}^\tau\cH^q(\Phi(D)))) \implies {}^\tau\cH^{p+q}(\Phi(D)).
\end{equation}
Since~$\Phi$ has right t-amplitude~$\leq b$ and~$D \in \cD^-_{\Phi,k}$,
the~$\rE_2^{p,q}$ terms vanish unless~$p \leq b$ and~$q \leq k$.
Hence the spectral sequence gives an isomorphism
\begin{equation*}
{}^\tau\cH^b(\Phi({}^\tau\cH^k(\Phi(D)))) \cong {}^\tau\cH^{b+k}(\Phi(D)).
\end{equation*}
But~$\Phi(D) \in {}^\tau\cD^{\le k}$ by definition of~$\cD^-_{\Phi,k}$
while $b + k > k$ by our assumption that~$b$ is positive, so the above object must vanish.

\ref{it:DPhi-b-1-looks-like-tilt}
Let~$D \in {^\tau}\cD^{\leq 0}$ and consider its truncation triangle
\begin{equation*}
\tau^{\leq -1} D \to D \to {^\tau}\cH^0(D).
\end{equation*}
Applying $\Phi$, we obtain a distinguished triangle
\begin{equation}
\label{triangle-PhiD}
\Phi(\tau^{\leq -1} D) \to \Phi(D) \to \Phi({^\tau}\cH^0(D)).
\end{equation}
As $\Phi$ has right t-amplitude $\leq b$ with respect to $\tau$,
the first term in the triangle is in~${^\tau}\cD^{\leq b-1}$ while the last two terms both lie in ${^\tau}\cD^{\leq b}$;
it follows that the map between the last two terms induces an isomorphism on ${^\tau}\cH^b$.
This implies the desired formula for $\cD^-_{\Phi,b-1}$. 

\ref{it:tau-minus}
This is an immediate consequence of~\ref{it:DPhi-b-1-looks-like-tilt} and the definition of tilting (see Proposition~\ref{proposition-tilting}). 

\ref{it:tau-minus-phi}
We must show that if~$D \in {}^{\tau^-_{\Phi,b-1}}\cD^{\leq 0}$ then~$\Phi(D) \in {}^{\tau^-_{\Phi,b-1}}\cD^{\leq b-1}$.
Since~$D \in {^\tau}\cD^{\leq 0}$, we have the distinguished triangle~\eqref{triangle-PhiD} above.
The first term~$\Phi(\tau^{\leq -1} D)$ lies in~${^\tau}\cD^{\leq b-1}$ because~$\Phi$ has right t-amplitude~$\leq b$.
In particular, $\tau^{\le -1}D \in \cD^-_{\Phi,b-1}$ by~\eqref{eq:cd-minus},
hence the degree~$b-1$ cohomology object~${^\tau}\cH^{b-1}(\Phi(\tau^{\leq -1} D))$
lies in~$\cT^-_{\Phi,b-1}$ by~\ref{it:ct-minus-phi} with~$k = b - 1$;
therefore,~$\Phi(\tau^{\leq -1} D) \in {}^{\tau^-_{\Phi,b-1}}\cD^{\leq b-1}$ by~\ref{it:tau-minus} and~\ref{it:DPhi-b-1-looks-like-tilt}.
On the other hand, the last term~$\Phi({^\tau}\cH^0(D))$ lies in~${^\tau}\cD^{\leq b-1}$
because~\mbox{${^\tau}\cH^0(D) \in \cT^-_{\Phi,b-1}$},
and the degree~$b-1$ cohomology object of~$\Phi({^\tau}\cH^0(D))$
lies in~$\cT^-_{\Phi,b-1}$ by~\ref{it:ct-minus-phi} with~$k = b-1$;
hence~\mbox{$\Phi({^\tau}\cH^0(D)) \in {}^{\tau^-_{\Phi,b-1}}\cD^{\leq b-1}$}.
We conclude that~$\Phi(D)$ lies in~${}^{\tau^-_{\Phi,b-1}}\cD^{\leq b-1}$,
being an extension of two objects in this category.
\end{proof}

In the following situation we can verify the extension to a torsion pair.

\begin{corollary}
\label{cor:exactify}
Let~$\cD$ be a triangulated category with a t-structure~$\tau$.
Let~\mbox{$\Phi \colon \cD \to \cD$} be an idempotent endofunctor.

\begin{enumerate}[label={\textup{(\alph*)}}]
\item
\label{it:tau-minus-0}
If~{$\tau$ is noetherian and}~$\Phi$ has right t-amplitude~$\leq 1$ with respect to~$\tau$,
then the subcategory~$\cT^-_{\Phi,0} \subset {}^\tau\cD^{\heartsuit}$ defined in~\eqref{eq:ct-minus}
extends to a torsion pair in~${}^\tau\cD^{\heartsuit}$
and~$\Phi$ is right t-exact with respect to the tilted t-structure~$\tau^-_{\Phi,0}$, whose connective part is 
\begin{equation}
\label{eq:tau-phi-minus}
{}^{\tau^-_{\Phi,0}}\cD^{\leq 0} = \set{ D \in {^\tau}\cD^{\le 0} \mid \Phi(D) \in {^\tau}\cD^{\le 0} }.
\end{equation}

\item
\label{it:tau-plus-0}
If~{$\tau$ is artinian and}~$\Phi$ has left t-amplitude~$\geq -1$ with respect to $\tau$,
then the subcategory~$\cF^+_{\Phi,0} \subset {}^\tau\cD^{\heartsuit}$ defined in~\eqref{eq:ct-plus}
extends to a torsion pair in~${}^\tau\cD^{\heartsuit}$
and~$\Phi$ is left t-exact with respect to the $[-1]$-shifted tilted t-structure~$\tau^+_{\Phi,0}$,
whose coconnective part is
\begin{equation}
\label{eq:tau-phi-plus}
{}^{\tau^+_{\Phi,0}}\cD^{\geq 0} = \set{ D \in {^\tau}\cD^{\ge 0} \mid \Phi(D) \in {^\tau}\cD^{\ge 0} }.
\end{equation}
\end{enumerate}
\end{corollary}

\begin{proof}
We prove case~\ref{it:tau-minus-0}, case~\ref{it:tau-plus-0} being analogous. 
Since~$\tau$ is noetherian,
combining Proposition~\ref{proposition-tauPhib}\ref{it:ct-minus} and Lemma~\ref{lemma-torsion-pair} shows that $\cT_{\Phi, 0}^-$ extends to a torsion pair in~${}^\tau\cD^\heartsuit$. 
Then applying Proposition~\ref{proposition-tauPhib}\ref{it:tau-minus} and~\ref{proposition-tauPhib}\ref{it:tau-minus-phi}
we obtain~\eqref{eq:tau-phi-minus} and conclude that~$\Phi$ is right t-exact
with respect to the corresponding tilted t-structure~$\tau^-_{\Phi,0}$.
\end{proof}

\subsection{The main theorem and complements} 
\label{ss:proof} 

Now we are ready to prove our main theorem, Theorem~\textup{\ref{main-theorem}},
stated in a more precise form below.

\begin{theorem}
\label{main-theorem-precise}
Let $\cD = \langle \cB, \cC \rangle$ be a semiorthogonal decomposition of a triangulated category. 
Let~$\beta \colon \cB \to \cD$ and~$\gamma \colon \cC \to \cD$ be the inclusion functors,
and~$\beta^* \colon \cD \to \cB$ and~$\gamma^! \colon \cD \to \cC$ their left and right adjoints. 
Let~$\tau$ be a t-structure on~$\cD$.
\begin{enumerate}[label={\textup{(\alph*)}}]
\item
\label{main-theorem-induce-C-precise}
Assume~$\tau$ is noetherian and restricts to a t-structure on~$\cB$.
Then~$\tau$ connectively induces a t-structure~$\cin{\cC}$ on~$\cC$ such that the following properties hold:
\begin{enumerate}[label={\textup{(\arabic*)}}]

\item \label{main-theorem-perverse-t-structure}
There is a t-structure $\cin{\gamma\gamma^!}$ on $\cD$ with connective part given by 
\begin{equation}
\label{eq:tau-gg}
{}^{\cin{\gamma\gamma^!}}\cD^{\le 0} =
\set{ D \in {^\tau}\cD^{\le 0} \mid \gamma\gamma^!(D) \in {^\tau}\cD^{\le 0} }, 
\end{equation}
which right projects to a t-structure $\gamma^!(\cin{\gamma\gamma^!})$ on $\cC$ that coincides with $\cin{\cC}$. 

\item
\label{induced-t-structure-bounded-objects-main-theorem}
The heart of~$\cin{\cC}$ and the subcategories of bounded objects are given by
\begin{equation*}
{}^{\cin{\cC}}\cC^{\heartsuit} = \gamma^! \left( {}^{\cin{\gamma\gamma^!}}\cD^{\heartsuit} \right)
\quad\text{and}\quad
{}^{\cin{\cC}}\cC^{?} = \gamma^! \left({}^{\cin{\gamma \gamma^!}}\cD^? \right)
\quad\text{for}\quad
? \in \{+,-,\mathrm{b}\}. 
\end{equation*}

\item
\label{induced-t-structure-adjoint-main-theorem}
$\gamma^!$ is t-exact with respect to $\cin{\gamma \gamma^!}$ and $\cin{\cC}$.

\item
\label{gamma-right-t-exact-tauC-main-theorem}
$\gamma$ is right t-exact with respect to~$\cin{\cC}$ and~$\cin{\gamma \gamma^!}$.

\item
\label{tau-bounded-implies-tauC-bounded-main-theorem}
If $\tau$ is bounded, then $\cin{\gamma\gamma^!}$ and $\cin{\cC}$ are bounded. 
\end{enumerate}

\item
\label{main-theorem-induce-B-precise}
Assume~$\tau$ is artinian and restricts to a t-structure on~$\cC$.
Then~$\tau$ coconnectively induces a t-structure~$\ccin{\cB}$ on~$\cB$ such that the following properties hold:
\begin{enumerate}[label={\textup{(\arabic*)}}]

\item There is a t-structure $\ccin{\beta\beta^*}$ on $\cD$ with coconnective part given by 
\begin{equation}
\label{eq:ccinbetabeta*}
{}^{\ccin{{\beta\beta^*}}}\cD^{{\ge 0}} = \set{ D \in {^\tau}\cD^{\ge 0} \mid \beta\beta^*(D) \in {^\tau}\cD^{\ge 0} },
\end{equation}
which left projects to a t-structure $\beta^*(\ccin{\beta\beta^*})$ on $\cB$ that coincides with $\ccin{\cB}$. 

\item
\label{induced-t-structure-bounded-objects-main-theorem-b}
The heart of~$\ccin{\cB}$ and the subcategories of bounded objects are given by
\begin{equation*}
{}^{\ccin{\cB}}\cB^{\heartsuit} = \beta^* \left( {}^{\ccin{\beta\beta^*}}\cD^{\heartsuit} \right)
\quad\text{and}\quad
{}^{\ccin{\cB}}\cB^{?} = \beta^* \left({}^{\ccin{\beta\beta^*}}\cD^? \right)
\quad\text{for}\quad
? \in \{+,-,\mathrm{b}\}. 
\end{equation*}

\item
\label{induced-t-structure-adjoint-main-theorem-b}
$\beta^*$ is t-exact with respect to $\ccin{\beta\beta^*}$ and $\ccin{\cB}$.

\item
\label{gamma-right-t-exact-tauC-main-theorem-b}
$\beta$ is left t-exact with respect to~$\ccin{\cB}$ and~$\ccin{\beta\beta^*}$.

\item
\label{tau-bounded-implies-tauC-bounded-main-theorem-b}
If $\tau$ is bounded, then $\ccin{\beta\beta^*}$ and $\ccin{\cB}$ are bounded. 
\end{enumerate}

\end{enumerate} 
\end{theorem}

\begin{proof}
We prove case~\ref{main-theorem-induce-C-precise}, case~\ref{main-theorem-induce-B-precise} being analogous. 
By Lemma~\ref{lemma-restrict-t-structure-vs-projection-amplitude},
if~$\tau$ restricts to a t-structure on~$\cB$, then~$\Phi \coloneqq \gamma \circ \gamma^!$
has right t-amplitude~$\leq 1$ with respect to~$\tau$.
Then Corollary~\ref{cor:exactify} provides a t-structure~$\cin{\gamma\gamma^!} \coloneqq \cin{\Phi,0}$ with connective part~\eqref{eq:tau-gg}, 
with respect to which~$\gamma \circ \gamma^!$ is right t-exact.
By Lemma~\ref{lemma-project-B-iff-lte}\ref{it:projection-right-criterion} this t-structure
right projects to a t-structure $\gamma^!(\cin{\gamma\gamma^!})$ on~$\cC$. 

The items~\ref{induced-t-structure-bounded-objects-main-theorem}--\ref{tau-bounded-implies-tauC-bounded-main-theorem}
hold with~$\cin{\cC}$ replaced by~$\gamma^!(\cin{\gamma\gamma^!})$.
Indeed, if~$\tau$ is bounded, then so is its tilt~$\cin{\gamma\gamma^!}$ (Remark~\ref{remark-characterizing-tilt}),
and the rest of the claims follow from Lemma~\ref{lemma-projection-tau-cB}.

It remains to show that $\gamma^!(\cin{\gamma\gamma^!})$ is connectively induced by~$\tau$, 
i.e. that~${}^{\gamma^!(\cin{\gamma\gamma^!})}\cC^{\le 0} = {}^\tau\cD^{\le 0} \cap \cC$.
Combining Lemma~\ref{lemma-projection-tau-cB}\ref{beta*t-geq0} and~\eqref{eq:tau-gg},
we obtain
\begin{equation*}
{}^{\gamma^!(\cin{\gamma\gamma^!})}\cC^{\le 0} =
{}^{\cin{\gamma\gamma^!}}\cD^{\le 0} \cap \cC =
\set{ C \in \cC \sth \gamma(C) \in {^\tau}\cD^{\le 0} \text{ and } \gamma\gamma^!(\gamma(C)) \in {^\tau}\cD^{\le 0} }. 
\end{equation*}
Since~$\gamma \circ \gamma^! \circ \gamma \cong \gamma$, we see that this is equal to~${}^\tau\cD^{\le 0} \cap \cC$, as required.
\end{proof}

The following lemma gives a mechanism for constructing objects in the hearts of the induced t-structures
in the situation of Theorem~\ref{main-theorem-precise}.

\begin{lemma}
\label{lemma-object-in-heart-tauB}
In the situation of Theorem~\textup{\ref{main-theorem-precise}}, the following hold:
\begin{enumerate}[label={\textup{(\alph*)}}]
\item
\label{it:objects-C} 
Assume that~$\tau$ is noetherian and restricts to a t-structure on~$\cB$.
Let~$D \in \cD$ be an object such that 
\begin{equation*}
\gamma\gamma^!(D) \in {}^\tau\cD^{\leq 0} \quad \text{and} \quad 
\tau^{\leq -1} \gamma\gamma^!(D) \in \cB. 
\end{equation*} 
Then~$\gamma^!(D) \in \cC$ is contained in the heart ${}^{\cin{\cC}}\cC^{\heartsuit}$ of the connectively induced t-structure. 
\item 
\label{it:objects-B}
Assume that~$\tau$ is artinian and restricts to a t-structure on~$\cC$.
Let~$D \in \cD$ be an object such that
\begin{equation*}
\beta\beta^*(D) \in {}^\tau\cD^{\geq 0} \quad \text{and} \quad \tau^{\geq 1} \beta \beta^*(D) \in \cC. 
\end{equation*} 
Then~$\beta^*(D) \in \cB$ is contained in the heart ${}^{\ccin{\cB}}\cB^{\heartsuit}$ of 
the coconnectively induced t-structure. 
\end{enumerate}
\end{lemma}

\begin{proof}
We prove part~\ref{it:objects-C}, part~\ref{it:objects-B} being analogous. 
Since~$\gamma\gamma^!\gamma\gamma^!(D) \cong \gamma\gamma^!(D) \in {}^\tau\cD^{\le 0}$,
we deduce from~\eqref{eq:tau-gg} that~$\gamma\gamma^!(D) \in {}^{\cin{\gamma\gamma^!}}\cD^{\le 0}$,
and since the t-structure~$\cin{\gamma\gamma^!}$ is obtained from~$\tau$ by tilting, we have
\begin{equation*}
{}^\tau\cH^0(\gamma\gamma^!(D)) \in {}^{\cin{\gamma\gamma^!}}\cD^\heartsuit.
\end{equation*}
Furthermore, since the connectively induced t-structure~$\cin{\cC}$ is given by the right projection of~$\cin{\gamma\gamma^!}$, we have  
\begin{equation*}
\gamma^!({}^\tau\cH^0(\gamma\gamma^!(D))) \in {}^{\cin{\cC}}\cC^\heartsuit.
\end{equation*}
Finally, applying~$\gamma^!$ to the truncation
triangle~$\tau^{\leq -1} \gamma\gamma^!(D) \to \gamma\gamma^!(D) \to {}^\tau\cH^0(\gamma\gamma^!(D))$
and taking into account that its first term is contained in~$\cB$ and hence goes to zero,
we deduce that 
\begin{equation*}
\gamma^!(D) \cong \gamma^!\gamma\gamma^!(D) \cong \gamma^!({}^\tau\cH^0(\gamma\gamma^!(D))) \in {}^{\cin{\cC}}\cC^\heartsuit, 
\end{equation*} 
as required. 
\end{proof}

Note that in the situation of Theorem~\ref{main-theorem-precise},
it follows from Lemma~\ref{lemma-restrict-t-structure-vs-projection-amplitude}
that~$\gamma \circ \gamma^!$ has right t-amplitude $\leq 1$ with respect to $\tau$,
and hence from Lemma~\ref{lemma-tauB-basic-properties}\ref{it:induced-amplitude}
that~$\gamma^!$ has t-amplitude in~$[0,1]$ with respect to~$\tau$ and~$\cin{\cC}$.
Turning this around, we obtain a recognition principle for connectively induced t-structures on~$\cC$
with~$\gamma^!$ of t-amplitude in~$[0,1]$:

\begin{lemma}
\label{lem:tauC-recognition}
In Setup~\textup{\ref{setup-sod}},
let~$\upsilon$ be a t-structure on~$\cC$ such that the following conditions hold:
\begin{enumerate}[label={\textup{(\arabic*)}}]
\item
\label{it:gamma-shriek-rca1}
The functor~$\gamma^! \colon \cD \to \cC$ has t-amplitude in~$[0,1]$ with respect to~$\tau$ and~$\upsilon$.
\item
\label{it:gamma-shriek-epi}
For any object~$0 \ne C \in {}^{\upsilon}\cC^\heartsuit$,
there exist an object~$D \in {}^\tau\cD^\heartsuit$ and a nonzero morphism~$C \to \gamma^!(D)$.
\end{enumerate}
Then~$\tau$ restricts to a t-structure on~$\cB$ and we have
\begin{equation}
\label{equation-recog-induced-t-structure}
{}^{\upsilon}\cC^{\le 0}  = {^{\tau}}\cD^{\le 0} \cap \cC.
\end{equation}
In particular, $\tau$ connectively induces a t-structure on $\cC$ which coincides with $\upsilon$.
\end{lemma}

We leave to the interested reader to formulate a similar recognition principle for coconnectively induced t-structures.

\begin{proof}
Since the functor~$\gamma^!$ is left t-exact by assumption,
its left adjoint~$\gamma$ is right t-exact by Lemma~\ref{lemma-t-exactness-adjoints}; therefore, \ref{it:gamma-shriek-rca1} implies that the composition~$\gamma \circ \gamma^!$ has right t-amplitude~$\leq 1$.
Applying Lemma~\ref{lemma-restrict-t-structure-vs-projection-amplitude}, we conclude that~$\tau$ restricts to a t-structure on~$\cB$.

The forward inclusion in~\eqref{equation-recog-induced-t-structure}
follows from the right t-exactness of the functor~$\gamma$ explained above.
To prove the reverse inclusion, consider an object~$C \in \cC$ such that~$\gamma(C) \in {}^\tau\cD^{\le 0}$.
Then since $\gamma^!$ has right t-amplitude $\leq 1$, we have
\begin{equation*}
C \cong \gamma^!\gamma(C) \in {}^{\upsilon}\cC^{\le 1}.
\end{equation*}
Assume~${}^{\upsilon}\cH^1(C) \ne 0$.
By assumption~\ref{it:gamma-shriek-epi} there is an object~$D \in {}^\tau\cD^\heartsuit$
and a nonzero morphism~$\phi \colon {}^{\upsilon}\cH^1(C) \to \gamma^!(D)$.
Since~$\gamma^!(D) \in {}^\upsilon\cC^{\ge 0}$,
the induced morphism~${}^{\upsilon}\cH^1(C) \to {}^\upsilon\cH^0(\gamma^!(D))$ is also nonzero.
Thus, composing~$\phi$ with the canonical morphism~\mbox{$C[1] \to {}^{\upsilon}\cH^1(C)$},
we obtain a nonzero morphism~$C[1] \to \gamma^!(D)$,
and by adjunction a nonzero morphism~\mbox{$\gamma(C)[1] \to D$},
which is impossible because the source is in~${}^\tau\cD^{\le -1}$ and the target is in~${}^\tau\cD^{\ge 0}$.
Thus, we have~${}^{\upsilon}\cH^1(C) = 0$ and~$C \in {}^{\upsilon}\cC^{\le 0}$.
This completes the proof of~\eqref{equation-recog-induced-t-structure}.
\end{proof}

\begin{remark}
\label{remark-tca-restriction}
The following example shows that assumption~\ref{it:gamma-shriek-epi} in Lemma~\ref{lem:tauC-recognition} is necessary.
Let~$\cD = \cB \oplus \cC$ be a completely orthogonal decomposition,
and let~$\tau = {(\upsilon_\cB, \upsilon_\cC[-1])}$,
where~$\upsilon_\cB$ is a t-structure on~$\cB$
and~$\upsilon_\cC$ is a t-structure on~$\cC$.
Then~$\gamma^!$ has t-amplitude in~$[0,1]$ with respect to~$\tau$ and~$\upsilon_\cC$,
but~${}^{\upsilon_\cC}\cC^{\leq 0} = {^\tau}\cD^{\leq -1} \cap \cC$,
so the formula~\eqref{equation-recog-induced-t-structure} fails.
\end{remark}

\begin{corollary}
\label{cor:geometric}
In the situation of Example~\textup{\ref{example-bridgeland}}, if~$\tau = \tau_X$
is the standard t-structure on~$\cD = \Dm(X)$, 
then the connectively induced t-structure~$\cin{\cC}$ on~$\cC \simeq \Dm(Y)$
coincides with the standard t-structure~$\tau_Y$.
In particular, the subcategory of bounded objects~${}^{\cin{\cC}}\cC^{\mathrm{b}}$ with respect to~$\cin{\cC}$
identifies with~$\Db(Y)$, and the restriction of~$\cin{\cC}$ to this subcategory is the standard t-structure on~$\Db(Y)$.
\end{corollary}

\begin{proof}
We show that the assumptions of Lemma~\ref{lem:tauC-recognition}
are satisfied for~\mbox{$\cD = \Dm(X)$}, \mbox{$\cC = \Dm(Y)$},
the functor~$\gamma^! = f_* \colon \Dm(X) \to \Dm(Y)$,
and the standard t-structures~\mbox{$\tau = \tau_X$} and~\mbox{$\upsilon = \tau_Y$}.
First, assumption~\ref{it:gamma-shriek-rca1} holds because the dimension of the fibers of~$f$ is~$\le 1$.
Furthermore, let~$0 \ne C \in {}^{\tau_Y}(\Dm(Y))^\heartsuit = \Coh(Y)$,
let~$y \in Y$ be a closed point in the support of~$C$, and let~$x \in X$ be a closed point such that~$f(x) = y$
(which exists because~$f$ is surjective).
Then
\begin{equation*}
\Hom(C,f_*(\cO_x)) \cong \Hom(C,\cO_y) \ne 0,
\end{equation*}
hence~\ref{it:gamma-shriek-epi} holds.
Therefore, Lemma~\ref{lem:tauC-recognition} applies to give the corollary.
\end{proof}

Finally, together with our results from~\S\ref{section-via-simple-generators},
we deduce the following strengthened version of Theorem~\ref{theorem-induce-via-exceptional-sequence} from the introduction.

\begin{theorem}
\label{theorem-induce-via-simple-objects}
Let~$\cD$ be a $\kk$-linear triangulated category for a field~$\kk$, with a t-structure~$\tau$.
Let $\cD = \langle \cB, \cC \rangle$ be a semiorthogonal decomposition.
\begin{enumerate}[label={\textup{(\alph*)}}]
\item
\label{induce-via-simple-B}
If~$\tau$ is noetherian and~$\cB$ is thickly generated
by a finite set of objects~$B_1, \dots, B_n\in {}^\tau\cD^\heartsuit$ in the heart of~$\tau$ such that
\begin{equation}
\label{Bi-simple-orthogonal}
\Hom(B_i,B_i) = \kk \text{ for all~$i$, and }
\Hom(B_i,B_j) = 0  \text{ for all~$i \ne j$},
\end{equation}
then~$\tau$ connectively induces a t-structure~$\cin{\cC}$ on~$\cC$, which is bounded if~$\tau$ is bounded.

\item
\label{induce-via-simple-C}
If~$\tau$ is artinian and~$\cC$ is thickly generated
by a finite set of objects~$C_1, \dots, C_n\in {}^\tau\cD^\heartsuit$ in the heart of~$\tau$ such that
\begin{equation}
\label{Ci-simple-orthogonal}
\Hom(C_i,C_i) = \kk \text{ for all~$i$, and }
\Hom(C_i,C_j) = 0  \text{ for all~$i \ne j$},
\end{equation}
then~$\tau$ coconnectively induces a t-structure~$\ccin{\cB}$ on~$\cB$, which is bounded if~$\tau$ is bounded.
\end{enumerate}
\end{theorem}

\begin{proof}
Proposition~\ref{prop:simple-heart} shows that $\tau$ restricts to a t-structure on $\cB$ in case~\ref{induce-via-simple-B} or on $\cC$ in case~\ref{induce-via-simple-C}, so we may apply Theorem~\ref{main-theorem}.
\end{proof}

\begin{proof}[Proof of Theorem~\textup{\ref{theorem-induce-via-exceptional-sequence}}]
Since exceptional objects are simple, the required result
follows from Theorem~\ref{theorem-induce-via-simple-objects}
by taking either~$B_i = E_i$ or~\mbox{$C_i = E_i$}.
\end{proof}

\subsection{t-exactability and perverse t-structures}
\label{subsection-t-exactability} 

The t-amplitude assumptions of Corollary~\ref{cor:exactify} that allowed us to find a t-structure
with respect to which the given idempotent endofunctor~$\Phi$ is right or left t-exact are rather restrictive.
In this somewhat speculative subsection, we discuss an approach to extending this corollary
to idempotent endofunctors with more general t-amplitude,
and in turn Theorem~\ref{main-theorem} to more general semiorthogonal decompositions
(for which we do not require the given t-structure on~$\cD$ to restrict to one of the components).

We confine ourselves to the right t-exact version of the story;
the interested reader will easily find the corresponding statements for the left t-exact version.

With this goal in mind, we introduce the following definition. 

\begin{definition}
\label{definition-t-exactable} 
Let~$\cD$ be a triangulated category with a t-structure $\tau$.
Let~$\Phi \colon \cD \to \cD$ be an idempotent endofunctor.
We say that $\Phi$ is \emphsf{right t-exactable with respect to $\tau$}
if the pair~$\cin{\Phi} = ({^{\cin{\Phi}}}\cD^{\leq 0}, {^{\cin{\Phi}}}\cD^{\geq 0})$ where
\begin{equation}
\label{eq:tau-perverse-minus}
\begin{aligned}
{^{\cin{\Phi}}}\cD^{\leq 0} & = \set{ D \in {^{\tau}}\cD^{\leq0} \sth \Phi(D) \in {^\tau}\cD^{\leq 0}}, \\
{^{\cin{\Phi}}}\cD^{\geq 0} & = \set{ D \in \cD 
\sth \Hom(D',D) = 0
\text{ \textup{for all} } D' \in {^{\cin{\Phi}}}\cD^{\leq 0}[1]},
\end{aligned}
\end{equation}
defines a t-structure on $\cD$, in which case $\tau^-_{\Phi}$ is called the \emphsf{right $\Phi$-perverse t-structure} associated to~$\tau$.
\end{definition} 

\begin{remark}
Let us explain the choice of terminology above: 
\begin{enumerate}
\item
If~$\Phi$ is right t-exactable with respect to~$\tau$, then it follows directly from the idempotency of~$\Phi$
that~$\Phi$ is right t-exact with respect to~$\tau^-_{\Phi}$.
\item 
In the context of Example~\ref{example-bridgeland}, Bridgeland defined a perverse t-structure on~$\Dm(X)$.
As we will explain later in~\S\ref{section-perverse-t-structures}, this t-structure (for a suitable choice of perversity)
coincides with~$\tau^-_{\Phi}$ for~$\Phi = f^* \circ f_*$.
\end{enumerate}
\end{remark}

The t-exactability of the projection functor for a semiorthogonal component
implies the existence of a well-behaved induced t-structure,
generalizing Theorem~\ref{main-theorem-precise} to this setting.

\begin{theorem}
\label{theorem-t-exactable-induced-tauC}
Let~$\cD$ be a triangulated category with a t-structure~$\tau = ({}^\tau\cD^{\leq 0}, {}^\tau\cD^{\geq 0})$.
Let~$\cC$ be a triangulated category with a fully faithful triangulated functor~$\gamma \colon \cC \to \cD$.
Assume that~$\gamma$ admits a right adjoint~$\gamma^! \colon \cD \to \cC$
and that the composition~$\gamma \circ \gamma^!$ is right t-exactable with respect to~$\tau$.
Then~$\tau$ connectively induces a t-structure~$\cin{\cC}$ on~$\cC$ such that the following properties hold:
\begin{enumerate}[label={\textup{(\arabic*)}}]

\item
\label{tauC-gamma!tau}
The right $\gamma \gamma^!$-perverse t-structure~$\cin{\gamma \gamma^!}$ right
projects to a t-structure $\gamma^!(\cin{\gamma \gamma^!})$ on~$\cC$ which coincides with~$\cin{\cC}$.

\item
\label{induced-t-structure-bounded-objects}
The heart of~$\cin{\cC}$ and the subcategories of bounded objects are given by
\begin{equation*}
{}^{\cin{\cC}}\cC^{\heartsuit} = \gamma^! \left( {}^{\cin{\gamma\gamma^!}}\cD^{\heartsuit} \right)
\quad\text{and}\quad
{}^{\cin{\cC}}\cC^{?} = \gamma^! \left({}^{\cin{\gamma \gamma^!}}\cD^? \right)
\quad\text{for}\quad
? \in \{+,-,\mathrm{b}\}. 
\end{equation*}

\item
\label{induced-t-structure-adjoint}
$\gamma^!$ is t-exact with respect to $\cin{\gamma \gamma^!}$ and $\cin{\cC}$.

\item
\label{gamma-right-t-exact-tauC}
$\gamma$ is right t-exact with respect to~$\cin{\cC}$ and~$\cin{\gamma \gamma^!}$.

\item
\label{tau-bounded-implies-tauC-bounded}
Assume that~$\tau$ is bounded and~$\gamma \circ \gamma^!$ has finite right t-amplitude with respect to~$\tau$.
Then~$\cin{\gamma\gamma^!}$ and~$\cin{\cC}$ are bounded.
\end{enumerate} 
\end{theorem} 

\begin{proof}
Everything holds by the same argument as in the proof of Theorem~\ref{main-theorem-precise},
except for~\ref{tau-bounded-implies-tauC-bounded}. 
If $b \geq 0$ is an integer such that~$\gamma \circ \gamma^!$ has right t-amplitude~$\leq b$ with respect to~$\tau$,
then we have chains of inclusions
\begin{equation*}
{}^{\tau}\cD^{\le -b} \subset {}^{\cin{\gamma \gamma^!}}\cD^{\le 0} \subset {}^{\tau}\cD^{\le 0}
\qquad
\text{and}\qquad
{}^{\tau}\cD^{\ge 0} \subset {}^{\cin{\gamma \gamma^!}}\cD^{\ge 0} \subset {}^{\tau}\cD^{\ge -b},
\end{equation*}
where the first holds by \eqref{eq:tau-perverse-minus} and the second by passing to (shifted) orthogonals. 
In particular, we see that if the t-structure~$\tau$ is bounded, then so is~$\cin{\gamma\gamma^!}$,
which by Lemma~\ref{lemma-projection-tau-cB}\ref{t-induced-bounded-cB}
implies the same for~$\cin{\cC} = \gamma^!(\cin{\gamma \gamma^!})$.
\end{proof}

The above result puts a premium on the following question. 

\begin{question}
\label{question-t-exactable}
Let~$\cD$ be a triangulated category with a t-structure~$\tau = ({}^\tau\cD^{\leq 0}, {}^\tau\cD^{\geq 0})$.
Given a right admissible subcategory $\cC \subset \cD$,
under what conditions can we guarantee that its projection functor~$\gamma \circ \gamma^!$
is right t-exactable with respect to~$\tau$?
\end{question} 

More generally, we could ask for conditions which guarantee that an idempotent endofunctor is right t-exactable. 
Proposition~\ref{proposition-tauPhib} suggests the following inductive approach to this question. 

\begin{remark}[t-exactability via tilting]
\label{remark-t-exactable-tilting}
Let $\cD$ be a triangulated category equipped with a t-structure~$\tau$. 
Let $\Phi \colon \cD \to \cD$ be an idempotent endofunctor
which has finite right t-amplitude with respect to~$\tau$, i.e. there exists an integer~$b \in \bZ$
such that~$\Phi$ has right t-amplitude~$\leq b$ with respect to~$\tau$. 
Then we may try to iteratively apply Proposition~\ref{proposition-tauPhib}
to construct a sequence of t-structures~$\tau_{\Phi,k}^-$ for~$0 \leq k \leq b$ on~$\cD$
such that~$\tau_{\Phi,b}^{-} = \tau$, each $\tau_{\Phi, k-1}^-$ is a tilt of $\tau_{\Phi,k}^{-}$, and the connective parts of the t-structures are given by
\begin{equation}
\label{tauPhi-k-via-tilting}
{}^{\tau^-_{\Phi,k}}\cD^{\leq 0} = \set{D \in  {^\tau}\cD^{\leq 0} \sth
\Phi(D) \in {^\tau}\cD^{\leq k} } ;
\end{equation}
in particular, for~$k = 0$ this would show that~$\Phi$ is right t-exactable and~$\tau_{\Phi}^- = \tau_{\Phi, 0}^-$
is the right $\Phi$-perverse t-structure. 
Note that if t-structures~$\tau_{\Phi, k}^-$ with connective parts~\eqref{tauPhi-k-via-tilting} really exist,
then we have obvious inclusions
\begin{equation*}
{}^{\tau^-_{\Phi,k}}\cD^{\leq -1} \subset {}^{\tau^-_{\Phi,k-1}}\cD^{\leq 0} \subset {}^{\tau^-_{\Phi,k}}\cD^{\leq 0},
\end{equation*}
which by Remark~\ref{remark-characterizing-tilt} 
imply that~$\tau_{\Phi, k-1}^-$ is necessarily obtained from~$\tau_{\Phi,k}^-$ by tilting.

The issue with this construction is the assumption in Proposition~\ref{proposition-tauPhib}
that~$\cT_{\Phi, b-1}^-$ extends to a torsion pair in ${^{\tau}}\cD^{\heartsuit}$.
Of course, if we assume that $\tau$ is noetherian, then this is automatic at the first step,
so we can construct~$\tau_{\Phi,b-1}^-$, as in the proof of Corollary~\ref{cor:exactify}.
However, the tilt~$\tau_{\Phi,b-1}^{-}$ is potentially non-noetherian, so it is not clear
that we can apply Proposition~\ref{proposition-tauPhib}
to~$\tau_{\Phi,b-1}^-$ and continue the induction. 
It would be very interesting to find criteria guaranteeing that the t-structures~$\tau_{\Phi,k}^-$ exist for~$k < b-1$. 
\end{remark}

The finiteness assumption on the t-amplitude of $\Phi$ in Remark~\ref{remark-t-exactable-tilting} is quite mild: 

\begin{remark}
\label{remark-boundedness-functors-spc}
Let~$\kk$ be a field and let~$\cD$ be a smooth proper $\kk$-linear category in the sense of~\cite{NCHPD};
for instance, we may take~$\cD$ to be (the natural enhancement of)~$\Db(X)$
where~$X$ is a smooth proper variety over~$\kk$~\cite[Lemma~4.9]{NCHPD}.
Let~$\tau$ be a bounded t-structure on~$\cD$.
Then any~$\kk$-linear endofunctor~$\Phi \colon \cD \to \cD$ has bounded t-amplitude with respect to~$\tau$,
i.e. $\Phi$ has t-amplitude in~$[a,b]$ for some integers~$a,b \in \bZ$~\cite[Proposition~3.1]{quasi-convergence-DHL}.
\end{remark} 

\begin{remark}[An approach to Question~\ref{question-soc-t-structure}]
\label{remark-conjecture-soc-t-structure-approach}
Let~$X$ be a smooth proper variety over a field~$\kk$
with a $\kk$-linear semiorthogonal decomposition~$\Db(X) = \langle \cB , \cC \rangle$.
If~$\tau$ is any bounded t-structure on~$\Db(X)$ --- for instance~$\tau = \tau_X$ the standard t-structure ---
then by Remark~\ref{remark-boundedness-functors-spc} the projection functor~$\gamma \circ \gamma^!$ has bounded t-amplitude with respect to~$\tau$.
Whenever it can be shown that~$\gamma \circ \gamma^!$ is right t-exactable
with respect to~$\tau$ (for instance by iterated tilting as in Remark~\ref{remark-t-exactable-tilting}),
then we obtain by Theorem~\ref{theorem-t-exactable-induced-tauC} 
an induced bounded t-structure on~$\cC$. 
In this way, a sufficiently robust answer to Question~\ref{question-t-exactable}
may lead to a solution to Question~\ref{question-soc-t-structure}.
\end{remark} 

\subsection{Other perversities}
\label{section-perverse-t-structures}

In the setting of Example~\ref{example-bridgeland},
Bridgeland \cite{bridgeland-flops} constructed a perverse t-structure $\tau_{X,p}^{\perv}$ on $\Dm(X)$
for any integer $p \in \bZ$, which can be thought of as a choice of perversity
(with $p = -1$ being the ``standard'' perversity).
In this auxiliary subsection, we generalize this story
to the setting of Theorem~\ref{main-theorem-precise}\ref{main-theorem-induce-C-precise};
as before, we leave the formulations and proofs of the left versions of the results to the interested reader.

Note that Theorem~\ref{main-theorem-precise}\ref{main-theorem-induce-C-precise}
shows the projection functor~$\gamma \circ \gamma^!$ is right t-exactable,
and the corresponding right $\gamma\gamma^!$-perverse t-structure $\cin{\gamma \gamma^!}$
right projects to a t-structure on~$\cC$ which coincides with the t-structure~$\cin{\cC}$ connectively induced by~$\tau$.
The next result shows that~$\cin{\gamma \gamma^!}$ is not unique with these properties
(under a mild additional assumption on~$\cB \subset \cD$).

\begin{lemma}
\label{lem:perv-p} 
Assume we are in the setting of Theorem~\textup{\ref{main-theorem-precise}}\ref{main-theorem-induce-C-precise}
and that the subcategory~\mbox{$\cB \subset \cD$} is admissible.
Let~$\tau\vert_\cB$ be the restricted t-structure on~$\cB$ and~$\cin{\cC}$ the connectively induced t-structure on~$\cC$. 
Then for each~$p \in \bZ$ there is a t-structure~$\tau^{\perv}_p$ on~$\cD$ such that
\begin{align*}
{^{\tau^{\perv}_p}}\cD^{\leq 0} & =
\set{ D \in \cD \sth \gamma^!(D) \in {^{\cin{\cC}}}\cC^{\leq 0}
\textup{ and } \Hom(D, \beta(B)) = 0 \textup{ for all } B \in {^{\tau\vert_{\cB}}}\cB^{\ge p + 1} },
\\
{^{\tau^{\perv}_p}}\cD^{\geq 0} & =
\set{ D \in \cD \sth \gamma^!(D) \in {^{\cin{\cC}}}\cC^{\geq 0}
\textup{ and } \Hom(\beta(B), D) = 0 \textup{ for all } B \in {^{\tau\vert_{\cB}}}\cB^{\le p - 1} }.
\end{align*}
Moreover, each t-structure~$\tau^{\perv}_p$ right projects to a t-structure on~$\cC$ which coincides with~$\cin{\cC}$.
\end{lemma}

\begin{proof}
The t-structure~$\tau^{\perv}_p$ is obtained by gluing (as defined in~\cite[{Th\'eor\`em}~1.4.10]{BBDG},
whose assumptions are satisfied by the admissibility of~$\cB$)
of the t-structure~$\tau\vert_{\cB}$ shifted by~$p$ with the t-structure~$\cin{\cC}$.
We note that the functor~$\gamma \circ \gamma^!$ is right t-exact with respect to each~$\tau^{\perv}_p$;
indeed, if~$D \in {^{\tau^{\perv}_p}}\cD^{\leq 0}$ and~$D' \coloneqq \gamma\gamma^!(D)$ then 
\begin{align*}
& \gamma^!(D') =
\gamma^!\gamma\gamma^!(D) \cong
\gamma^!(D) \in {^{\cin{\cC}}}\cC^{\leq 0},
\\
& \Hom(D', \beta(B)) =
\Hom(\gamma\gamma^!(D), \beta(B)) \cong
\Hom(\gamma^!(D), \gamma^!(\beta(B))) = 0
\quad\text{for all~$B \in \cB$},
\end{align*} 
and hence~$D' \in {^{\tau^{\perv}_p}}\cD^{\leq 0}$.
Thus, 
Lemma~\ref{lemma-project-B-iff-lte} implies that~$\tau^{\perv}_p$ right projects to~$\cC$.
Therefore, by Lemma~\ref{lemma-projection-tau-cB}\ref{beta*t-geq0}
to show that the projected t-structure~$\gamma^!(\tau^{\perv}_p)$ is connectively induced by~$\tau$,
it remains to identify~${}^{\gamma^!\tau^{\perv}_p} \cC^{\le 0} = {}^{\tau^{\perv}_p} \cD^{\le 0} \cap \cC$
with~${^{\cin{\cC}}}\cC^{\leq 0}$.
For this we note that
\begin{equation*}
{}^{\tau^{\perv}_p} \cD^{\le 0} \cap \cC =
\set{ C \in \cC \sth \gamma^!(\gamma(C)) \in {^{\cin{\cC}}}\cC^{\leq 0}
\text{ and } \Hom(\gamma(C), \beta(B)) = 0 \text{ for all } B \in {^{\tau\vert_{\cB}}}\cB^{> p} },
\end{equation*}
where the first condition is equivalent to~$C \in {^{\cin{\cC}}}\cC^{\leq 0}$ and the second condition is vacuous.
\end{proof}

When the component $\cB$ is zero, so that~$\cD = \cC$, then
the perverse t-structures~$\tau_p^{\perv}$ are all equal to~$\tau$.
Otherwise, these t-structures are distinct and precisely two of them can be obtained from~$\tau$ by tilting:

\begin{proposition}
If~$\cB$ is nonzero, then the t-structure~$\tau^{\perv}_p$ constructed in Lemma~\textup{\ref{lem:perv-p}}
is a tilt of~$\tau$ if and only if~$p \in \{-1,0\}$.
Moreover, the torsion parts~${\cT^{\perv}_p} \coloneqq {}^\tau\cD^\heartsuit \cap {^{\tau^{\perv}_p}}\cD^{\leq 0}$
of the corresponding torsion pairs in~${}^\tau\cD^\heartsuit$ are
\begin{align}
{\cT^{\perv}_0} &=
\set{ D \in {}^\tau\cD^\heartsuit \sth
\gamma\gamma^!(D) \in {}^\tau\cD^{\le 0} },
\\
\label{cT-1}
{\cT^{\perv}_{-1}} &=
\set{ D \in {}^\tau\cD^\heartsuit \sth
\gamma\gamma^!(D) \in {}^\tau\cD^{\le 0} \textup{ and }
\Hom(D,\beta(B)) = 0  \textup{ for all } B \in {}^{\tau\vert_{\cB}}\cB^\heartsuit},
\end{align}
In particular, we have~$\tau^{\perv}_0 = \cin{\gamma\gamma^!}$.
\end{proposition}

\begin{proof}
Recall from Remark~\ref{remark-characterizing-tilt} that~$\tau^{\perv}_p$ is a tilt of~$\tau$
if and only if there are inclusions
\begin{equation}
\label{eq:tau-tau-perv}
{}^\tau\cD^{\leq -1} \subset
{}^{\tau^{\perv}_p}\cD^{\leq 0} \subset
{}^\tau\cD^{\leq 0}.
\end{equation}
So, we test for which~$p$ each of these inclusions hold. 

To test the second inclusion, let~$D \in {}^{\tau^{\perv}_p}\cD^{\le 0}$ and consider the decomposition triangle
\begin{equation*}
\gamma\gamma^!(D) \to D \to \beta\beta^*(D).
\end{equation*}
Since~$\gamma^!(D) \in {^{\cin{\cC}}}\cC^{\leq 0}$ by the definition of~$\tau^{\perv}_p$,
we have~$\gamma\gamma^!(D) \in {}^\tau\cD^{\le 0}$ by Lemma~\ref{lemma-tauB-basic-properties}\ref{it:induced-gamma};
therefore, $D \in {}^\tau\cD^{\le 0}$ if and only if~$\beta\beta^*(D) \in {}^\tau\cD^{\le 0}$.
Since~$\beta$ is t-exact by Lemma~\ref{lemma-beta-exact}\ref{beta-t-exact}, 
this holds if and only if~$\beta^*(D) \in {}^{\tau\vert_\cB}\cB^{\le 0}$,
which is equivalent to~$\Hom(\beta^*(D), B) = 0$ for any~$B \in {}^{\tau\vert_\cB}\cB^{\ge 1}$.
Using the adjunction between~$\beta^*$ and~$\beta$ and the definition of~$\tau^{\perv}_p$,
we finally see that the second inclusion in~\eqref{eq:tau-tau-perv} holds if and only if~$p \le 0$.

To test the first inclusion in~\eqref{eq:tau-tau-perv}, let~$D \in {}^\tau\cD^{\le -1}$.
Then the first condition in the definition of~$\tau^{\perv}_p$ holds for all~$p$
because~$\gamma^!$ has right t-amplitude~$\le 1$
by Lemma~\ref{lemma-restrict-t-structure-vs-projection-amplitude}
and Lemma~\ref{lemma-tauB-basic-properties}\ref{it:induced-amplitude}.
On the other hand, the second condition in the definition of~$\tau^{\perv}_p$ holds for all~$D \in {}^\tau\cD^{\le -1}$
if and only if~$\beta({}^{\tau\vert_\cB}\cB^{\ge p + 1}) \subset {}^\tau\cD^{\ge 0}$. 
Since~$\beta$ is t-exact, we conclude that the first inclusion in~\eqref{eq:tau-tau-perv} holds
if and only if~$p \ge -1$.

It remains to describe~${\cT^{\perv}_p} = {}^\tau\cD^\heartsuit \cap {}^{\tau^{\perv}_p}\cD^{\le 0}$ for~$p \in \{-1,0\}$.
In either case, the condition~$\gamma^!(D) \in {}^{\cin{\cC}}\cC^{\le 0}$ in the definition of~$\tau^{\perv}_p$
is equivalent to~$\gamma\gamma^!(D) \in {}^\tau\cD^{\le 0}$
by the definition of~${}^{\cin{\cC}}\cC^{\le 0}$,
and it remains to test the condition~$\Hom(D, \beta(B)) = 0$ for all~$B \in {^{\tau\vert_{\cB}}}\cB^{\ge p + 1}$.

If~$p = 0$ this condition is vacuous because~$D \in {}^\tau\cD^{\le 0}$
and~$\beta({^{\tau\vert_{\cB}}}\cB^{\ge 1}) \subset {}^\tau\cD^{\ge 1}$ by t-exactness of~$\beta$.
This completes the proof of the formula for~${\cT^{\perv}_0}$,
and since this formula coincides with~\eqref{eq:ct-minus} for~$\Phi = \gamma\gamma^!$ and~$k = 0$,
we conclude that~$\tau^{\perv}_0 = \cin{\gamma\gamma^!}$.

Finally, if~$p = -1$, for all~$D \in {}^\tau\cD^\heartsuit$ and~$B \in {^{\tau\vert_{\cB}}}\cB^{\ge 0}$ we have
\begin{equation*}
\Hom(D, \beta(B)) \cong
\Hom(D, \tau^{\le 0}(\beta(B))) \cong
\Hom(D, \beta({(\tau\vert_\cB)^{\le 0}}(B))),
\end{equation*}
and it remains to note that~$(\tau\vert_\cB)^{\le 0}(B) \cong {}^{\tau\vert_\cB}\cH^0(B) \in {^{\tau\vert_{\cB}}}\cB^\heartsuit$.
\end{proof}

\begin{remark}
Motivated by the above proposition, the definition of $\tau_{-1}^{\perv}$ can be extended
to the situation where $\cB \subset \cD$ is not assumed admissible (as it was in Lemma~\ref{lem:perv-p}). 
Namely, arguing as in Proposition~\ref{proposition-tauPhib}\ref{it:ct-minus}, one can check directly that in the situation of Theorem~\textup{\ref{main-theorem-precise}}\ref{main-theorem-induce-C-precise},
the subcategory ${\cT^{\perv}_{-1}} \subset {^\tau}\cD^{\heartsuit}$ defined by~\eqref{cT-1}
is closed under extensions and quotients, and hence (as~$\tau$ is noetherian) 
extends to a torsion pair $({\cT^{\perv}_{-1}}, {\cF^{\perv}_{-1}})$. 
Then we may define $\tau_{-1}^{\perv}$ as the tilt of $\tau$ with respect to this torsion pair.
\end{remark}

%%%%%%%%%%%%%%%%%%%%%%%%%%%%%%%%%%%%%%%%%%%%%%%%%%%%%%%

\section{Applications}
\label{section-applications}

In this section, we work out the applications of our results, 
including those described in~\S\ref{subsection-intro-applications}. 

\subsection{Birational contractions}

We start with a few birational examples.

\begin{example}
\label{ex:point-blowup}
Let~$f \colon X \to Y$ be the blowup of a smooth variety~$Y$ in a smooth codimension~$2$ subvariety~$Z \subset Y$,
with exceptional divisor~$E \subset X$.
Let~$p \colon E \to Z$ and~$i \colon E \to X$ be the natural morphisms.
Then the functors
\begin{align*}
\cB &\coloneqq \Db(Z) \to \Db(X), \qquad 
F \mapsto i_*p^*F \otimes \cO_X({E}), 
\\
\cC &\coloneqq \Db(Y) \to \Db(X), \qquad
G \mapsto f^*G,
\end{align*}
are fully faithful and induce a semiorthogonal decomposition~$\Db(X) = \langle \cB, \cC \rangle$.
It is easy to see that the embedding functor of~$\cB = \Db(Z)$
is t-exact with respect to the standard t-structures~$\tau_Z$ and~$\tau_X$,
and hence, by Lemma~\ref{lemma-beta-exact}, the standard t-structure~$\tau_X$ restricts to~$\cB$.
Then by Theorem~\ref{main-theorem}, $\tau_X$ connectively induces a t-structure on~$\cC$.
Furthermore, using the argument of Corollary~\ref{cor:geometric}, 
we can apply Lemma~\ref{lem:tauC-recognition} to deduce that the induced t-structure on~$\cC = \Db(Y)$
coincides with the standard t-structure~$\tau_Y$.

Note that in this example, the subcategory~$\cC = f^*\Db(Y) \subset \Db(X)$ is also left admissible,
with the corresponding semiorthogonal decomposition
\begin{equation*}
\Db(X) = \langle \cC, \cB \otimes \cO_X(-E) \rangle.
\end{equation*}
Moreover, the embedding functor~$i_* \circ p^*$ of the second component is t-exact up to shift
with respect to the dual standard t-structures of~$\Db(Z)$ and~$\Db(X)$ (see Example~\ref{example-dual-standard-t-structure}); indeed,
\begin{align*}
\cRHom(i_*p^*(F), \omega^\bullet_X) &\cong
i_*\cRHom(p^*(F), \omega^\bullet_E) \\ &\cong
i_*(\cRHom(p^*(F), p^*\omega^\bullet_Z) \otimes \omega^\bullet_{E/Z}) \cong
i_*(p^*\cRHom(F, \omega^\bullet_Z) \otimes \omega^\bullet_{E/Z}),
\end{align*}
and~$\omega^\bullet_{E/Z}$ is a line bundle shifted by~1.
Therefore, by Lemma~\ref{lemma-beta-exact} and Theorem~\ref{main-theorem}
the dual standard t-structure~$\tau_X^\vee$ coconnectively induces a t-structure on~$\cC$.
Explicitly, the coconnective part of the induced t-structure is given by
\begin{align*}
{}^{\tau_\cC^+}\cC^{\ge 0}
&= \set{ G \in \cC \sth f^*G \in {}^{\tau_X^\vee}\Db(X)^{\ge 0} } \\
&= \set{ G \in \cC \sth \cRHom(f^*G, \omega^\bullet_X) \in {}^{\tau_X}\Db(X)^{\le 0} } \\
&= \set{ G \in \cC \sth f^*\cRHom(G, \omega^\bullet_Y) \otimes \omega^\bullet_{X/Y} \in {}^{\tau_X}\Db(X)^{\le 0} }.
\end{align*}
Since~$\omega^\bullet_{X/Y} \cong \cO_X(E)$ is a line bundle,
the endofunctor~$- \otimes \omega^\bullet_{X/Y}$ is a t-exact autoequivalence of~$\Db(X)$ with respect to the standard t-structure.
It follows that~${}^{\tau_\cC^+}\cC^{\ge 0}$ is obtained from the connectively induced
(hence standard) t-structure on~$\cC = \Db(Y)$ by the Grothendieck duality functor~$\cRHom(-, \omega^\bullet_Y)$.
Thus, the coconnectively induced t-structure is the dual standard t-structure of~$\Db(Y)$. 
\end{example}

\begin{example}
\label{ex:node-blowup}
Let~$f \colon X \to Y$ be the blowup in an ordinary double point of a surface with exceptional curve~$C \subset X$.
Then the functor~$f_* \colon \Db(X) \to \Db(Y)$ is a Verdier localization
with the kernel generated by the sheaf~$\cO_C(-1) \in \Db(X)$ (see~\cite[Theorem~5.8]{KS24}).
However, the object~$\cO_C(-1)$ is not exceptional (but it is spherical),
nor does it give a semiorthogonal decomposition for~$\Db(X)$;
instead, as in Example~\ref{example-bridgeland} there is a semiorthogonal decomposition of the bounded above categories
\begin{equation*}
\Dm(X) = \langle \ker(f_*), f^*\Dm(Y) \rangle,
\end{equation*}
and the standard t-structure $\tau_X$ of~$\Dm(X)$ restricts to a t-structure on~$\ker(f_*)$.
Applying Corollary~\ref{cor:geometric}, we see that $\tau_X$ connectively induces the standard t-structure on~$\Dm(Y)$
(which in turn restricts to the standard t-structure on~$\Db(Y)$).
\end{example}

\begin{example}
\label{ex:flopping-contraction}
Let~$f \colon X \to Y$ be a flopping contraction of threefolds.
As in Example~\ref{example-bridgeland}, there is a semiorthogonal decomposition of the bounded above categories
\begin{equation*}
\Dm(X) = \langle \ker(f_*), f^*\Dm(Y) \rangle,
\end{equation*}
and the standard t-structure $\tau_X$ of~$\Dm(X)$ restricts to a t-structure on~$\ker(f_*)$.
Moreover, by Corollary~\ref{cor:geometric}, $\tau_X$ connectively induces the standard t-structure on~$\Dm(Y)$ (which in turn restricts to the standard t-structure on $\Db(Y)$).
\end{example}

As the above examples show, the connectively induced t-structure is often geometrically meaningful. 

\subsection{Exceptional objects} 
Let $\kk$ be a field, let~$\cD$ be a $\kk$-linear triangulated category with a t-structure~$\tau$,
and let~$E \in \cD$ be an exceptional object contained in the heart~${}^\tau\cD^\heartsuit$. 
Let
\begin{equation}
\label{D-E-B-C}
\cD = \langle \cB, E \rangle \quad \text{and} \quad 
\cD = \langle E, \cC \rangle 
\end{equation} 
be the corresponding semiorthogonal decompositions (Lemma~\ref{lemma-exceptional-sequence}).
The length one exceptional sequence~\mbox{$E_1 = E$}
satisfies the hypotheses of Theorem~\ref{theorem-induce-via-exceptional-sequence},
so if~$\tau$ is noetherian, then it induces a t-structure~$\cin{\cC}$
on the semiorthogonal component~$\cC$, which is bounded if $\tau$ is bounded.
Similarly, by Theorem~\ref{theorem-induce-via-simple-objects}\ref{induce-via-simple-C},
if~$\tau$ is artinian, then it induces a t-structure~$\ccin{\cB}$ on~$\cB$, which is bounded if~$\tau$ is bounded.
Below, we focus on the component $\cC$ because in our examples the t-structure $\tau$ on $\cD$ will be noetherian.

\subsubsection{Fano varieties}

Let~$X$ be a Fano variety for which Kodaira vanishing holds (automatic, for instance, if the base field is characteristic~$0$).
If we write~$\omega_X = \cO_X(-i)$ where~$\cO_X(1)$ is an ample line bundle and~$i \geq 1$,
then the sequence~$\cO_X, \cO_X(1), \dots, \cO_X(i-1)$ is exceptional.
Thus, there is a semiorthogonal decomposition 
\begin{equation}
\label{definition-RX}
\Db(X) = \langle \cO_X, \cO_X(1), \dots, \cO_X(i-1), {\cR_X} \rangle,
\end{equation}
where~$\cR_X$ is called the \emphsf{residual category},
which has been the subject of many recent works. 

For instance, let  
\begin{equation*}
X \subset \bP(w_0, w_1, \dots,  w_n) 
\end{equation*} 
be a smooth hypersurface of degree~$d < w \coloneqq \sum w_k$, where the right-hand side
is a weighted projective space over a field~$\kk$ considered as a stack.
Then we may take~$i = w - d$ above, and the residual category~$\cR_X$ is fractional Calabi--Yau \cite[Corollary~4.2]{K19}.
For example, when~$X$ has index~$1$, i.e. when $d = w- 1$,
then the more precise statement is that~$\rS_X^d \simeq  [{(n+1)d - 2w}]$.

Note that for~$i \ge 2$ the forward $\Hom$-vanishing condition of Theorem~\ref{theorem-induce-via-exceptional-sequence}
fails for~\eqref{definition-RX}; for that reason below we only consider the case where~$i = 1$.
Applying Theorem~\ref{theorem-induce-via-exceptional-sequence} in this case gives the following result.

\begin{theorem}
\label{thm:fcy}
Let~$X \subset \bP(w_0,w_1,\dots,w_n)$ be a Fano hypersurface of degree~$\sum w_k - 1$, and hence of index~$1$.
Then the standard t-structure on~$\Db(X)$  connectively induces a bounded t-structure
on the fractional Calabi--Yau category~$\cR_X$
defined by the semiorthogonal decomposition~$\Db(X) = \langle \cO_X, {\cR_X} \rangle$.
\end{theorem}

\begin{remark}
By browsing through~\cite[Section~4.3]{K19} and~\cite[Section~5.1]{serre-functors-soc}
one can find many other examples of Fano varieties of index~$1$ 
whose residual category~$\cR_X$ is (almost) fractional Calabi--Yau.
In all these cases, Theorem~\ref{theorem-induce-via-exceptional-sequence} applies to produce a bounded t-structure on~$\cR_X$.
\end{remark}

\subsubsection{Categorical resolution of a nodal cubic curve}

Consider the quiver with relations
\begin{equation*}
\label{Q}
Q \coloneqq \left(
\begin{tikzcd}
\mathop{\bullet}\limits \arrow[r, "a_2"', bend right] \arrow[r, "a_1", bend left] &
\mathop{\bullet}\limits \arrow[r, "b_1", bend left] \arrow[r, "b_2"', bend right] &
\mathop{\bullet}\limits
\end{tikzcd}\Bigg{\vert}\quad b_2a_1 = b_1a_2 =0
\right).
\end{equation*}
Let~$\kk$ be a field, let~$\Rep(Q)$ denote the category of finite-dimensional representations of~$Q$ over~$\kk$,
and let~$\Db(Q) \coloneqq \Db(\Rep(Q))$ be the bounded derived category of~$\Rep(Q)$.
Then
\begin{equation*}
E_+ \coloneqq
\left(
\begin{tikzcd}
\kk \arrow[r, "1", bend left] \arrow[r, "0"', bend right] &
\kk \arrow[r, "1", bend left] \arrow[r, "0"', bend right] &
\kk
\end{tikzcd} \right)
\qquad\text{and}\qquad
E_- \coloneqq
\left(
\begin{tikzcd}
\kk \arrow[r, "0", bend left] \arrow[r, "1"', bend right] &
\kk \arrow[r, "0", bend left] \arrow[r, "1"', bend right] &
\kk
\end{tikzcd} \right)
\end{equation*}
are exceptional objects of~$\Db(Q)$, which give rise to semiorthogonal decompositions
\begin{equation}
\label{eq:bondal}
    \Db(Q) = \langle E_+, \cC_+ \rangle = \langle E_-, \cC_- \rangle.
\end{equation}
Geometrically, the categories~$\cC_{\pm}$ can also be described
as crepant categorical resolutions of a nodal cubic curve (see~\cite[\S3.5]{K16}), related by a flop equivalence.
Theorem~\ref{theorem-induce-via-exceptional-sequence} shows:

\begin{theorem}
\label{theorem-bondal-soc}
The standard t-structure on~$\Db(Q)$ connectively induces a bounded
t-structure on the semiorthogonal components~$\cC_\pm \subset \Db(Q)$ defined by~\eqref{eq:bondal}.
\end{theorem}

\begin{remark}
Note that~$\{E_+,E_-\}$ is a $\Hom$-orthogonal pair of $\Hom$-simple objects in~$\Dm(Q)$.
Therefore, by Proposition~\ref{prop:simple-heart} the standard t-structure of~$\Dm(Q)$
restricts to the subcategory~\mbox{$\cB \subset \Dm(Q)$} consisting of objects~$B \in \Dm(Q)$
such that every cohomology~$\cH^n(B)$ (with respect to the standard t-structure)
belongs to the subcategory~$\langle E_+, E_- \rangle_{\mathrm{ext}} \subset \Rep(Q)$.
Moreover, it follows from~\cite[\S3.5]{K16} that there is a semiorthogonal decomposition
\begin{equation*}
\Dm(Q) = \langle \cB, \Dm(C) \rangle,
\end{equation*}
where~$C$ is a nodal cubic curve.
Hence, the standard t-structure of~$\Dm(Q)$ connectively induces a t-structure on~$\Dm(C)$, 
and using the recognition principle (Lemma~\ref{lem:tauC-recognition})
it is not hard to see that the induced t-structure is the standard t-structure of~$\Dm(C)$,
which, as usual, restricts to the standard t-structure on~$\Db(C)$.
\end{remark}

The original interest in decomposition~\eqref{eq:bondal} and category~$\cC = \cC_+$
is that it gives a counterexample to the Jordan--H\"{o}lder property for semiorthogonal decompositions:
the category~$\Db(Q)$ has a full exceptional collection of length~$3$,
whereas the category~$\cC$ does not admit any exceptional objects  (see~\cite{kuznetsov-jordan-holder}).

More recently, Haiden and Wu~\cite{haiden-wu} showed that the category~$\cC$
is pathological from another point of view: it does not admit a stability condition.
Motivated by this, they asked whether the category~$\cC$ admits a bounded t-structure.
Theorem~\ref{theorem-bondal-soc} answers their question.

\subsubsection{Brill--Noether modifications}

Let~$C$ be a smooth proper curve over a field~$\kk$.
In~\cite{kuznetsov-alexeev}, the \emphsf{augmentation}~$\Db(\cO,C)$ of~$C$ was defined
as the triangulated category obtained by gluing~$\Db(\kk)$ to~$\Db(C)$.
By construction, there is a semiorthogonal decomposition
\begin{equation}
\label{DOC-sod} 
\Db(\cO, C) = \langle {E_0}, \Db(C) \rangle
\end{equation} 
where~$E_0$ is an exceptional object satisfying
\begin{equation*}
\RHom_{\Db(\cO,C)}(E_0,F) = \RHom_{\Db(C)}(\cO_C, F)
\quad\text{for any~$F \in \Db(C)$.}
\end{equation*}
Therefore, given~$V \in \Db(\kk)$, $F \in \Db(C)$, and~$\phi \in \Hom(V \otimes \cO_C, F)$, we obtain an object
\begin{equation}
\label{VFphi}
(V, F, \phi) \coloneqq \cone(V {{} \otimes E_0} \to F)[-1] \in \Db(\cO,C),
\end{equation}
and all objects of $\Db(\cO, C)$ can be described in this way. 

The augmentation~$\Db(\cO,C)$ carries a standard bounded t-structure~$\tau$ with the heart
\begin{equation*}
{}^\tau \Db(\cO,C)^\heartsuit \coloneqq \set{ (V,F,\phi) \sth V \in \Vect_{\kk}^{\fd}~\text{and}~F \in \Coh(C) },
\end{equation*}
the category of \emphsf{ generalized coherent systems}.
It is easy to see that~$\tau$ is noetherian.

Therefore, if~${E} = (V, F, \phi) \in {}^\tau \Db(\cO,C)^\heartsuit$ is any exceptional object,
Theorem~\ref{theorem-induce-via-exceptional-sequence} shows that~$\tau$ induces a t-structure
on the category~$\cC$ defined by the semiorthogonal decomposition
\begin{equation*}
\Db(\cO, C) = \langle E, \cC \rangle.
\end{equation*} 
When~$E = E_0$ this decomposition coincides with~\eqref{DOC-sod}, so~$\cC = \Db(C)$,
and using the recognition principle of Lemma~\ref{lem:tauC-recognition}
it is easy to see that the induced t-structure on~$\cC$ is the standard one.
However, there are other much more interesting choices for~$E$.

Following~\cite{kuznetsov-alexeev}, for any~$F \in \Coh(C)$ we can consider its \emphsf{augmentation}
\begin{equation*}
\fa(F) = (\rH^0(C, F), F, \ev_F) \in {}^\tau\Db(\cO, C)^\heartsuit
\end{equation*}
where~$\ev_F \colon \rH^0(C, F) \otimes \cO_C \to F$ is the evaluation morphism.
By~\cite[Proposition~3.12]{kuznetsov-alexeev}, if~$L$ is a Brill--Noether--Petri (BNP) extremal line bundle on~$C$,
meaning that the Petri map
\begin{equation*}
\rH^0(C, L) \otimes \rH^0(C, L^{\vee} \otimes \omega_C) \to \rH^0(C, \omega_C)
\end{equation*}
is an isomorphism, then the object~$\fa(L)$ is exceptional.
In this case, the semiorthogonal component~$\mathscr{BN}_L(C) \subset \Db(C)$ defined by the semiorthogonal decomposition
\begin{equation*}
\Db(\cO, C) = \langle \fa(L), \BN_L(C) \rangle 
\end{equation*}
is called the \emphsf{BN-modification} of~$C$ with respect to~$L$.

Since~$\fa(L) \in {}^\tau \Db(\cO,C)^\heartsuit$, summarizing the above discussion, we have the following result.

\begin{theorem}
\label{theorem-BN-modification}
Let~$L$ be a BNP extremal line bundle on a smooth proper curve~$C$.
Then the standard t-structure on the augmentation~$\Db(\cO, C)$
connectively induces a bounded t-structure on the BN-modification~$\BN_L(C)$ of~$C$ with respect to~$L$.
\end{theorem} 

\begin{example}
\label{example-cubic-threefold}
If~$X \subset \bP^4$ is a cubic threefold with a node,
then the residual category~$\cR_X$ is defined by the semiorthogonal decomposition
\begin{equation*}
\Db(X) = \langle \cO_X, \cO_X(1), \cR_X \rangle.
\end{equation*} 
The category~$\cR_X$ is not smooth and proper, but it admits two strongly crepant categorical resolutions~$\widetilde{\cR}_{X,\pm}$,
defined as appropriate semiorthogonal components of the blowup of~$X$ in the node;
moreover, there is an equivalence~$\widetilde{\cR}_{X,\pm} \simeq \BN_{L_\pm}(C)$,
where~$C$ is the associated curve of genus~$4$
and~$L_{\pm}$ are the trigonal line bundles (see~\cite[Proposition~A.1]{kuznetsov-alexeev}).
Thus, we obtain a bounded t-structure on~$\widetilde{\cR}_{X,\pm}$.

Moreover, the bounded above version~$\cR_X^- \subset \Dm(X)$ of the category~$\cR_X$
embeds fully and faithfully into the bounded above version~$\Dm(\cO,C)$ of the augmentation of~$C$.
It is easy to see that~$\{\fa(L_+), \fa(L_-)\}$ is a $\Hom$-orthogonal pair of $\Hom$-simple objects, hence 
by Proposition~\ref{prop:simple-heart} the standard t-structure of~$\Dm(\cO,C)$
restricts to the subcategory~$\cB \subset \Dm(\cO,C)$ consisting of objects~$B \in \Dm(\cO,C)$
such that every cohomology~$\cH^n(B)$ (with respect to the standard t-structure)
belongs to the subcategory~$\langle \fa(L_+), \fa(L_-) \rangle_{\mathrm{ext}} \subset {}^\tau\Db(\cO, C)^\heartsuit$,
and it is not hard to show that there is a semiorthogonal decomposition
\begin{equation*}
\Dm(\cO,C) = \langle \cB, \cR_X^- \rangle.
\end{equation*}
Therefore, the standard t-structure of~$\Dm(\cO,C)$ connectively induces a t-structure on~$\cR_X^-$,
which restricts to a bounded t-structure on~$\cR_X$.
\end{example}

It is worth noting that the category~$\widetilde{\cR}_{X,\pm}$ above is a priori given
as the semiorthogonal complement of an exceptional sequence
to which Theorem~\ref{theorem-induce-via-exceptional-sequence} does not directly apply,
but we are nonetheless able to apply the theorem by embedding~$\widetilde{\cR}_{X,\pm}$
as a semiorthogonal component in a different category (namely~$\Db(\cO, C)$).
In this way, Theorem~\ref{theorem-induce-via-exceptional-sequence} may apply to many more examples than it would appear on the surface.

\subsection{Enriques surfaces} 

Yet another situation where Theorem~\ref{theorem-induce-via-exceptional-sequence} readily applies
is when there is a completely orthogonal exceptional collection.
To give an interesting such example, let~$X$ be an Enriques surface
defined over an algebraically closed field~$\kk$ of characteristic not equal to~$2$.
Then by~\cite[Proposition~3.5]{LNSZ} there exists an exceptional sequence of line bundles~$L_1, \dots, L_{10}$,
and we may define the \emphsf{residual category} of~$X$ by the semiorthogonal decomposition
\begin{equation}
\label{RX-enriques}
\Db(X) = \langle L_1, \dots, L_{10},  \cR_X  \rangle. 
\end{equation} 
When~$X$ is generic, the exceptional sequence turns out to be completely orthogonal
(and was originally constructed in~\cite{zube}).
For arbitrary~$X$, complete orthogonality fails; in fact, even the Hom-vanishing condition~$\Hom(L_i,L_j) = 0$ fails,
so we cannot directly apply Theorem~\ref{theorem-induce-via-exceptional-sequence}.
Still, for arbitrary~$X$ we can construct a t-structure on~$\cR_X$ using Theorem~\ref{theorem-induce-via-simple-objects}.

\begin{theorem}
\label{thm:enriques}
Let~$X$ be an Enriques surface defined over an algebraically closed field~$\kk$ of characteristic not equal to~$2$.
Then the standard t-structure on~$\Db(X)$ induces a bounded t-structure on the residual category~$\cR_X$.
\end{theorem}

\begin{proof}
It is enough to show that the subcategory~$\cB \coloneqq \langle L_1, \dots, L_{10} \rangle \subset \Db(X)$
satisfies the hypotheses of Theorem~\ref{theorem-induce-via-simple-objects}\ref{induce-via-simple-B}.
For this, recall that the exceptional sequence from~\cite[Proposition~3.5]{LNSZ}
consists of subcollections~$L_{i, 0},\dots, L_{i,n_i-1}$ for~$1 \leq i \leq c$
such that~$L_{1,0},\dots,L_{c,0}$ is a completely orthogonal exceptional collection of line bundles and
\begin{equation*}
L_{i,j} \cong L_{i,0}(R_{i,1} + \dots + R_{i,j}),
\end{equation*}
where~$R_{i,1},R_{i,2},\dots,R_{i,n_i-1}$ is a chain of rational $(-2)$-curves.
Moreover, different chains do not intersect, hence 
different subcollections are completely orthogonal, and~$\sum n_i = 10$.

Now for~$1 \leq i \leq c$, consider the following collection of objects
\begin{equation*}
B_{i,0} \coloneqq L_{i,0}
\quad \text{and} \quad
B_{i,j} \coloneqq \cone(L_{i,j-1} \to L_{i,j}) \cong L_{i,0}(R_{i,1} + \dots + R_{i,j}) { \otimes \cO_{R_{i,j}}} \text{ for $j \geq 1$.}
\end{equation*}
Then~$B_{i,0}$ are line bundles on~$X$ and~$B_{i,j}$ are line bundles on rational curves in~$X$;
hence these objects are $\Hom$-simple, i.e. 
satisfy the first condition in the hypothesis~\eqref{Bi-simple-orthogonal} of Theorem~\ref{theorem-induce-via-simple-objects}\ref{induce-via-simple-B}.
Now let us check $\Hom$-orthogonality, i.e. the second condition in~\eqref{Bi-simple-orthogonal}.

First, the objects~$B_{i,j}$ and~$B_{i,k}$ with~$j,k \ge 1$ and~$j \ne k$
are line bundles supported on distinct curves in~$X$; hence~$\Hom(B_{i,j},B_{i,k}) = 0$.
Similarly, $\Hom(B_{i,j},B_{i,0}) = 0$ for any~$j \ge 1$ because~$B_{i,j}$ is a torsion sheaf while~$B_{i,0}$ is torsion free.
Next, for any~$j \ge 1$ we have
\begin{align*}
\Hom(B_{i,0}, B_{i,j}) &=
\Hom(L_{i,0}, L_{i,0}(R_{i,1} + \dots + R_{i,j}) {\otimes \cO_{R_{i,j}}}) \\ &=
\rH^0(R_{i,j}, \cO_X(R_{i,1} + \dots + R_{i,j})\vert_{R_{i,j}}) =
\rH^0(R_{i,j}, \cO_{R_{i,j}}({-\delta_{j,1} - 1})) = 0 
\end{align*}
because the curves form a chain and~$R_{i,j}$ is a $(-2)$-curve.
Finally, $B_{i,j} \in \langle L_{i,0}, \dots, L_{i,n_i-1} \rangle$ by construction; 
hence~$B_{i,j}$ and~$B_{i',j'}$ are orthogonal if~$i \ne i'$.

It remains to note that by definition~$L_{i,0} = B_{i,0}$ and there are distinguished triangles
\begin{equation*}
L_{i,j-1} \to L_{i,j} \to B_{i,j}, 
\end{equation*}
so it follows by induction on~$j$ that all~$L_{i,j}$ are contained in the triangulated subcategory of~$\Db(X)$ generated by the~$B_{i,j}$.
Therefore, this category coincides with~$\cB$.
\end{proof}

\begin{remark}
For a generic Enriques surface, the situation is analogous to 
the blowup of a smooth surface in $10$ distinct points, where the exceptional curves are disjoint. 
For special Enriques surfaces, the situation is instead analogous
to an iterated blowup of a smooth surface,
where the exceptional curves form a few chains.
\end{remark}

\subsection{(Quasi)phantom categories}
\label{ss:phantoms}
Recall that a semiorthogonal component~$\cP \subset \Db(X)$ of the derived category of a smooth projective variety is a \emphsf{quasiphantom} if its Grothendieck group~$\rK_0(\cP)$ is finite and its Hochschild homology~$\HH_\bullet(\cP)$ vanishes,
and a quasiphantom is a \emphsf{phantom} if~$\rK_0(\cP) = 0$.
We will see that our construction of bounded t-structures applies to the phantom categories
of B\"{o}hning--Graf von Bothmer--Katzarkov--Sosna~\cite{phantoms-bohning}, Krah~\cite{krah},
and others~\cite{KKLLMMPRV},
as well as to the quasiphantom categories of Alexeev--Orlov~\cite{AO}, Galkin--Schinder~\cite{GS},
and B\"{o}hning--Graf von Bothmer--Sosna~\cite{BBS}.

All these quasiphantom categories~$\cP$ are defined by a semiorthogonal decomposition of the form 
\begin{equation}
\label{eq:phantom}
    \Db(X) = \langle L_1, L_2, \dots, L_d, {\cP} \rangle,
\end{equation}
where~$X$ and~$L_i$ are one of the following and we work over the complex numbers: 
\begin{enumerate}[label={\textup{(\alph*)}}]
\item
\label{it:barlow}
$X$ is a generic determinantal Barlow surface,
$L_i$ are described in~\cite[Proposition~4.22]{phantoms-bohning}, $d = 11$;
\item
\label{it:rational}
$X$ is the blowup of~$\bP^2$ in~$10$ points in general position,
$L_i$ are described in~\cite[Theorem~1.1]{krah}, $d = 13$; 
\item
\label{it:more-rational}
$X$ is the blowup of~$\bF_2$ or~$\bP^2$ in~$9$, $10$, or~$11$ points in general position,
$L_i$ are described in~\cite[Theorems~1.3, 3.7, and~1.1]{KKLLMMPRV}, and $d = 13$, $13$, or~$14$, respectively;
\item
\label{it:burniat}
$X$ is a primary Burniat surface,
$L_i$ are described in~\cite[Theorem~4]{AO}, $d = 6$;
\item
\label{it:beauville}
$X$ is the Beauville surface,
$L_i$ are described in~\cite[Theorem~3.5]{GS}, $d = 4$;
\item
\label{it:godeaux}
$X$ is the classical Godeaux surface,
$L_i$ are obtained from~\cite[Theorem~8.2]{BBS} by a small modification
explained in the proof of Theorem~\ref{theorem-phantoms} below, and~$d = 11$.
\end{enumerate}
The categories~$\cP$ are phantoms or quasiphantoms by~\cite[Theorem~1.1]{phantoms-bohning}, \cite[Corollary~5.4]{krah}, 
\cite[Theorems~1.3, 3.7, and~1.1]{KKLLMMPRV}, 
\cite[Theorem~5]{AO}, \cite[Proposition~3.10]{GS}, and~\cite[Theorem~1.1]{BBS}.

\begin{theorem}
\label{theorem-phantoms}
Let~$\cP$ be one of the phantom or quasiphantom categories~\ref{it:barlow}--\ref{it:godeaux}.
Then the standard t-structure on~$\Db(X)$ connectively induces a bounded t-structure on~$\cP$.
\end{theorem}

\begin{proof}
To apply Theorem~\ref{theorem-induce-via-exceptional-sequence} we only need
to check that~\eqref{eq:hom-vanishing-intro} holds.
This was checked in~\cite[Proposition~7.1]{phantoms-bohning} for case~\ref{it:barlow},
\cite[Proof of Theorem~1.1]{krah} for case~\ref{it:rational}, 
\cite{KKLLMMPRV} for case~\ref{it:more-rational},
and~\cite[Propositions~5.3 and~5.5]{K15} for cases~\ref{it:burniat} and~\ref{it:beauville}.

In the last case~\ref{it:godeaux} the exceptional collection~$L^0_1,\dots,L^0_{11}$ from~\cite[Theorem~8.2]{BBS}
does not satisfy the forward $\Hom$-vanishing condition~\eqref{eq:hom-vanishing-intro}
(in fact, $\Hom(L^0_2,L^0_3) \cong \bC$, see the proof of~\cite[Lemma~5.7]{K15}),
so we replace it by the collection
\begin{equation*}
(L_1,L_2,L_3,L_4,\dots,L_{11}) = (L^0_3, L^0_4, \dots, L^0_{11}, L^0_1(-K_X), L^0_2(-K_X))
\end{equation*}
Then, the proof of~\cite[Proposition~5.8]{K15} verifies~\eqref{eq:hom-vanishing-intro}; 
hence Theorem~\ref{theorem-induce-via-exceptional-sequence} applies
and we conclude that~$\cP$ has a bounded t-structure.
\end{proof}

\begin{remark}\label{remark-phantoms-question}
    It would be interesting to find a conceptual explanation
    for the vanishing of the groups $\Hom(L_i, L_j)$ for $i < j$
    in the exceptional sequences defining the phantoms.
    In particular, the above argument suggests the following question:
    Is it possible that any phantom~$\cP$ can be realized as the orthogonal complement of such an exceptional sequence of coherent sheaves?
    A positive answer would imply the existence of a bounded t-structure on~$\cP$.
\end{remark}

As we mentioned in the introduction,
from the existence of a bounded t-structure, we also deduce the following result, which answers a question of Ben Antieau.

\begin{corollary}
\label{corollary-coconnective-DG-algebra}
Let~$\cP$ be one of the phantom or quasiphantom categories~\ref{it:barlow}--\ref{it:godeaux}.
Then there exists a classical generator~$G \in \cP$,  i.e. an object which thickly generates~$\cP$,
with the property that~$\Ext^{i}(G,G) = 0$ for~$i < 0$.
In particular, $A = \RHom(G,G)$ is a coconnective DG algebra such that~$\Dperf(A)$ is a phantom or quasiphantom category.
\end{corollary}

\begin{proof}
More generally, suppose $\cC \subset \Db(X)$ is a semiorthogonal component of a smooth proper variety
which admits a bounded t-structure~{$\tau_\cC$}.
Then $\cC$ admits a classical generator $G$, given by the projection of a classical generator of $\Db(X)$
(which exists for instance by~\cite{neeman-generators}).
Replacing~$G$ by the direct sum of its cohomology objects with respect to~{$\tau_\cC$},
we may assume that~$G \in {{}^{\tau_\cC}}\cC^{\heartsuit}$ (see~\cite[Lemma~B.1]{serre-functors-soc}),
and thus $\Ext^i(G,G) = 0$ for $i < 0$.
Moreover, since~$\cC$ is enhanced, there is an equivalence~$\cC \simeq \Dperf(A)$
where~$A = \RHom(G,G)$ (see for instance~\cite[Theorem~7.1.2.1]{HA}).
\end{proof}

\begin{remark}
In the case where $\cP$ is Krah's phantom~\ref{it:rational},
Mattoo~\cite{Mattoo} recently independently proved Corollary~\ref{corollary-coconnective-DG-algebra} by a different, more computational argument:
he constructed an explicit generator~$G$ for which the vanishing~$\Ext^i(G,G) = 0$ for~$i < 0$
can be calculated using a spectral sequence. 
Even more recently, Mattoo's method was applied in~\cite{MXY}
to prove Corollary~\ref{corollary-coconnective-DG-algebra} for the phantom on the blowup of~$\bP^2$ in~$11$ points from~\ref{it:more-rational}.
\end{remark}

\begin{remark}
There is a notable example to which our construction of bounded t-structures does not directly apply: 
the phantoms $\cP$ constructed by Gorchinskiy and Orlov \cite{phantoms-orlov} by taking the tensor product
of two appropriate quasiphantoms of distinct types~\ref{it:burniat}, \ref{it:beauville}, or~\ref{it:godeaux}.
However, Corollary~\ref{corollary-coconnective-DG-algebra} still holds for such a phantom~$\cP$,
because the tensor product of coconnective DG algebras over a field is coconnective.
\end{remark} 

As promised in the introduction, we now describe some explicit objects in the heart of an induced bounded t-structure on Krah's phantom. 

\begin{example}
\label{example-mattoo-in-heart}
Let~$X$ be the blowup of~$\bP^{2}$ in~$10$ points in general position.
We consider Krah's phantom~$\cP$ from case~\ref{it:rational} above, defined by
\begin{equation*}
\Db(X) = \langle L_1, \dots, L_{13}, \cP \rangle
\end{equation*}
where the $L_i$ are defined in~\cite[Theorem~1.1]{krah}.
Following~\cite[\S2.1]{Mattoo}, we also consider the dual semiorthogonal decomposition
\begin{equation*}
\Db(X) = \langle \cP', L_{13}^\vee, \dots, L_1^\vee \rangle,
\end{equation*}
where~$\cP'$ is 
antiequivalent to~$\cP$ via the dualization functor; in particular, $\cP'$ is also a phantom category.
By Theorem~\ref{theorem-phantoms} the standard t-structure~$\tau_X$ on~$\Db(X)$
connectively induces a bounded t-structure~$(\tau_X)_{\cP}^-$ on~$\cP$.

Let~$\rho \colon \cP' \to \Db(X)$ be the inclusion
and~$\rho^* \colon \Db(X) \to \cP'$ its left adjoint.
In~\cite[\S5.2]{Mattoo}, Mattoo constructs an interesting class of objects in~$\cP'$ of the form
\begin{equation*}
P' = \rho^* i_* M,
\end{equation*}
where~$i \colon C \to X$ is a carefully chosen curve in~$X$ and~$M$ is a suitable line bundle on~$C$.
In~\cite[Proposition~5.10]{Mattoo}, it is shown that~$\rho(P')$ is concentrated in degrees~$[0,1]$
with respect to the standard t-structure~$\tau_X$ and, moreover, that
\begin{itemize}
\item
${}^{\tau_X}\cH^1(\rho(P'))$ is a sum of copies of~$L_{13}^\vee$, and
\item
${}^{\tau_X}\cH^0(\rho(P'))$ can be written as an extension of a sum of copies of~$L_{13}^\vee$ and~$L_{12}^\vee$ by~$i_*M$.
\end{itemize}
Now, consider the object
\begin{equation*}
P \coloneqq (\rho(P'))^\vee[1] \in \cP.
\end{equation*}
The above description of the cohomology objects of~$P'$ implies that~$P$ is concentrated in degrees~$[-2,0]$
with respect to the standard t-structure~$\tau_X$ and, moreover, that
\begin{equation*}
{}^{\tau_X}\cH^0(P) \cong i_*(M^\vee) \otimes \cO_X(C),
\end{equation*}
while~${}^{\tau_X}\cH^p(P)$ for~$p \in \{-2,-1\}$ are sums of copies of~$L_{12}$ and~$L_{13}$;
in particular, $\tau_X^{\le -1}(P)$ is contained in~$\langle L_1, \dots, L_{13} \rangle$.
Hence, we may apply Lemma~\ref{lemma-object-in-heart-tauB}\ref{it:objects-C}
to conclude that~$P \in {}^{{(\tau_X)^-_{\cP}}}\cP^{\heartsuit}$.
\end{example}

\begin{remark}
In~\cite{Mattoo}, several other interesting examples of objects in $\cP'$ with vanishing negative self-Ext groups are constructed. 
It seems likely that they also correspond to objects contained in the heart of  the induced t-structure~$(\tau_X)^-_{\cP}$, but we have not carried out the necessary computations to verify this.
Similarly, using the techniques of~\cite{Mattoo}, we expect that it is possible to construct interesting objects in the hearts of the induced t-structures on other phantoms from Theorem~\ref{theorem-phantoms}.
\end{remark} 

%%%%%%%%%%%%%%%%%%%%%%%%%%%%%%%%%%%%%%%%%%%%%%%%%%%%%%%

\appendix

\section{Induced t-structures via compactly generated t-structures}
\label{sec:ind}

In this appendix we sketch an alternative construction of the induced t-structure
in Theorem~\ref{main-theorem}\ref{main-theorem-induce-C}.
The proof requires a mild extra hypothesis,
but we believe it is still worth presenting since the ideas involved may be useful in other settings. 

\subsection{Compactly generated t-structures} 
\label{section-cg-t-structure} 
In what follows we assume that~$\cD$ is identified with the subcategory of compact objects
in a compactly generated cocomplete triangulated category~$\hcD$, i.e. $\cD = \hcD^c$. 
Note that this assumption implies that~$\cD$ is idempotent complete. 

There are at least two ways such an identification may be achieved.
First, it could be a priori given: for instance, if~$\cD = \Dperf(X)$ for a quasi-compact quasi-separated scheme~$X$,
then~$\hcD = \Dqc(X)$ has the required property.
Second, if~$\cD$ is enhanced, i.e., given as the homotopy category of a small DG category or stable $\infty$-category,
one can take~$\hcD$ to be the Ind-completion of the enhanced category
(see~\cite[\S5.2.5]{HTT} or~\cite[Appendix~B]{serre-functors-soc} for details).

Now, given a category~$\hcD$ as above, for any t-structure $\tau$ on $\cD$ we define a t-structure~$\htau$ on~$\hcD$ by
\begin{equation}
\label{eq:htau}
\begin{aligned}
{}^{\htau}\hcD^{\le 0} & \coloneqq \big\langle {}^\tau\cD^{\le 0} \big\rangle_{\oplus,\ext}, \\ 
{}^{\htau}\hcD^{\ge 0} & \coloneqq {\set{ D \in \hcD \sth \Hom(D',D) = 0 \text{ \textup{for all} } D' \in {}^{\htau}\hcD^{\le 0}[1]}},
\end{aligned}
\end{equation}
i.e., ${}^{\htau}\hcD^{\le 0}$ is the closure of~${}^\tau\cD^{\le 0}$ in~$\hcD$ with respect to infinite coproducts and extensions
and~${}^{\htau}\hcD^{\ge 0}$ is its shifted orthogonal;
by~\cite[Theorem~A.1]{AJS:tstructures} (see also~\cite[Theorem~2.3.3]{passage-weakly-approximable}
and~\cite[Propositions~1.4.4.11 and~1.4.4.13]{HA}),  this is indeed a t-structure.
Moreover, it follows easily from this construction that the inclusion functor $\cD \to \hcD$ is t-exact with respect to $\tau$ and $\htau$; in other words, $\htau$ restricts to the t-structure $\tau$ on $\cD$.  

For our purposes, the crucial feature of this construction is the following result.

\begin{lemma}
\label{lemma:hphi-exactness}
If~$\tau$ is noetherian, then its heart~${}^\tau\cD^\heartsuit$ is a Serre subcategory of the heart~${}^{\htau}\hcD^\heartsuit$ of $\htau$,
i.e., ${}^\tau\cD^\heartsuit$ is closed in~${}^{\htau}\hcD^\heartsuit$ under taking subobjects, quotient objects, and extensions.
\end{lemma}

\begin{proof}
This is essentially~\cite[Remark~C.6.8.10]{SAG}; we sketch the argument for reader's convenience.

First, let~$D \in {}^\tau\cD^\heartsuit \subset {}^{\htau}\hcD^\heartsuit$,
and let~$D' \subset D$ be a subobject in~${}^{\htau}\hcD^\heartsuit$.
By \cite[Theorem~3.0.1]{passage-weakly-approximable},  
we can write~$D'$ as a filtered colimit of objects~$D_\alpha$ in~${}^\tau\cD^\heartsuit$;
up to replacing the~$D_{\alpha}$ with their image in~$D$, we may assume that the maps~$D_{\alpha} \to D'$ are inclusions.
Then the noetherian property implies that~$D' = D_\alpha$ for some $\alpha$ and therefore~$D' \in {}^\tau\cD^\heartsuit$.
This proves that~${}^\tau\cD^\heartsuit$ is closed under taking subobjects.

Similarly, if~$D \to {D''}$ is an epimorphism in~${}^{\htau}\hcD^\heartsuit$,
its kernel in~${}^{\htau}\hcD^\heartsuit$ is in~${}^{\tau}\cD^\heartsuit$ by the previous argument; 
hence~${D''}$ is in~$\cD$, and as it is also in~${}^{\htau}\hcD^\heartsuit$, it is in~${}^{\tau}\cD^\heartsuit$.

Finally, closedness of~${}^{\tau}\cD^\heartsuit$ under extensions is obvious.
\end{proof}

Next, assume given a fully faithful functor~$\gamma \colon \cC \to \cD$ from a triangulated category~$\cC$
such that~$\cC$, when regarded via~$\gamma$ as a subcategory of~$\cD$, is idempotent complete in~$\cD$.
Then there is a cocomplete category~$\hcC$ such that~$\hcC^c = \cC$
and such that~$\gamma$ extends to a continuous (i.e. commuting with arbitrary coproducts)
fully faithful functor~$\hgamma \colon \hcC \to \hcD$.
Indeed, identifying~$\cC$ with a subcategory of~$\cD$ via~$\gamma$
we can define~$\hcC \subset \hcD$ as the coproduct closure of~$\cC$ in~$\hcD$. 
Then the inclusion~$\cC \subset \hcC^c$ is obvious.
Conversely, if~$C \in \hcC^c$ then since~$C \in \hcC$ we can write~$C = \colim C_\alpha$ where~$C_{\alpha} \in \cC$;
as~$C$ is compact in~$\hcC$, it follows that~$C$ is a direct summand of some~$C_{\alpha}$,
and hence contained in~$\cC$ by idempotent completeness.
This proves that~$\hcC^c = \cC$, and by construction the inclusion~$\hcC \hookrightarrow \hcD$ gives the required extension of~$\gamma$.

Note that the continuity of~$\hgamma$ implies the existence of a right adjoint functor~$\hgamma^! \colon \hcD \to \hcC$,
so since~$\hgamma$ is fully faithful, we obtain a semiorthogonal decomposition
\begin{equation}
\label{eq:sod-hcd}
\hcD = \langle \widehat{\cB}, \hcC \rangle.
\end{equation}
Furthermore, since~$\hgamma(\cC) \subset \cD$ (i.e., $\hgamma$ preserves compactness),
it follows that~$\hgamma^!$ is continuous.
Indeed, for any~$C \in \cC$ and a set of objects~$D_\alpha \in \hcD$ we have
\begin{multline*}
\Hom(C, \hgamma^!(\oplus D_\alpha)) \cong
\Hom(\hgamma(C), \oplus D_\alpha) \cong
\oplus \Hom(\hgamma(C), D_\alpha) \\ \cong
\oplus \Hom(C, \hgamma^!(D_\alpha)) \cong
\Hom(C, \oplus \hgamma^!(D_\alpha))
\end{multline*}
by adjunction, compactness of~$\hgamma(C)$, adjunction again, and compactness of~$C$; since~$\cC$ generates~$\hcC$, it follows that the canonical morphism~$\oplus \hgamma^!(D_\alpha) \to \hgamma^!(\oplus D_\alpha)$
is an isomorphism.
Note also that if~$\gamma$ has a right adjoint~$\gamma^! \colon \cD \to \cC$ then~$\hgamma^!\vert_\cD = \gamma^!$;
in particular, $\hgamma^!(\cD) \subset \cC$.

\subsection{Alternative proof of {Theorem~\ref{main-theorem}\ref{main-theorem-induce-C}}}

We work in the following setting, which is that of Theorem~\ref{main-theorem}\ref{main-theorem-induce-C} together with a mild additional assumption on the existence of an embedding of $\cD$ into a cocomplete triangulated category, as in \S\ref{section-cg-t-structure} above.

\begin{setup}
\label{setup-alternative-proof}
Let $\cD$ be a triangulated category which is equal to the subcategory of compact objects in a compactly generated cocomplete triangulated category $\hcD$. 
Let $\cD = \langle \cB, \cC \rangle$ be a semiorthogonal decomposition. 
Let $\tau$ be a noetherian t-structure on $\cD$ which restricts to a t-structure on $\cB$. 
\end{setup} 

By the discussion in~\S\ref{section-cg-t-structure}, in this situation we obtain a t-structure~$\htau$ on~$\hcD$, a category~$\hcC$,
and a continuous extension~$\hgamma \colon \hcC \to \hcD$ of the inclusion functor~$\gamma \colon \cC \to \cD$
with continuous right adjoint~$\hgamma^! \colon \hcD \to \hcC$.

We consider
\begin{equation}
\label{eq:tau-cc}
\begin{aligned}
\cC^{\le 0} &\coloneqq \set{ C \in \cC \sth \gamma(C) \in {}^\tau\cD^{\le 0} },\\
\cC^{\ge 0} &\coloneqq {\set{ C \in \cC \sth \Hom(C',C) = 0 \text{ \textup{for all} } C' \in \cC^{\leq 0}[1]}}.
\end{aligned}
\end{equation}
We aim to show it is a t-structure.
To prove this we first define a t-structure $\htau_{\cC}$ on the cocomplete category~$\hcC$, by taking
\begin{equation}
\label{eq:htau-cc}
\begin{aligned}
{}^{\htau_\cC}\hcC^{\le 0} &\coloneqq \big\langle \cC^{\le 0} \big\rangle_{\oplus,\ext},
\\
{}^{\htau_\cC}\hcC^{\ge 0} &\coloneqq {\set{ C \in \hcC \sth \Hom(C',C) = 0 \text{ \textup{for all} } C' \in \hcC^{\leq 0}[1]}}.
\end{aligned}
\end{equation}
As before, \cite[Theorem~A.1]{AJS:tstructures} implies that~$\htau_\cC$ is a t-structure.
To prove Theorem~\ref{main-theorem}, we will show that~$\htau_{\cC}$
restricts to a t-structure on~$\cC$ which is given by~\eqref{eq:tau-cc}.

First, observe that the functor~$\hgamma \colon \hcC \to \hcD$ is right t-exact with respect to~$\htau_\cC$ and~$\htau$;
this follows immediately from~\eqref{eq:htau}, \eqref{eq:tau-cc}, \eqref{eq:htau-cc}, and the continuity of~$\hgamma$.
The crucial observation is that the functor~$\hgamma^! \colon \hcD \to \hcC$ has t-amplitude in~$[0,1]$, i.e. the following holds.

\begin{lemma}
\label{lemma-Phi!-amplitude}
We have~$\hgamma^!({}^{\htau}\hcD^{\ge 0}) \subset {}^{\htau_\cC}\hcC^{\ge 0}$
and~$\hgamma^!({}^{\htau}\hcD^{\le 0}) \subset {}^{\htau_\cC}\hcC^{\le 1}$.
\end{lemma}

\begin{proof}
As we mentioned above, the functor~$\hgamma \colon \hcC \to \hcD$ is right t-exact,
so Lemma~\ref{lemma-t-exactness-adjoints} implies its right adjoint~$\hgamma^! \colon \hcD \to \hcC$
is left t-exact and the first inclusion holds.
For the second inclusion, note that by Lemma~\ref{lemma-restrict-t-structure-vs-projection-amplitude}
and the assumption that~$\tau$ restricts to~$\cB$,
we have an inclusion~$\gamma \gamma^!({^\tau}\cD^{\leq 0}) \subset {^\tau}\cD^{\leq 1}$.
By the definitions~\eqref{eq:htau} and~\eqref{eq:htau-cc} of~${}^{\htau}\hcD^{\le 0}$ and~${}^{\htau_\cC}\hcC^{\le 0}$,
this implies the desired inclusion~$\hgamma^!({}^{\htau}\hcD^{\le 0}) \subset {}^{\htau_\cC}\hcC^{\le 1}$. 
\end{proof}

Now we use the above results to prove the following. 

\begin{proposition}
\label{prop:ch0}
If~$C \in \cC \cap {}^{\htau_\cC}\hcC^{\le 0}$ then~$C_0 \coloneqq {}^{\htau_\cC}\cH^0(C)$ is contained in~$\cC$.
\end{proposition}

\begin{proof}
Consider the triangle~${\htau_\cC}^{\le -1}(C) \to C \to C_0$.
Applying the functor~$\hgamma$, considering the cohomology with respect to~$\htau$,
and taking into account that~$\hgamma$ is right t-exact, we obtain
\begin{equation*}
{}^{\htau}\cH^0(\hgamma(C_{-1})) \to {}^{\htau}\cH^{-1}(\hgamma(C)) \to {}^{\htau}\cH^{-1}(\hgamma(C_0)) \to
0 \to {}^{\htau}\cH^0(\hgamma(C)) \to {}^{\htau}\cH^0(\hgamma(C_0)) \to 0,
\end{equation*}
where~$C_{-1} = {}^{\htau_\cC}\cH^{-1}(C)$.
Since~$C \in \cC$, we have~$\hgamma(C) \cong \gamma(C) \in \cD$, and hence
\begin{equation*}
{}^{\htau}\cH^i(\hgamma(C)) \cong {}^{\tau}\cH^i(\gamma(C)) \in {}^\tau\cD^\heartsuit
\end{equation*}
for all $i$. 
Furthermore, since~${}^\tau\cD^\heartsuit$ is closed under taking quotients in ${^{\htau}}\hcD^{\heartsuit}$ by Lemma~\ref{lemma:hphi-exactness},
we obtain
\begin{equation}
\label{eq:ch-phi-e0}
{}^{\htau}\cH^{-1}(\hgamma(C_0)) \in {}^\tau\cD^\heartsuit
\qquad\text{and}\qquad
{}^{\htau}\cH^0(\hgamma(C_0)) \in {}^\tau\cD^\heartsuit.
\end{equation}

Next, consider the spectral sequence (of a similar flavor as~\eqref{eq:ss})
\begin{equation*}
\rE_2^{p,q} = {}^{\htau_\cC}\cH^p\hgamma^!({}^{\htau}\cH^q\hgamma(C_0)) \implies
{}^{\htau_\cC}\cH^{p+q}(\hgamma^!\hgamma(C_0)) =
{}^{\htau_\cC}\cH^{p+q}(C_0).
\end{equation*}
The right-hand side vanishes unless~$p + q = 0$ by the definition of~$C_0$.
On the other hand, by Lemma~\ref{lemma-Phi!-amplitude} the left-hand side vanishes unless~$p \in \{0,1\}$ and~$q \le 0$.
Therefore, the spectral sequence degenerates, which has the following consequences. 

First, it implies that~${}^{\htau_\cC}\cH^1\hgamma^!({}^{\htau}\cH^0\hgamma(C_0)) = 0$; hence by Lemma~\ref{lemma-Phi!-amplitude} we have 
\begin{equation*}
{}^{\htau_\cC}\cH^0\hgamma^!({}^{\htau}\cH^0\hgamma(C_0)) =
\hgamma^!({}^{\htau}\cH^0\hgamma(C_0)) , 
\end{equation*} 
which by~\eqref{eq:ch-phi-e0} is contained in $\gamma^!({}^\tau\cD^\heartsuit) \subset \cC$. 

Similarly, it implies that~${}^{\htau_\cC}\cH^0\hgamma^!({}^{\htau}\cH^{-1}\hgamma(C_0)) = 0$;
hence again by Lemma~\ref{lemma-Phi!-amplitude} we have
\begin{equation*}
{}^{\htau_\cC}\cH^1\hgamma^!({}^{\htau}\cH^{-1}\hgamma(C_0)) =
\hgamma^!({}^{\htau}\cH^{-1}\hgamma(C_0))[1], 
\end{equation*}
which again by \eqref{eq:ch-phi-e0} is contained in $\cC$. 

Finally, the spectral sequence gives an exact sequence
\begin{equation*}
0 \to
{}^{\htau_\cC}\cH^1\hgamma^!({}^{\htau}\cH^{-1}\hgamma(C_0)) \to
C_0 \to
{}^{\htau_\cC}\cH^0\hgamma^!({}^{\htau}\cH^0\hgamma(C_0)) \to
0. 
\end{equation*}
Since we showed the first and last terms are in~$\cC$, we conclude that~$C_0 \in \cC$.
\end{proof}

\begin{proof}[Proof of {Theorem~\textup{\ref{main-theorem}\ref{main-theorem-induce-C}}}
in Setup~\textup{\ref{setup-alternative-proof}}]
First, it follows immediately from Proposition~\ref{prop:ch0} that for any integer~$n$,
if~$C \in \cC \cap {}^{\htau_\cC}\hcC^{\le n}$ then~${}^{\htau_\cC}\cH^n(C) \in \cC$.
Applying induction, it is easy to deduce from this that for such~$C$ and any integer~$m$,
we have~$\htau_\cC^{\ge m}(C) \in \cC$,
and therefore also~$\htau_\cC^{\le m}(C) \in \cC$.

Next let~$C \in \cC$ be arbitrary.
By Lemma~\ref{lemma-Phi!-amplitude} we have~$\hgamma^!\htau^{\ge m+1}\hgamma(C) \in {}^{\htau_\cC}\hcC^{\ge m + 1}$ for any integer~$m$,
and hence from the distinguished triangle
\begin{equation*}
\hgamma^!\htau^{\le m}\hgamma(C) \to C \to \hgamma^!\htau^{\ge m+1}\hgamma(C) 
\end{equation*} 
we find 
\begin{equation}
\label{htauCleqmC}
\htau_\cC^{\le m}(C) \cong
\htau_\cC^{\le m}\hgamma^!\htau^{\le m}\hgamma(C). 
\end{equation}
By Lemma~\ref{lemma-Phi!-amplitude} we have~$\hgamma^!\htau^{\le m}\hgamma(C) \in {^{\htau_{\cC}}}\hcC^{\le m+1}$.
On the other hand, since~$\hgamma(C) = \gamma(C) \in \cD$,
we have~$\htau^{\le m}\hgamma(C) = \tau^{\le m}\gamma(C) \in {}^\tau\cD^{\le m}$,
and hence~$\hgamma^!\htau^{\le m}\hgamma(C) \in \gamma^!({}^\tau\cD^{\le m}) \subset \cC$.
Together, this shows that 
\begin{equation*}
\hgamma^!\htau^{\le m}\hgamma(C)  \in \cC \cap {^{\htau_{\cC}}}\hcC^{\le m+1}. 
\end{equation*}
Thus from~\eqref{htauCleqmC} and the previous paragraph, we conclude that~$\htau_\cC^{\le m}(C) \in \cC$.
This proves that~$\htau_{\cC}$ restricts to a t-structure on~$\cC$,
given by~$\tau_{\cC} = ({}^{\htau_\cC}\hcC^{\le 0} \cap \cC, {}^{\htau_\cC}\hcC^{\geq 0} \cap \cC)$.

We claim that~$\tau_{\cC}$ is the connectively induced t-structure on~$\cC$, i.e. that
\begin{equation*}
{}^{\htau_\cC}\hcC^{\le 0} \cap \cC = \{ C \in \cC \mid \gamma(C) \in {}^\tau\cD^{\le 0} \}. 
\end{equation*}
By the definition of~$\htau_\cC$, the right-hand side is contained in the left.
Conversely, if~$C$ is an object in the left-hand side,
then~$\gamma(C) \in {^{\htau}}\hcD^{\leq 0} \cap \cD$ by the right t-exactness of~$\hgamma$,
but~${^{\htau}}\hcD^{\leq 0} \cap \cD = {^\tau}\cD^{\le 0}$ by the construction of~$\htau$.

It remains to prove that if~$\tau$ is bounded, then~$\tau_{\cC}$ is bounded.
First, note that since~$\tau_{\cC}$ is connectively induced by~$\tau$,
it is bounded above by Lemma~\ref{lemma-tauB-basic-properties}\ref{it:induced-bounded}.
Next note that~$\tau_{\cC}$ is left separated, i.e., $\bigcap_{n \in \bZ} {^{\tau_{\cC}}}\cC^{\leq n} = \set{ 0 }$;
indeed, since~$\tau$ is bounded we have~\mbox{$\bigcap_{n \in \bZ} {^{\tau}}\cD^{\leq n} = \set{ 0}$},
so the claim follows from the right t-exactness of the inclusion functor~$\gamma \colon \cC \to \cD$
with respect to~$\tau_{\cC}$ and~$\tau$.
To prove that~$\tau_{\cC}$ is bounded below, it then suffices to show that for any object~\mbox{$C \in \cC$},
its cohomology objects~${^{\tau_{\cC}}}\cH^n(C)$ vanish for~$n \ll 0$.
As in the proof of Proposition~\ref{prop:ch0}, Lemma~\ref{lemma-Phi!-amplitude} implies that the $\rE_2^{p,q}$ terms in the spectral sequence 
\begin{equation*}
\rE_2^{p,q} = {}^{\tau_\cC}\cH^p\gamma^!({}^{\tau}\cH^q\gamma(C)) \implies
{}^{\tau_\cC}\cH^{p+q}(C)
\end{equation*} 
vanish unless $p \in \set{0,1}$. 
For any integer $n$, this yields a short exact sequence
\begin{equation*}
0 \to
{}^{\tau_\cC}\cH^1\gamma^!({}^{\tau}\cH^{n-1}\gamma(C)) \to
{^{\tau_{\cC}}}\cH^n(C) \to
{}^{\tau_\cC}\cH^0\gamma^!({}^{\tau}\cH^n\gamma(C)) \to
0. 
\end{equation*} 
Since $\tau$ is bounded, it follows that ${^{\tau_{\cC}}}\cH^n(C)$ vanishes for all but finitely many $n$.
\end{proof}

%%%%%%%%%%%%%%%%%%%%%%%%%%%%%%%%%%%%%%%%%%%%%%%%%%%%%%%

\newcommand{\etalchar}[1]{$^{#1}$}
\providecommand{\bysame}{\leavevmode\hbox to3em{\hrulefill}\thinspace}
\providecommand{\MR}{\relax\ifhmode\unskip\space\fi MR }
\providecommand{\MRhref}[2]{%
  \href{http://www.ams.org/mathscinet-getitem?mr=#1}{#2}
}
\providecommand{\href}[2]{#2}

%%%%%%%%%%%%%%%%%%%%%%%%%%%%%%%%%%%%%%%%%%%%%%%%%%%%%%%

\end{document}